\documentclass{article}
\usepackage{arxiv}
\usepackage{amsthm,amsmath,amsfonts,amssymb}

\DeclareMathOperator{\proj}{proj}

\DeclareMathOperator{\RLT}{RLT}
\DeclareMathOperator{\conv}{conv}
\DeclareMathOperator{\conc}{conc}
\DeclareMathOperator{\interior}{int}
\DeclareMathOperator{\graph}{gr}

\DeclareMathOperator{\vertex}{vert}

\DeclareMathOperator{\hypo}{hyp}

\DeclareMathOperator{\R}{\mathbb{R}}
\DeclareMathOperator{\Z}{\mathbb{Z}}

\DeclareMathOperator{\ri}{ri}
\DeclareMathOperator{\PPolytope}{AU}

\newtheorem{theorem}{Theorem}[section]
\newtheorem{corollary}{Corollary}[section]
\newtheorem{lemma}{Lemma}[section]
\newtheorem{proposition}{Proposition}[section]
\newtheorem{remark}{Remark}[section]
\newtheorem{example}{Example}[section]
\newtheorem{definition}{Definition}[section]
\usepackage{mathrsfs}
\usepackage{listofitems}
\usepackage{graphicx}
\usepackage{comment}
\excludecomment{OldVersion}
\usepackage{bbm}
\usepackage{blindtext}
\usepackage{stmaryrd}

\usepackage[inline]{enumitem}
\usepackage{tikz, tikz-3dplot}
\usepackage{tikz-qtree}
\usetikzlibrary{positioning,calc}
\usepackage{etoolbox}
\usepackage{xcolor}
\usepackage{algorithm}
\usepackage{algpseudocode}
\usepackage{amsmath,amsfonts}
\usepackage{ulem}
\usepackage{todonotes}
\usepackage{relsize}
\usepackage{caption}
\usepackage{subcaption}
\usepackage{amssymb}
\usepackage[default]{mycolor}
\usepackage{appendix}

\newboolean{paper1}
\newboolean{paper2}
\newboolean{paper3}
\boolfalse{paper1}
\boolfalse{paper2}
\booltrue{paper3}

\usepackage{algorithm}
\usepackage{algpseudocode}
\usepackage{amsmath,amsfonts}
\usepackage{microtype}
\usepackage{relsize}
\usepackage{caption}
\usepackage{subcaption}

\newenvironment{compositeproof}{\begin{proof}}{\end{proof}}

\let\Halmos\qed

\title{Extracting structure from functional expressions for continuous and discrete relaxations of MINLPs}
\author{Taotao He\\
Antai College of Economics and Management\\ Shanghai Jiao Tong University \\
\texttt{hetaotao@sjtu.edu.cn} 
\And
Mohit Tawarmalani\\
Krannert School of Management\\
Purdue University \\
\texttt{mtawarma@purdue.edu}
}
\date{}
\def\Halmos{\ensuremath{\square}}

\begin{document}
\maketitle
\begin{abstract}%

In this paper, we develop new continuous and discrete relaxations for nonlinear expressions in an MINLP. In contrast to factorable programming, our techniques utilize the inner-function structure by encapsulating it in a polyhedral set, using a technique first proposed in~\cite{he2021new}. We tighten the relaxations derived in~\cite{tawarmalani2013explicit,he2021tractable} and obtain new relaxations for functions that could not be treated using prior techniques. We develop new discretization-based mixed-integer programming relaxations that yield tighter relaxations than similar relaxations in the literature. These relaxations utilize the simplotope that captures inner-function structure to generalize the incremental formulation of~\cite{dantzig1960significance} to multivariate functions. In partic- ular, when the outer-function is supermodular, our formulations require exponen- tially fewer continuous variables than any previously known formulation.
\keywords{Factorable programming  \and Staircase triangulation \and Incremental formulation \and MIP relaxations \and Supermodularity}
\end{abstract}%

\keywords{Mixed-integer nonlinear programs; Factorable programming; Supermodularity; Staircase triangulation; Convexification via optimal transport}

\def \mcirc {\mathop{\circ}}

\def \CR {R} 
\def \J {{\bar{J}}} 
\def \s {{\bar{s}}} 
\def \u {{\bar{u}}} 
\def \p {{\bar{\phi}}} 
\def \z {{\bar{z}}} 
\def \ts {{\tilde{s}}}
\def \tu {{\tilde{u}}}
\def \ta {{\tilde{\alpha}}}
\def \tb {{\tilde{b}}}
\def \mJ {{\mathcal{J}}} 
\def \mX {{\mathcal{X}}}
\section{Introduction}
Mixed-integer nonlinear programming (MINLP) algorithms and software rely on factorable programming techniques and convexification results for specially structured functions to relax composite functions 
\cite{mccormick1976computability,tawarmalani2004global,belotti2009branching,misener2014antigone,vigerske2018scip,nagarajan2019adaptive}. Classical factorable programming does not utilize any structural information on the inner function besides bounds~\cite{mccormick1976computability}, while special-structured results are only available for specific function types~\cite{rikun1997convex,sherali1997convex,benson2004concave,meyer2004trilinear}. The approach introduced recently in \cite{he2021new,he2021tractable} develops a framework that lies between these extremes. The technique allows the inner-function structure to be exploited in a completely generic fashion without the need for special-structure identification \cite{he2021new} and develops tractable relaxations when outer-function satisfies certain technical conditions \cite{he2021tractable}.

In the first part of the paper, we extend the composite relaxation technique of \cite{he2021new,he2021tractable} to develop new insights and tighter relaxations. This development relies on an alternate simpler proof of validity for the cuts which were recently derived in~\cite{he2021tractable} for the case when the outer-function is supermodular. With this new proof, the main convex hull result of~\cite{he2021tractable} follows by verifying that the affine cuts interpolate the function on one of the triangulations of the simplotope.
The proof also reveals various extensions. First, the main argument does not rely on the concave-extendability, a key assumption in \cite{he2021tractable}.
Second, we are able to exploit even more structure of the composite function to develop tighter relaxations for certain functions than in \cite{he2021tractable}. 
Third, we show that tighter relaxations can be derived by using bounding functions instead of bounds for underestimators. Fourth, we consider extensions where inner-functions are replaced with vectors of functions.



In the second part of the paper, we use discretization schemes to derive MIP relaxations of nonlinear programs. The progress in mixed-integer programming (MIP) has sparked significant interest in such relaxations~\cite{misener2011apogee,misener2012global,gupte2013solving,huchette2019combinatorial,nagarajan2019adaptive}.
In contrast to our techniques, following the factorable programming paradigm, earlier MIP relaxations decompose composite functions into univariate and multilinear functions and then discretize the multilinear expressions~\cite{misener2011apogee,misener2012global,huchette2019combinatorial}. Unfortunately, the decomposition step abstracts away much of the nonlinear structure. 
Instead, we model the inner function structure using a simplotope, a set that generalizes the incremental simplex often used to derive ideal formulations of piecewise-linear univariate functions~\cite{dantzig1960significance}. Then, we derive new relaxations utilizing the convex/concave envelopes of the outer-function over this simplotope. We show that these relaxations are tighter than those obtained using prevalent discretization schemes. Moreover, when the envelopes are polynomially separable, our relaxations require exponentially fewer continuous variables. Interestingly, ideality of our relaxations is guaranteed when the envelopes are extendable from vertices. Finally, the relaxations we construct, upon branching, automatically tighten by using the local bounding information for underestimators of inner-functions.

The paper is laid out as follows. First, in Section~\ref{section:exp}, we give a new validity proof of inequalities in \cite{he2021tractable}. Then, we discuss various extensions and generalizations.
In Section~\ref{section:discrete-relaxations}, we construct MIP relaxations for composite functions. This construction generalizes the incremental formulation~\cite{dantzig1960significance} from one-dimensional setting to an $n$-dimensional setting by using the proof technique of Section~\ref{section:exp}. We show various structural properties of our constructions.
First, we show that our relaxations strengthen prevalent MIP relaxations by exploiting inner-function structure. Second, we show that whenever the outer-function can be relaxed over a simplotope, an MIP discretization can surprisingly be obtained by simply appending a few inequalities involving binary variables. Third, we provide a new model of the simplotope that is structured to exploit local bounding information for the underestimators of inner-functions over each discrete piece. 

\paragraph{\bf{Notation.}} 
We shall denote the convex hull of set $S$ by $\conv(S)$, the projection of a set $S$ to the space of $x$ variables by $\proj_x(S)$, the extreme points of $S$ by $\vertex(S)$, and the convex (resp. concave) envelope of $f(\cdot)$ over $S$ by $\conv_{S}(f)(\cdot)$ (resp. $\conc_S(f)(\cdot)$). Let $g:D \to \R^m$ be a vector of functions. The restriction of $g$ to a subset of $S$ of $D$ is defined as:
\[
g|_S = \begin{cases}
	g(x) & x \in S \\
	-\infty & \text{otherwise}.
\end{cases}
\]
The hypograph of $g$ is defined as $\hypo(g): =\{ (x,\mu) \mid \mu \geq g(x), x \in D \}$, and the graph of $g$ is defined as $\graph(g): =\{ (x,\mu) \mid \mu =  g(x), x \in D \}$.
\section{Staircase expansions and termwise relaxations}\label{section:exp}
\def\L{\mathbb{L}}
\def\G{\mathcal{G}}
\def\extreme{\mathop{\text{ext}}}
\def\Qsimplex{\upsilon}
\def\V{\mathcal{V}}
\def\one{\mathbbm{1}}
\def\sfunct{\eta}
\def\DO{\mathcal{D}}
\def\BO{\mathcal{B}}
\def\mcirc{\mathop{\circ}}

Consider a composite function $\phi \mcirc f \colon X \subseteq \R^m \to \R$ defined as $(\phi \mcirc f)(x) = \phi\bigl(f(x)\bigr)$. Here, for $x  \in X$, $f(x) : = \bigl(f_1(x), \ldots, f_d(x) \bigr)$, where each $f_i \colon \R^m \to \R$ will be referred to as \textit{inner-function} while $\phi \colon \R^d \to \R$ will be referred to as the \textit{outer-function}.  In this section, we are interested in deriving overestimators for the composite function $\phi \mcirc f$. Towards this end, we will first express $\phi \mcirc f $ as a telescoping sum in a specific way, and then overestimate each summand. 
In order to telescope the composite function $\phi \mcirc f $, we will use the following information on the structure of inner-functions $f(\cdot)$. First, we assume that inner-functions $f( \cdot)$ are bounded, \textit{i.e.}, for every $x \in X$, $f(x) \in [f^L,f^U]$ for some vectors $f^L$ and $f^U$ in $\R^d$.
 Second, we assume that there is a vector of bounded underestimators of the inner-function $f_i(\cdot)$. More precisely, let $(n_1,\ldots,n_d)\in \Z^d$, we assume that $u: \R^m \to \R^{\sum_{i = 1}^d(n_i+1)}$ and $a: \R^m \to \R^{\sum_{i = 1}^d(n_i+1)}$ are function vectors that satisfy the following inequalities for all $x \in X$
\begin{equation}\label{eq:ordered-oa}
	\begin{aligned}
		&f^L_i \leq a_{i0}(x) \leq \ldots \leq a_{in_i}(x) \leq f^U_i,\\
		&u_{ij}(x) \leq \min \bigl\{f_i(x), a_{ij}(x)\bigr\} \qquad \text{for all } j \in \{0, \ldots, n_i\}, \\
		&u_{i0}(x) = a_{i0}(x), u_{in_i}(x) = f_{i}(x).
	\end{aligned}
\end{equation} 
We refer to $\bigl(u_{ij}(\cdot)\bigr)_{j = 0}^{n_i}$ as underestimators and $\bigl(a_{ij}(\cdot)\bigr)_{j = 0}^{n_i}$ as bounding functions. Constraint~(\ref{eq:ordered-oa}), inspired from \cite{he2021new}, requires that, for $i \in \{1, \ldots, d\}$ and $x\in X$, the pair $\bigl(u_i(\cdot),a_i(\cdot)\bigr)$ satisfies the following conditions: \begin{enumerate*}[label=(\roman*)]
	\item the bounding functions $\bigl(a_{ij}(\cdot)\bigr)_{j = 0}^{n_i}$ increase with $j$;
	\item for $j \in \{0, \ldots, n_i\}$, $u_{ij}(\cdot)$ underestimates $f_i(\cdot)$ and $a_{ij}(\cdot)$; and
	\item the first underestimator $u_{i0}(\cdot)$ (resp. the last underestimator $u_{in_i}(\cdot)$) matches the smallest bounding function $a_{i0}(\cdot)$ (resp. the inner-function $f_i(\cdot)$).
\end{enumerate*}
For notational convenience, we assume without loss of generality that $n_1 = \cdots = n_d := n$ for some $n \geq 1$.  We remark that it is shown in \cite{he2021new} that, if $(o_{ij}(x),a_{ij}(x))$ is such that $o_{ij}(x)\ge \min\{a_{ij}(x), f_i(x)\}$ then this information can be used within \eqref{eq:ordered-oa} by using $(u_{ij}(x),a_{ij}(x))$ instead, where $u_{ij}(x) = f_{i}(x) - o_{ij}(x) + a_{ij}(x)$, because it can be easily verified that $u_{ij}(x)\le \max\{a_{ij}(x),f_{i}(x)\}$.  

Throughout this paper, we introduce a vector of variables $u_i$ (resp. $a_i$) to represent the vector of functions $u_i(\cdot)$ (resp. $a_i(\cdot)$). In addition, we associate the subvector $(u_{1p_1}, \ldots, u_{dp_d})$ of $u:=(u_1, \ldots, u_d)$ with the point $p:=(p_1, \ldots, p_d)$ on a grid $\G$ given by $ \{ 0, \ldots, n\}^d$, and thus denote the subvector as $\Pi(u;p)$. The telescoping sum expansion of $\phi \mcirc f$ will be derived using lattice paths on $\G$. A lattice path in $\G$ is a sequence of points $p^{0}, \ldots, p^{r}$ in $\G$ such that $p^{0}= (0, \ldots, 0)$ and $p^{r}= (n, \ldots, n)$. In particular, a \textit{staircase} is a lattice path of length $dn+1$ such that for all $t \in \{1, \ldots, dn\}$, $p^{t} - p^{t-1} = e_{i_t}$ where $i_t \in \{1, \ldots, d\}$ and $e_{i_t}$ is the principal vector in $i_t^{\text{th}}$ direction. We refer to the movement $p^{t-1}$ to $p^{t}$ as the $t^{\text{th}}$ move. Clearly, there are exactly $n$ moves along each coordinate direction. Therefore, a staircase can be specified succinctly as $\omega:= (\omega_1, \ldots, \omega_{dn})$, where, for $t \in \{1, \ldots, dn \}$, we let $\omega_t = i$ if $p^{t} - p^{t-1} =e_i$.  We will refer to such a vector as \textit{direction vector} in $\G$, and will denote by $\Omega$ the set of all direction vectors in the grid $\G$. It follows easily that $|\Omega| = \frac{(dn)!}{(n!)^d}$. In Figure~\ref{fig:staircase}, we depict three staircases in the $3 \times 3$ grid specified by vectors $(2,2,1,1)$, $(2,1,2,1)$, and $(1,2,1,2)$, respectively.

\begin{figure}[h]
\centering
\begin{subfigure}{0.25\linewidth}	
\begin{tikzpicture}[scale=0.7]
\begin{scope}
    \node[below] at (0.5 ,0) {\scriptsize $0$};
    \node[below] at (1.5 ,0) {\scriptsize $1$};
    \node[below] at (2.5 ,0) {\scriptsize $2$};
    \node[left] at (0, 0.5) {\scriptsize $0$};
    \node[left] at (0, 1.5) {\scriptsize $1$};
    \node[left] at (0, 2.5) {\scriptsize $2$};
   	\filldraw[black] (0.5,0.5) circle (3pt);
   	\filldraw[black] (0.5,1.5) circle (3pt);
   	\filldraw[black] (0.5,2.5) circle (3pt);
   	\filldraw[black] (1.5,2.5) circle (3pt);
   	\filldraw[black] (2.5,2.5) circle (3pt);
	\draw[->,dashed,gray,ultra thick]  (0.5,0.6) --  (0.5,1.4) node [right] at (0.4,0.8) {{\color{black} \scriptsize $\omega_1$}};
	\draw[->,dashed,gray,ultra thick]  (0.5,1.6) --  (0.5,2.4) node [right] at (0.4,1.8) {{\color{black} \scriptsize $\omega_2$}};
	\draw[->,dashed,gray,ultra thick]  (0.6,2.5) --  (1.4,2.5) node [above] at (0.8,2.5) {{\color{black} \scriptsize $\omega_3$}};
	\draw[->,dashed,gray,ultra thick]  (1.6,2.5) --  (2.4,2.5) node [above] at (1.8,2.5) {{\color{black} \scriptsize $\omega_4$}};
    \draw (0, 0) grid (3, 3);
\end{scope}
\end{tikzpicture}
\caption{$\omega = (2,2,1,1)$}
\end{subfigure}
\begin{subfigure}{0.25\linewidth}	
\begin{tikzpicture}[scale=0.7]
\begin{scope}
    \node[below] at (0.5 ,0) {\scriptsize $0$};
    \node[below] at (1.5 ,0) {\scriptsize $1$};
    \node[below] at (2.5 ,0) {\scriptsize $2$};
    \node[left] at (0, 0.5) {\scriptsize $0$};
    \node[left] at (0, 1.5) {\scriptsize $1$};
    \node[left] at (0, 2.5) {\scriptsize $2$};
   	\filldraw[black] (0.5,0.5) circle (3pt);
   	\filldraw[black] (0.5,1.5) circle (3pt);
   	\filldraw[black] (1.5,1.5) circle (3pt);
   	\filldraw[black] (2.5,1.5) circle (3pt);
   	\filldraw[black] (2.5,2.5) circle (3pt);
	\draw[->,dashed,gray,ultra thick]  (0.5,0.6) --  (0.5,1.4) node [right] at (0.4,0.8) {{\color{black} \scriptsize $\omega_1$}};
	\draw[->,dashed,gray,ultra thick]  (0.6,1.5) --  (1.4,1.5) node [above] at (0.8,1.5) {{\color{black} \scriptsize $\omega_2$}};
	\draw[->,dashed,gray,ultra thick]  (1.6,1.5) --  (2.4,1.5) node [above] at (1.8,1.5) {{\color{black} \scriptsize $\omega_3$}};
	\draw[->,dashed,gray,ultra thick]  (2.5,1.6) --  (2.5,2.4) node [right] at (2.4,1.8) {{\color{black} \scriptsize $\omega_4$}};
    \draw (0, 0) grid (3, 3);
\end{scope}
\end{tikzpicture}
\caption{$\omega = (2,1,2,1)$}
\end{subfigure}
\begin{subfigure}{0.25\linewidth}	
\begin{tikzpicture}[scale=0.7]
\begin{scope}
    \node[below] at (0.5 ,0) {\scriptsize $0$};
    \node[below] at (1.5 ,0) {\scriptsize $1$};
    \node[below] at (2.5 ,0) {\scriptsize $2$};
    \node[left] at (0, 0.5) {\scriptsize $0$};
    \node[left] at (0, 1.5) {\scriptsize $1$};
    \node[left] at (0, 2.5) {\scriptsize $2$};
   	\filldraw[black] (0.5,0.5) circle (3pt);
   	\filldraw[black] (1.5,0.5) circle (3pt);
   	\filldraw[black] (1.5,1.5) circle (3pt);
   	\filldraw[black] (2.5,1.5) circle (3pt);
   	\filldraw[black] (2.5,2.5) circle (3pt);
	\draw[->,dashed,gray,ultra thick]  (0.6,0.5) --  (1.4,0.5) node [above] at (0.8,0.5) {{\color{black} \scriptsize $\omega_1$}};
	\draw[->,dashed,gray,ultra thick]  (1.5,0.6) --  (1.5,1.4) node [right] at (1.4,0.8) {{\color{black} \scriptsize $\omega_2$}};
	\draw[->,dashed,gray,ultra thick]  (1.6,1.5) --  (2.4,1.5) node [above] at (1.8,1.5) {{\color{black} \scriptsize $\omega_3$}};
	\draw[->,dashed,gray,ultra thick]  (2.5,1.6) --  (2.5,2.4) node [right] at (2.4,1.8) {{\color{black} \scriptsize $\omega_4$}};
    \draw (0, 0) grid (3, 3);
\end{scope}
\end{tikzpicture}
\caption{$\omega = (1,2,1,2)$}
\end{subfigure}
\caption{direction vectors in the $3 \times 3$ grid.} \label{fig:staircase}
\end{figure}
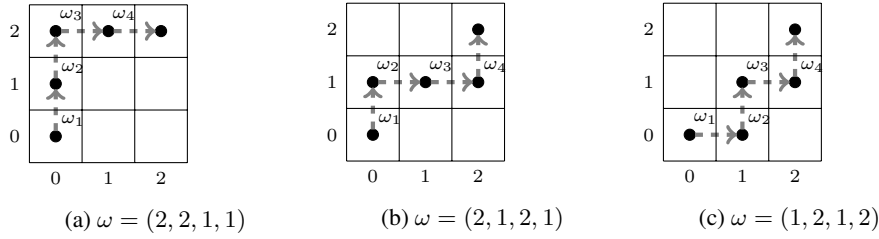
Given  $u \in \R^{d\times(n+1)}$ and a direction vector $\omega \in \Omega$, a staircase $p^{0}, \ldots, p^{dn}$ is specified as $p^{ 0} = 0$ and, for $t \in \{1, \ldots, dn\}$, $p^{t} = p^{ t-1} +e_{\omega_t}$, thus, yielding a sequence $\Pi(u;p^0), \ldots, \Pi(u;p^{dn})$. 
This sequence is associated with an expansion  
\begin{equation}\label{eq:str-exp}
	\begin{aligned}
\phi(u_{1n}, \ldots, u_{dn}) &=  \phi \bigl(\Pi(u,p^0)\bigr) +  \sum_{t=1}^{dn}\Bigl( \phi\bigl(\Pi(u; p^t)\bigr) - \phi\bigl(\Pi(u;p^{t-1})\bigr)\Bigr) \\
& =:\DO^\omega(\phi)(u), 
	\end{aligned}
\end{equation}
where the right hand side of the first equality telescopes and yields $\phi(u_{1n}, \ldots, u_{dn})$. 
Therefore, for any vector of functions $u: X \to \R^{d \times (n+1)}$ so that, for each $i \in \{1, \ldots, d\}$, $u_{in}(x) = f_i(x)$, we obtain that, for every $x \in X$, $(\phi \mcirc f)(x) =\bigl(\DO^\omega(\phi) \mcirc u \bigr)(x)$,
and we shall refer to the latter as a \textit{staircase expansion} of $\phi \mcirc f$. 

The staircase expansion $\DO^\omega(\phi)\bigl(u_1(x), \ldots, u_d(x)\bigr) $ will be relaxed in a termwise manner. In particular, this termwise relaxation is achieved by replacing the underestimating functions $u(\cdot)$ with their bounding functions $a(\cdot)$. Before we formally describe how the replacement is done, we illustrate the main idea using the following example. The example analyzes the product of two functions that was recently studied in~\cite{he2021new}. Here, we provide an alternative derivation of the validity of one of the inequalities derived in Theorems 1 and 5 of~\cite{he2021new}. The remaining inequalities admit similar derivations by altering the staircase used in the expansion. 
 
\begin{example}[Theorem 5 in~\cite{he2021new}]\label{ex:factorable-over}
Consider the product of two functions $f_1(x)f_2(x)$ over a convex set $X \subseteq \R^m$. For $i = 1,2$, assume that there exist $u_i(x):=\bigl(u_{i0}(x), u_{i1}(x), u_{i2}(x) \bigr)$ and $a_i:=(a_{i0}, a_{i1}, a_{i2})$ such that, for $x \in X$, 
\begin{equation*}\label{eq:order-ex}
 a_{i0} \leq a_{i1} \leq a_{i2},  \quad u_{i0}(x) = a_{i0}, \quad u_{i1}(x) \leq u_{i2}(x) = f_i(x),  \quad \text{and} \quad  u_{i1}(x) \leq a_{i1}. 
\end{equation*}
In other words, the pair $(u(x), a)$ satisfies the requirement in~(\ref{eq:ordered-oa}). Observe that there are $\frac{4!}{2!2!} = 6$ staircases which yield $6$ inequalities detailed in Theorem 5 of~\cite{he2021new}. In the following, we explicitly derive one of those inequalities that corresponds to the staircase $\bigl((0,0), (1,0), (1,1), (2,1),(2,2)\bigr)$ and does not admit a simple derivation using RLT (see Example~3 in \cite{he2021new}). Let $x \in X $, and let $(s_{i0}, s_{i1}, s_{i2})$ denote $\bigl(u_{i0}(x), \max\{u_{i0}(x),u_{i1}(x)\}, u_{i2}(x)\bigr)$. Then, we obtain
\[
\begin{aligned}
	f_1(x)f_1(x) &= s_{12}s_{22} \\
	&= s_{10}s_{20} + (s_{11}-s_{10})s_{20} +  s_{11} (s_{21} - s_{20})  + (s_{12}-s_{11})s_{21} \\
	&\qquad \qquad + s_{12}(s_{22} - s_{21}) \\
	& \leq a_{10}a_{20} + (s_{11}-s_{10})a_{20} +  a_{11}(s_{21}- s_{20}) + (s_{12} - s_{11})a_{21} \\
	& \qquad \qquad + a_{12}(s_{22} - s_{21}) \\
	& = a_{10} a_{20} - s_{10}a_{20} - a_{11}s_{20} + s_{11}(a_{20} - a_{21})   +  (a_{11} - a_{12})s_{21}\\
	&\qquad \qquad + s_{12}a_{21} + a_{12}s_{22} \\
	& \leq - a_{11}a_{20}+  (a_{20}-a_{21})u_{11}(x) + (a_{11}-a_{12})u_{21}(x)  + a_{21}u_{12}(x) \\
	&\qquad \qquad + a_{12}u_{22}(x), 
\end{aligned}
\]
where the first equality follows by definition, the second equality holds due to staircase expansion \eqref{eq:str-exp}, the first and second inequalities follow from termwise relaxations, and the third equality holds by rearrangement.\Halmos
\end{example}

Now, we formally describe the termwise relaxations used in Example~\ref{ex:factorable-over}. For a vector $v: = (v_1, \ldots, v_d) \in \R^{d \times (n+1)}$ and for $ i \in \{1, \ldots, d\}$, we denote by $v_{-i}$ the subvector $(v_1, \ldots, v_{i-1}, v_{i+1}, \ldots, v_d)$. For any direction vector $\omega \in \Omega$, let $\BO^\omega(\phi) \colon \R^{d\times (n+1)+d\times (n+1)} \to \R$ be a function so that 
\begin{equation}\label{eq:step-rlx}
\begin{aligned}
\BO^\omega(\phi)(u,a) &=  \phi(a_{1 0}, \ldots, a_{d 0})  + \sum_{t = 1}^{dn} \Bigl[ \phi\bigl(\Pi\bigl((a_{-\omega_t}, u_{\omega_t});p^t\bigr)\bigr) \\
& \qquad \qquad - \phi\bigl(\Pi\bigl( (a_{-\omega_t}, u_{\omega_t});p^{t-1}\bigr)\bigr)  \Bigr],
\end{aligned}
\end{equation}
where $(a_{-i},u_i)$ denotes the vector $(a_1, \ldots, a_{i-1}, u_i, a_{i+1}, \ldots, a_d)$.  Consider function $\DO^\omega(\phi)(\cdot)$ defined in~(\ref{eq:str-exp}). Then, observing  $\phi\bigl(\Pi(u,p^0)\bigr) = \phi(a_{10}, \ldots, a_{d0})$, we obtain $\BO^\omega(\phi)(u,a)$ from the function $\DO^\omega(\phi)(u)$ by replacing the summation of telescoping difference terms with the summation in~(\ref{eq:step-rlx}). This step changes the $t^{\text{th}}$ telescoping difference term $\phi\bigl(\Pi(u;p^t)\bigr) - \phi\bigl(\Pi(u;p^{t-1})\bigr)$ to 
\[
\begin{aligned}
&\phi(a_{1p^t_1}, \ldots, a_{\omega^t-1 p^t_{\omega^t-1}}, u_{\omega^t p^t_{\omega^t}}, a_{\omega^t+1 p^t_{\omega^t+1}}, \ldots, a_{dp^t_d} )  \\
& \qquad \qquad - \phi(a_{1p^t_1}, \ldots, a_{\omega^t-1 p^t_{\omega^t-1}}, u_{\omega^t p^t_{\omega^{t}}-1}, a_{\omega^t+1 p^t_{\omega^t+1}}, \ldots, a_{dp^t_d} ),
\end{aligned}
\] 
which is denoted as $\phi\bigl(\Pi\bigl((a_{-\omega_t}, u_{\omega_t});p^t\bigr)\bigr) - \phi\bigl(\Pi\bigl( (a_{-\omega_t}, u_{\omega_t});p^{t-1}\bigr)\bigr)$ in~(\ref{eq:step-rlx}).
The main result in this section is to show that the \textit{supermodularity} of the outer-function $\phi(\cdot)$ allows this replacement in the construction of an overestimator for $\phi\mcirc f$.
\begin{definition}[\cite{topkis2011supermodularity}]
A function $\sfunct(x): S \subseteq \R^n \to \R$ is said to be supermodular if $\sfunct(x' \vee x'') + \sfunct(x' \wedge x'') \geq \sfunct(x') + \sfunct(x'')$ for all $x'$, $x'' \in S$. Here, $x' \vee x''$ (resp. $x' \wedge x''$) denotes the component-wise maximum (resp. minmum), and we assume that $x'\vee x''$ and $x'\wedge x''$ belong to $S$ whenever $x'$ and $x''$ belong to $S$. \Halmos

\end{definition}
Although detecting whether a function is supermodular is NP-Hard \cite{crama1989recognition}, there are important special cases where this property can be readily detected \cite{topkis2011supermodularity}. For example, a product of nonnegative, increasing (decreasing) supermodular functions is nonnegative increasing (decreasing) and supermodular; see Corollary 2.6.3 in~\cite{topkis2011supermodularity}. Also, a conic combination of supermodular functions is supermodular. A canonical example of a supermodular function is $\prod_{i=1}^n x_i$ over the non-negative orthant.

\begin{theorem}\label{them:stair-ineq}
Consider a composite function $\phi \mcirc f \colon X \subseteq \R^m \to \R$. 
If the outer-function $\phi(\cdot)$ is supermodular over $[f^L,f^U]$, for any  $\bigl(u(x),a(x)\bigr)$ satisfying~(\ref{eq:ordered-oa}),  
$(\phi \mcirc f)(x)  \leq  \BO^\omega(\phi)\bigl(u(x),a(x)\bigr)$ for  $x \in X$, where $\BO^\omega(\phi)$ is defined as in~(\ref{eq:step-rlx}).
\end{theorem}
\begin{compositeproof}
    See Appendix~\ref{app:stair-ineq}. \Halmos
\end{compositeproof}

To extend the applicability of Theorem~\ref{them:stair-ineq}, we consider a particular linear transformation, referred to as \textit{switching}~\cite{crama1989recognition}, that can be used to make the outer-function supermodular. Such a transformation is useful, for example, to derive underestimators for $f_1(x)f_2(x)$ as in Theorem~1 of \cite{he2021new}. To do so, we replace $f_1(x)f_2(x)$ with $f_1(x)\bigl(f_2^U-f_2(x)\bigr)$, where $f_2^U$ is an upper-bound on $f_2(x)$, before using Theorem~\ref{them:stair-ineq}. More generally, for  $f \in \R^d$ and $T \subseteq \{1, \ldots, d\}$, let $f(T)$ be the vector defined as $f(T)_i = f_i^U - f_i$ if $i \in T$ and $f(T)_i = f_i$ otherwise. For a function $\phi: \R^d  \to \R$, we define a function $\phi(T):\R^d \to \R$ such that $\phi(T)(f) = \phi\bigl(f(T)\bigr)$. In addition, we introduce two affine maps to switch estimating functions and their bounding functions. Let $U(T): u \mapsto \tilde{u}$ and $A(T): a \mapsto \tilde{a}$ so that, for $i \notin T$, $(\tilde{u}_i, \tilde{a}_i)=(u_i,a_i)$, and otherwise $\tilde{u}_{ij} =  a_{in-j} -  u_{in-j} + u_{in}$ and $\tilde{a}_{ij}= a_{in-j}$ for $j \in \{0, \ldots, n \}$.


\begin{corollary}\label{cor:stair-switching}
Consider a composite function $\phi \circ f:X \subseteq \R^m \to \R $, and consider a pair $\bigl(u(x),a(x)\bigr)$ satisfying~(\ref{eq:ordered-oa}). Let $T$ be a subset of $ \{0, \ldots,d\}$ for which $f(T)(\cdot)$ is supermodular over $\bigl\{f(T) \bigm| f^L \leq f \leq f^U\bigr\}$. Then, for every $x \in X$, $(\phi\circ f) (x)  \leq \BO^\omega(\phi)\bigl(\tilde{u}(x),\tilde{a}(x) \bigr)$,  where $\tilde{u}(x):= U(T)\bigl(u(x)\bigr)$ and $\tilde{a}(x): =A(T)\bigl(a(x)\bigr) $. 
\end{corollary}
\begin{compositeproof}
    See Appendix~\ref{app:stair-switching}. \Halmos
\end{compositeproof}

We remark that Theorem~\ref{them:stair-ineq} and Corollary~\ref{cor:stair-switching} can be used recursively to develop relaxations for functions specified using their expression tree. Consider an arbitrary node in such an expression tree. Inductively, we construct relaxations for each of its children. These relaxations yield underestimators and overestimators for the function represented by each child node. Then, for each of these estimators, bounds can be obtained in various ways, including, for example, interval arithmetic. The underestimating functions and bounds are then rearranged to satisfy \eqref{eq:ordered-oa}. Finally, the node in consideration is recursively relaxed using Theorem~\ref{them:stair-ineq} and/or Corollary~\ref{cor:stair-switching}. The resulting relaxation is tighter because information about the structure of children is retained via their estimators, not only in the relaxation of the parent, but, by induction, in the relaxation of all their ancestors. Next, we illustrate the use of this technique on an example, where we show that it produces tighter inequalities than those obtained using factorable programming (FP).

\begin{example}\label{ex:recursive}
Consider $x_1^2x_2^2x_3^2$ over  $[1,2]^3$, and consider an expression tree depicted in Figure~\ref{fig:expression-tree-1}, whose edges are labeled with bounds of tail nodes. Using information that the function $x_i^2$ is bounded from below (resp. above) by $1$ (resp. $4$) over $[1,2]$, FP yields the following convex underestimator for the root node's left child, $x_1^2x_2^2 \geq \max \bigl\{x_1^2 + x_2^2 - 1, 4x_1^2 + 4x_2^2 - 16 \bigr\}$.
Then, inferring that $x_1^2x_2^2 \in [1,16]$, FP recursively constructs the following convex underestimator for the root node, 
\[
x_1^2x_2^2x_3^2 \geq \max
\left\{
\begin{aligned}
&\max \{x_1^2 + x_2^2 - 1, 4x_1^2 + 4x_2^2 - 16 \} + x_3^2 -1 \\ 
&4 \max \bigl\{x_1^2 + x_2^2 - 1, 4x_1^2 + 4x_2^2 - 16 \bigr\} + 16 x_3^2 -64
\end{aligned} 
\right\}.
\]
If $e(x_1,x_2)\le \min\{7,x_1^2x_2^2\}$, Corollary~\ref{cor:stair-switching} shows $3 e(x_1,x_2) + x_1^2x_2^2  +   7x_3^2 - 28 \leq x_1^2x_2^2x_3^2$. We may choose $e(x_1,x_2) = x_1^2+x_2^2-1$ and relax $x_1^2x_2^2$ to $4x_1^2+4x_2^2 -16$,  obtaining a convex underestimator $7x_1^2 + 7x_2^2 +  7x_3^2 - 47$ for $x_1^2x_2^2x_3^2$. This function strictly dominates the factorable relaxation at various points in the domain $[1,2]^3$. 


\begin{figure}[h]
\begin{subfigure}{0.48\linewidth}
	\centering
	\begin{tikzpicture}[scale=0.6, every tree node/.style={draw,circle,minimum size=2.5em}, level distance = 1.3cm, sibling distance=0.3cm,
edge from parent/.style={draw, <-,edge from parent path={(\tikzparentnode) -- (\tikzchildnode)}}]
\Tree[ .{$*$} \edge node[auto=right,pos=.5] {$(1,16)$};[.{ $ *$}  \edge node[auto=right,pos=.5] {$(1,4)$};[.$\wedge$  \edge node[auto=right,pos=.5] {$(1,2)$};$x_1$ 2 ]  \edge node[auto=left,pos=.5] {$(1,4)$};[.$\wedge$ \edge node[auto=right,pos=.5] {$(1,2)$};$x_2$ 2 ] ]  \edge node[auto=left,pos=.5] {$(1,4)$};[. $\wedge $  \edge node[auto=right,pos=.5] {$(1,2)$};$x_3$   $2$ ] ]
\end{tikzpicture}
\caption{Bound propagation}\label{fig:expression-tree-1}
\end{subfigure}
\begin{subfigure}{0.48\linewidth}
 \begin{tikzpicture}[scale=0.6, every tree node/.style={draw,circle,minimum size=2.5em}, level distance = 1.3cm, sibling distance=0.3cm,
edge from parent/.style={draw, <-,edge from parent path={(\tikzparentnode) -- (\tikzchildnode)}}]
\Tree[ .{$*$} \edge node[auto=right,pos=.5] {$\bigl(e(x),e^U\bigr)$};[.{ $ *$}  \edge node[auto=right,pos=.5] {$\bigl(u_1(x),a_1\bigr)$};[.$\wedge$  \edge node[auto=right,pos=.5] {$(1,2)$};$x_1$ 2 ]  \edge node[auto=left,pos=.5] {$\bigl(u_2(x),a_2\bigr)$};[.$\wedge$ \edge node[auto=right,pos=.5] {$(1,2)$};$x_2$ 2 ] ]  \edge node[auto=left,pos=.5] {$\bigl(u_3(x),a_3\bigr)$};[. $\wedge $   \edge node[auto=right,pos=.5] {$(1,2)$};$x_3$   $2$ ] ]
\end{tikzpicture}
\caption{Estimator and bound propagation}\label{fig:expression-tree-2}
\end{subfigure}
\end{figure}
More generally, consider an expression tree depicted in Figure~\ref{fig:expression-tree-2}, whose edges are labeled with underestimators of tail nodes and their upper bounds over the box domain $[1,2]^3$. Namely, for $i \in \{1, 2, 3\}$, let $u_i(x):=(1,2x_i-1, x_i^2)$ and $a_i := (1,3,4)$. Exploiting the ordering relation, $u_{ij}(x) \leq x_i^2$, $u_{13}(x) = x_i^2$, and $u_i(x) \leq a_i$ for $i \in \{1,2\}$, we obtain a relaxation for the left child of the root node:
\[
x_1^2x_2^2 \geq \max \left \{
\begin{aligned}
e_0(x)& := 1  \\
e_1(x)&:= u_{12}(x) + u_{22}(x) - 1\\
e_2(x)&:=2u_{11}(x)+ u_{12}(x) +  2u_{21}(x) + u_{22}(x) - 9\\
e_3(x)&:=3u_{11}(x) + u_{12}(x) + 3u_{22}(x) - 12\\
e_4(x)&:=3u_{12}(x) + 3u_{21}(x) + u_{22}(x) - 12\\
e_5(x)&:=u_{11}(x) + 3u_{12}(x) +u_{21}(x) + 3u_{22}(x) - 15 	\\
e_6(x)&:= \max \bigl\{e_0, \ldots, e_5, 4u_{12}(x)+ 4u_{22}(x) -16\bigr\}
\end{aligned}
\right\}
.
\]
Then, we obtain upper bounds for these estimators as $e^U:= (1, 7,11,13,13,15,16)$. Additionally, the relation $e_i(x) \leq x_1^2x_2^2$ and $e_i(x) \leq e^U_i$, for all $i \in \{1, \ldots, 6\}$, is exploited to construct convex underestimators for the root node, which are listed in Appendix~\ref{app:inequalities}. Observe that $u_{ij}(\cdot)$ can be substituted with their defining relations to obtain inequalities that do not require introduction of variables  beyond factorable programming scheme. For notational simplicity, we considered the product of $x_1^2$, $x_2^2$ and $x_3^2$ but the construction generalizes naturally to $f(x)g(y)h(z)$, and as described before the example, to arbitrary expression trees, if each node can be transformed to be supermodular and/or submodular.
\Halmos
\end{example}

\subsection{Convex relaxations for composite functions}\label{section:stair-convex}



In this subsection, we are interested in constructing convex relaxations for the hypograph of a composite function. Here, we impose an additional requirement on the bounding functions $a(\cdot)$ introduced in~(\ref{eq:ordered-oa}). Namely, we assume that there exists a pair $\bigl(u(\cdot), a\bigr)$ satisfying~(\ref{eq:ordered-oa}), where $a:=(a_1, \ldots, a_d)$ is a vector in $\R^{d \times (n+1)}$ so that, for each $i$, $f_i^L= a_{i0} < \cdots < a_{in} = f_i^U$. For any  direction vector $\omega \in \Omega$, we denote by $\BO^\omega(\phi)(\cdot;a)$ the function~(\ref{eq:step-rlx}) to emphasize the dependence of $\BO^\omega(\phi)$ on the  vector $a$. 

Although, in general, the function $\BO^\omega(\phi)(\cdot;a)$ is not  concave, we identify a class of functions $\phi(\cdot)$ for which it is so. We remark that Proposition~\ref{prop:sup-concave} differs from the relaxation technique of~\cite{tawarmalani2013explicit} in that it does not require $\phi(\cdot)$ to be \textit{concave-extendable}, formally defined below, from the vertices of $[f^L,f^U]$.
\begin{definition}[\cite{tawarmalani2002convex,tawarmalani2013explicit}]
	A function $g: S \to \R$, where $S$ is a polytope, is said to be concave-extendable (resp. convex-extendable) from $Y \subseteq S$ if the concave (resp. convex) envelope of $g(y)$ is determined by $Y$ only, that is, the concave envelope of $g$ and $g|_Y$ over $S$ are identical, where $g|_Y$ is the restriction of $g$ to $Y$.\Halmos
\end{definition} 

\begin{proposition}\label{prop:sup-concave}
If $\phi:[f^L,f^U] \to \R$ is supermodular, for any permutation $\omega$ of $\{1, \ldots, d\}$, an overestimator for $\phi(\cdot)$ over $[f^L,f^U]$ is given as follows:
\begin{equation}\label{eq:sup-concave}
\begin{aligned}
\phi(f^L) &+ \sum_{i=1}^d\Biggl[ \phi\Biggl( \sum_{j=1}^{i-1} e_{\omega_j}f_{\omega_j}^U + e_{\omega_i}f_{\omega_i} + \sum_{j=i+1}^de_{\omega_j}f_{\omega_j}^L \Biggr) \\
&  \qquad \qquad \qquad \qquad \qquad  - \phi\Biggl( \sum_{j=1}^{i-1} e_{\omega_j}f_{\omega_j}^U +  \sum_{j=i}^de_{\omega_j}f_{\omega_j}^L \Biggr) \Biggr].
\end{aligned}
\end{equation}
If $\phi(\cdot)$ is  concave in each argument when others are fixed then~(\ref{eq:sup-concave}) is concave.

\end{proposition}
\begin{compositeproof}
Let $n = 1$, and define $s_{\cdot 0}=f^L$ and $s_{\cdot n} = f$. Then, for any staircase $\omega$ in the grid $\{0,1\}^d$, \textit{i.e.},  a permutation $\omega$ of $\{1, \ldots, d\}$, $\BO^\omega(\phi)(\cdot;a)$ reduces to~(\ref{eq:sup-concave}). Thus, the validity of~(\ref{eq:sup-concave}) follows from Theorem~\ref{them:stair-ineq}. Moreover, the concavity of~(\ref{eq:sup-concave}) follows from that of $\phi(\cdot)$ in each argument when others are fixed. \Halmos
\end{compositeproof}
The following example illustrates how Proposition~\ref{prop:sup-concave} can be applied to derive relaxations for functions that are not concave-extendable.
\begin{example}
Consider $ \phi(f_1,f_2):= 4f_1^{0.8}f_2^{0.8} - f_1^{0.6}f_2^{0.6}$, which is supermodular over $[1,2]^2$, and is concave in each argument when the other argument is fixed. By Proposition~\ref{prop:sup-concave}, we obtain that for each $f \in [1,2]^2$:
\[
\begin{aligned}
\phi(f_1,f_2) &= \phi(1, 1) + \bigl(\phi(f_1,1) - \phi(1,1)\bigr) + \bigl(\phi(f_1,f_2) - \phi(f_1,1)\bigr)\\
& \leq \phi(f_1,1) + \bigl(\phi(2,f_2) - \phi(2,1)\bigr).
\end{aligned}
\]
Similarly, we obtain another concave overestimator of $\phi(\cdot)$, that is $\phi(1,f_2) + \phi(f_1,2) - \phi(1,2)$. In contrast, factorable programming~\cite{mccormick1976computability} yields an overestimator 
$\min \bigl\{L_i(f) + R_j(f) \bigm| i = 1,2, \; j = 1,2\bigr\}$, where 
\[
\begin{aligned}
	L_1(f)&:=4f_1^{0.8}+ 2^{2.8}f_2^{0.8} - 2^{2.8}, \\
	L_2(f)&:=  2^{2.8}f_1^{0.8} + 4f_2^{0.8}  - 2^{2.8}, \\
	  R_1(f)&:= - \bigl((2^{0.6} - 1)f_1 -2^{0.6}+2\bigr) - \bigl((2^{0.6} - 1)f_2 -2^{0.6}+2\bigr) +1, \\
	 R_2(f)&:= -2^{0.6} \bigl((2^{0.6} - 1)f_1 -2^{0.6}+2\bigr) -2^{0.6} \bigl((2^{0.6} - 1)f_2 -2^{0.6}+2\bigr) +2^{1.2},
\end{aligned}
\]
where, in $R_1(f)$ and $R_2(f)$, we have replaced $f_i^{0.6}$ by its linear underestimator $(2^{0.6}-1)(f_i-1) +1$ over $[1,2]$. Then, for $f \in [1,2]^2$, we have 
\[
\begin{aligned}
	& \phi(f_1,1) +\phi(2,f_2) - \phi(2,1) - L_1(f)  = (2^{0.6} -f_1^{0.6}) - 2^{0.6} f_2^{0.6} \\
	& \quad \leq   \min\bigl\{ -f_1^{0.6} - f_2^{0.6} + 1, 2^{0.6}(f_1^{0.6} - 2^{0.6}) - 2^{0.6}f_2^{0.6} \bigr\}  \leq \min\bigl\{R_1(f), R_2(f)\bigr\}.
\end{aligned}
\]
Similarly, we obtain that $\phi(1,f_2) + \phi(f_1,2) - \phi(1,2) - L_2(f)\leq \min\{R_1(f),R_2(f)\}$. This shows that the relaxation derived from Proposition~\ref{prop:sup-concave}, which is at least as tight as factorable relaxation, can dominate the latter relaxation strictly.\Halmos
\end{example}

\def \ephi {\bar{\phi} }

Given a pair $\bigl(u(x),a\bigr)$ satisfying~(\ref{eq:ordered-oa}), it is not apparent that the overestimator $\BO^\omega(\phi)\bigl(u(x);a\bigr)$ from Theorem~\ref{them:stair-ineq} is concave. However, we can relax it further by exploiting the structure of $\BO^\omega(\phi)(\cdot;a)$.  Observe that the function $\BO^\omega(\phi)(\cdot;a)$, defined as in~(\ref{eq:step-rlx}), is additively \textit{separable} with respect to $(u_1, \ldots, u_d)$, \textit{i.e.}, it can be written as a sum of $d$ functions, where for each $i \in \{1, \ldots, d\}$, the $i^{\text{th}}$ function depends only on $u_i$. More specifically, let $\BO^\omega_i(\phi)(\cdot;a): \R^{n+1} \to \R$ be defined as
\[
\BO^\omega_i(\phi)(u_i;a): = \sum_{t: \omega_t=i}  \phi\Bigl(\Pi\bigl((u_{i}, a_{-i});p^t\bigr)\Bigr) - \phi\Bigl(\Pi\bigl((u_{i}, a_{-i});p^{t-1}\bigr)\Bigr).
\]
Then, $\BO^\omega(\phi)(u;a)$ from~(\ref{eq:step-rlx}) equals $\phi\bigl(\Pi(a;p^0)\bigr) + \sum_{i =1}^d \BO_i^\omega(\phi)(u_i;a)$. Therefore, given the same setup as Theorem~\ref{them:stair-ineq}, we obtain that for every $x \in X$ 
\begin{equation}~\label{eq:termwise}
\begin{aligned}
(\phi \mcirc f)(x)& \leq \bigl(\BO^\omega(\phi) \mcirc u\bigr)(x;a) \\
& \leq \phi\bigl(\Pi(a;p^0)\bigr) + \sum_{i=1}^d  \conc_{\conv(X)}\bigl(\BO_i^\omega(\phi) \mcirc u_i\bigr)(x;a),
\end{aligned}
\end{equation}
where the first inequality follows from Theorem~\ref{them:stair-ineq}, and the second inequality holds since  $\bigl(\BO_i^\omega(\phi) \mcirc u_i\bigr)(x;a)$ is relaxed to its concave envelope. The next result states that if  $f_i(\cdot)$ depends on a different set of variables restricted to lie in $X_i$ and $X=\prod_{i=1}^d X_i$, the second inequality in~(\ref{eq:termwise}) is the tightest possible. This is because the right hand side equals to the concave envelope of $\bigl(\BO^\omega(\phi) \mcirc u\bigr)(x;a)$ over $\conv(X)$. 

\begin{proposition}\label{prop:decomposition} 
Assume the same setup as Theorem~\ref{them:stair-ineq} except that $a(\cdot)$ is assumed to be a pre-specified vector $a$. Let $M_1, \ldots, M_d$ be a partition of $\{1, \ldots, m\}$, and assume that, for $x = (x_{M_1}, \ldots, x_{M_d})  \in X:=\prod_{i=1}^dX_i$ where $X_i$ is a subset of $\R^{|M_i|}$,  $f(x) = \bigl(f_1(x_{M_1}), \ldots, f_d(x_{M_d}) \bigr)$ and $u(x) = \bigl(u_1(x_{M_1}), \ldots, u_d(x_{M_d}) \bigr)$. Then, 
\[
\begin{aligned}
(\phi\mcirc f)(x)	&\leq  \phi\bigl(\Pi(a;p^0)\bigr) + \sum_{i=1}^d  \conc_{\conv(X_i)}\bigl(\BO_i^\omega(\phi) \mcirc u_i\bigr)(x_{M_i};a) \\
 &= 	\conc_{\conv(X)}\bigl(\BO^\omega(\phi) \mcirc u\bigr)(x;a),
\end{aligned} 
\]
for every $x \in \conv(X)$. \Halmos
\end{proposition}
 This result reduces the problem of relaxing a $m$ dimensional composite function to that of relaxing several lower-dimensional functions, \textit{e.g.}, univariate functions when $M_i = \{i\}$ for $i \in \{1, \ldots, d\}$.  This is useful since envelopes or tight relaxations for certain lower dimensional functions are available (see, for example, \cite{tawarmalani2001semidefinite,gounaris2008tight,anstreicher2010computable,tawarmalani2013explicit,locatelli2014convex,burer2015gentle}). We give an example in Appendix~\ref{app:example-more-underestimators} to demonstrate that the relaxed staircase inequalities obtained using additional underestimators of inner functions can help tighten the relaxation by revealing new inequalities for the composite function.

The main result of this subsection is to compare the strength of relaxations given by Proposition~\ref{prop:decomposition} and~\cite{he2021tractable}. Under a mild condition on the outer-function, we show in Theorem~\ref{them:improve_mor} that the former yields tighter relaxations than the latter. In order to show this result, we modify inequalities from Theorem~\ref{them:stair-ineq} to provide an alternative derivation of the relaxation in~\cite{he2021tractable}. More specifically, under the assumption that $\bar{\phi}(\cdot)$ is concave-extendable from $\vertex(Q)$, where $\ephi:\R^{d \times (n+1)} \to \R$ so that $\ephi(s) = \phi(s_{1n}, \ldots, s_{dn})$ and $Q:= \prod_{i=1}^dQ_i$, and $Q_i$ is a simplex in $\R^{n+1}$ with the following extreme points:
\begin{equation}\label{eq:Q-V}
v_{ij} = (a_{i0}, \ldots, a_{ij-1}, a_{ij}, \ldots, a_{ij} ) \qquad \text{for } j =  0, \ldots, n ,
\end{equation} 
we will argue in Propostion~\ref{prop:stair-Q} that, to obtain the concave envelope of $\ephi(\cdot)$ over $Q$, it suffices to linearly interpolate  $\BO_i^\omega(\phi)(\cdot;a)$ over $ \vertex(Q_i)$.

 We begin by briefly reviewing the construction in \cite{he2021new,he2021tractable} below. First, the ordering relationship in the pair $\bigl(u(x),a\bigr)$ satisfying~(\ref{eq:ordered-oa}) is encoded into a polytope $P:=\prod_{i=1}^d P_i$, where 
\begin{equation}\label{eq:P-3}
P_i = \left\{ u_i \in \R^{n+1} \,\middle|\,
  \begin{aligned}
a_{i0} \leq u_{i j} \leq  \min\{a_{i j}, u_{in}\},\ u_{i 0} = a_{i 0},\; a_{i 0} \leq u_{i n} \leq a_{i n}
  \end{aligned}
  \right\}.
\end{equation}
Then, a relaxation for  $\phi \mcirc f$ is obtained by convexifying the hypograph of $\ephi(\cdot)$ over the polytope $P$. It is showed in~\cite{he2021new} that the convexification over $P$ is equivalent to that over the subset $Q$ of $P$. Thus, the concave envelope of $\ephi(\cdot)$ over $P$, denoted as $\conc_P(\ephi)(u)$, can be easily obtained if that of $\ephi(\cdot)$ over $Q$, denoted as $\conc_Q(\ephi)(s)$.
Here, we summarize results  relevant to our discussion.
\begin{proposition}[\cite{he2021new}]\label{prop:CR}
Consider a vector of composite functions $\theta \mcirc f: X \to \R^\kappa$ defined as $(\theta \mcirc f)(x): = \bigl((\theta_1 \mcirc f)(x), \ldots, (\theta_\kappa \mcirc f) (x)  \bigr)$. Let $\bigl(u(x),a\bigr)$ be a pair satisfying~(\ref{eq:ordered-oa}), where $a:=(a_1, \ldots, a_d)$ so that $a_{i0} < \cdots < a_{in}$, and define $\Theta^P: = \bigl\{(u,\theta) \bigm| \theta = \theta(u_{1n}, \ldots, u_{dn}),\ u \in P \bigr\}$. Then, we obtain 
\[
\graph(\phi \mcirc f) \subseteq \Bigl\{ (x, \phi) \Bigm| (u,\theta) \in \conv\bigl(\Theta^P\bigr) ,\ u(x) \leq u,\  (x,u_{\cdot n}) \in W  \Bigr\},
\]
where $u_{\cdot n}=(u_{1n}, \ldots, u_{dn})$  and $W$ outer-approximates  $\{(x, u_{\cdot n}) \bigm| u_{\cdot n} = f(x), x\in X\}$. Moreover, for a convex relaxation $R$ of $\Theta^Q: = \bigl\{(s, \theta) \bigm| \theta = \theta(s_{1n}, \ldots, s_{dn}), s\in Q \bigr\}$, we obtain  
\begin{equation}\label{eq:extended-P}
\conv\bigl(\Theta^P\bigr) \subseteq  \bigl\{(u,\theta)\bigm| (s, \theta) \in R,\  u \in P,\ u\leq s,\ u_{\cdot n} = s_{\cdot n} \bigr\}, 
\end{equation}
where the equality holds if $R = \conv(\Theta^Q)$. \end{proposition}
\begin{proof}
	The third statement follows from Theorem 2 of~\cite{he2021new}, and the fourth statement follows directly from Lemma 7 of~\cite{he2021new}.   \Halmos 
\end{proof}
We now focus on the case when $\ephi(\cdot)$ is concave-extendable from $\vertex(Q)$. It can be shown that affinely interpolating $\BO_i^\omega(\phi)(\cdot;a)$ over vertices of $Q_i$ yields a concave overestimator for $\ephi(\cdot)$ over $Q_i$. This affine interpolation is given as follows:
\begin{equation}\label{eq:B-hat}
\hat{\BO}_i^\omega(\phi)(s_i;a) := \sum_{t:\omega_t=i} \frac{\phi\bigl(\Pi(a;p^t)\bigr) - \phi\bigl(\Pi(a;p^{t-1})\bigr)}{a_{i p^t_{i}} - a_{i p^{t-1}_{i}}} \bigl(s_{i p^t_{i}} - s_{i p^{t-1}_{i}}\bigr). 
\end{equation}
Consider the affine function $\hat{\BO}^{\omega}(\phi)(s;a)$ obtained from $\BO^{\omega}(\phi)(s;a)$ by replacing $\BO^\omega_i(\phi)(s_i;a)$  with its affine interpolation $\hat{\BO}^{\omega}_i(\phi)(s_i;a)$, that is  
\begin{equation}\label{eq:step-rlx-interpolation}
\begin{aligned}
\hat{\BO}^{\omega}(\phi)(s;a) & = \phi\bigl(\Pi(a;p^0) \bigr)  + \sum_{i=1}^d \hat{\BO}^{\omega}_i(\phi)(s_i;a).
\end{aligned}
\end{equation}
It can be further shown that $\min_{\omega \in \Omega} \hat{\BO}^\omega(\phi)(\cdot;a)$ yields an explicit description of the concave envelope of $\ephi(\cdot)$ over $Q$. 
\begin{proposition}[Theorem 2 and Corollary 4 in~\cite{he2021tractable}]\label{prop:stair-Q}
Assume that $\ephi(\cdot)$ is concave-extendable from $\vertex(Q)$. If $\phi(s_{1n}, \ldots, s_{dn})$ is supermodular over $[f^L,f^U]$ then $\conc_Q(\ephi)(s) = \min_{\omega \in \Omega}\hat{\BO}^\omega(\phi)(s;a)$ for every $s \in Q$. Moreover, let $T\subseteq \{1, \ldots, d\}$ for which $\phi(T)(s_{1n}, \ldots s_{dn})$ is supermodular over $\bigl\{f(T) \bigm| f^L\leq f\leq f^U \bigr\}$. We obtain that $\conc_{Q}(\ephi)(s) = \min_{\omega \in \Omega}\hat{\BO}^\omega(\phi)\bigl(U(T)(s);A(T)(a)\bigr)$ for every $s \in Q$, where $U(T)(\cdot)$ and $A(T)(\cdot)$ are defined preceding Corollary~\ref{cor:stair-switching}. 
\end{proposition}
\begin{compositeproof}
	See Appendix~\ref{app:stair-Q} for an alternate proof that uses Theorem~\ref{them:stair-ineq}. \Halmos
\end{compositeproof}
 Proposition~\ref{prop:stair-Q} makes it easy to derive the inequalities of \cite{he2021tractable} 
 by interpolating $\BO^\omega_i(\phi)(\cdot;a)$ over each simplex $Q_i$ using the explicit formula \eqref{eq:B-hat}. To illustrate the construction, we next specialize Proposition~\ref{prop:stair-Q} to three special cases. 
\begin{example}
Consider the multilinear monomial $m(f) = \prod_{i = 1}^d f_i$ over the hypercube $[f^L,f^U]$, where, for all $i$, $f_i^U> f_i^L \geq 0$. Consider an arbitrary staircase $\omega$ from the grid $\{0, \ldots, n\}^d$. Then, the function $\hat{\BO}^{\omega}(m)(s;a)$ defined in~(\ref{eq:step-rlx-interpolation}) reduces to: 
\[
\prod_{i'=1}^da_{i'0} + \sum_{i=1}^d\sum_{t:\omega_t=i}\biggl(\prod_{i'\neq i}a_{i'p^t_{i'}}\bigl(s_{ip^t_i} - s_{ip^{t-1}_{i}}\bigr)\biggr). 
\]
Since $m(\cdot)$ is supermodular over $[f^L,f^U]$ and concave-extendable from $\vertex(Q)$, by Proposition~\ref{prop:stair-Q}, the convex envelope of $m(\cdot)$ over $Q$ is given by $\min_{\omega \in \Omega}\hat{\BO}^\omega(m)(s;a)$. In particular, when $n=1$, a staircase $\omega$ from $\{0,1\}^d$ reduces to a permutation of $\{1, \ldots, d \}$, and we obtain the following inequalities from Theorem 1 of~\cite{benson2004concave}:
\[
m(f) \leq \prod_{i'=1}^df_{i'}^L  + \sum_{i = 1}^d\Biggl( \prod_{i' \in \{\omega_1,\ldots,\omega_{i-1}\}} f^U_{i'}\Biggr) \cdot \Biggl(\prod_{i' \in \{\omega_{i+1},\ldots,\omega_{d}\}}f^L_{i'} \Biggr) \cdot (f_{\omega_i} - f_{\omega_i}^L),
\]
since $a_{i'0} = f_{i'}^L$, $a_{i'p_{i'}^t} = f_{i'}^L$ if $i'\in \{\omega_{i+1},\ldots,\omega_{d}\}$, and $a_{i'p_{i'}^t} = f_{i'}^U$ if $i'\in \{\omega_{1},\ldots,\omega_{i-1}\}$.
\Halmos
%
\end{example} 
Since the bilinear term is convex- and concave-extendable from $\vertex(Q)$ and is supermodular over $[f_1^L,f_1^U] \times [f_2^L,f_2^U]$, while $s_{1n}(f_2^U-s_{2n})$ is submodular over $[f_1^L,f_1^U] \times [0, f_2^U-f_2^L]$, Proposition~\ref{prop:stair-Q} yields the convex and concave envelope of a bilinear term $s_{1n}s_{2n}$ over $Q$.

\begin{corollary}\label{cor:bilinear-stair}
The convex envelope of $s_{1n}s_{2n}$ over $Q$ is given by 
\begin{equation}\label{eq:bilinear-stair-1}
\begin{aligned}
	\check{b}(s) := \max_{\omega \in \Omega} \biggl\{ a_{10}a_{2n} &+ \sum_{t:\omega_t=1}a_{2(n-p^t_2)}\bigl(s_{1p^{t-1}_1} - s_{1p^{t}_1}\bigr) + \\
& \quad \sum_{t:\omega_t=2} a_{1p^t_1}\bigl(a_{2(n-p_2^t)} - a_{2(n-p_2^{t-1})} - s_{2p^{t}_2} + s_{2p^{t-1}_2}\bigr)\biggr\},
\end{aligned}
\end{equation}
and the concave envelope of $s_{1n}s_{2n}$ over $Q$ is given by
\begin{equation}\label{eq:bilinear-stair-2}
 \hat{b}(s) := \min_{ \omega \in \Omega }\biggl\{ a_{10}a_{20} +  \sum_{t:\omega_t=1} a_{2p^t_2}\bigl(s_{1p^t_1} - s_{1p^{t-1}_1}\bigr) + \sum_{t:\omega_t=2} a_{1p^t_1}\bigl(s_{2p^t_2} - s_{2p^{t-1}_2}\bigr)\biggr\},
\end{equation}
where $\Omega$ is the set of all direction vectors in $\{0, \ldots, n\}^2$. \Halmos
\end{corollary} 
Corollary~\ref{cor:bilinear-stair} generalizes Corollary~5 in \cite{he2021tractable} and provides a compact constructive derivation of the inequalities therein. Last, we observe that the hypercube $[0,1]^d$ arises as a special case of the Cartesian product of simplices $Q$ when $n=1$, $[f^L,f^U]=[0,1]^d$, and the variables $(s_{10}, \ldots, s_{d0})$ are projected out using $s_{i0}=0$. In this case, for each direction vector $\omega$ in $\{0, 1\}^d$, the function in~(\ref{eq:step-rlx-interpolation}) reduces to
\[
\hat{\phi}^\omega(f_1, \ldots, f_d) :=  \phi(0) + \sum_{i=1}^d\Biggl(\phi\biggl( \sum_{j=1}^i e_{\omega_j} \biggr) - \phi\biggl( \sum_{j=1}^{i-1} e_{\omega_j} \biggr)\Biggr)f_{\omega_i}.
\]
As a corollary of Proposition~\ref{prop:stair-Q}, we also obtain that the Lov\'asz extension describes the concave envelope of supermodular concave-extendable functions over $[0,1]^d$.
\begin{corollary}[Proposition 4.1 in~\cite{lovasz1983submodular} and Theorem 3.3 in~\cite{tawarmalani2013explicit}]\label{cor:lovasz}
If $\phi(\cdot)$ is concave-extendable from $\{0,1\}^d$ and supermodular when restricted to $\{0,1\}^d$ then $\conc_{[0,1]^d}(\phi)(f) = \min_{\omega \in \Omega}\hat{\phi}^{\omega}(f)$ for every $f \in [0,1]^d$. \Halmos
\end{corollary}

Now, we are ready to show that, under certain conditions, the relaxation given by Proposition~\ref{prop:decomposition} is tighter than that given by Proposition~\ref{prop:stair-Q}. 

\begin{theorem}~\label{them:improve_mor}
Consider  $\phi \mcirc f:X \to \R$, and consider a convex relaxation  $W$ of the graph of  $f(\cdot)$. Assume that $s:W \to \R^{d \times (n+1)}$ is a vector of convex functions such that $s(W) \subseteq Q$ and for $i\in \{1, \ldots, d\}$ $s_{in}(x,f) =f_i$. If $\phi(\cdot)$ is supermodular and convex in each argument when others are fixed, $\bigl\{(x,f,\phi) \bigm| \phi \leq \phi(f), (x,f) \in W \bigr\} \subseteq R_+ \subseteq R$, where $R$ and $R_+$ are two convex sets defined as follows:
\[
\begin{aligned}
R &:= \Bigl\{(x,f,\phi) \Bigm|  \phi \leq \min_{\omega \in \Omega} \hat{\BO}^\omega(\phi)\bigl(s;a\bigr) ,\  s(x,f) \leq s ,\ (x,f) \in W\Bigr\},\\
R_+ &:= \left\{(x,f,\phi) \left| \; \begin{aligned} &\phi \leq \min_{\omega \in \Omega}\phi\bigl(\Pi(a;p^0)\bigr) + \sum_{i=1}^d  \conc_{W}\bigl(\BO_i^\omega(\phi) \mcirc s_i\bigr)\bigl((x,f);a\bigr)   \\
&(x,f) \in W 
\end{aligned}
\right.
\right\}.
\end{aligned}
\]
\end{theorem}
\begin{compositeproof}
Clearly, $R$ and $R^+$ are  convex sets, and, by Theorem~\ref{them:stair-ineq}, $R_+$ is a relaxation of $\bigl\{(x,f,\phi) \bigm| \phi \leq \phi(f), (x,f) \in W \bigr\}$. Next, we show that $R_+ \subseteq R$. By Proposition~\ref{prop:stair-Q}, $\min_{\omega \in \Omega} \hat{\BO}^\omega(\phi)(s;a) = \conc_Q(\phi)(s)$. Then, by Lemma 9 of~\cite{he2021new}, for every $\omega \in  \Omega$, $\hat{\BO}^\omega(\phi)(s,a)$ is a linear function so that  for $i \in \{1, \ldots, d\}$ and $j \neq \{0,n\}$ the coefficient of $s_{ij}$ is non-positive. 
%
Together with the condition on $s(\cdot)$, we obtain that $R = \bigl\{(x,f,\phi) \bigm|  \phi \leq \min_{\omega \in \Omega} \hat{\BO}^\omega(\phi)\bigl(s(x,f);a\bigr) ,\ (x,f) \in W\bigr\}$.
Therefore, the proof is complete if we can show that, for  $\omega \in \Omega$ and $i \in \{1, \ldots, d\}$, $\conc_W(\BO_i^\omega (\phi) \mcirc s_i)\bigr((x,f);a\bigr) \leq  \hat{\BO}_i^\omega(\phi)\bigl(s_i(x,f);a\bigr)$ holds for  every $(x,f) \in W$.  This holds because $\hat{\BO}_i^\omega(\phi)(s_i(x,f);a)$ is concave on $W$, and, for $(x,f) \in W$ and
	\[
	\begin{aligned}
\BO^\omega_i(\phi)\bigl(s_i(x,f);a\bigr) &= \sum_{t:\omega_t=i} \Bigl[ \phi\bigl(a_{1p^t_1}, \ldots, a_{i-1p^t_{i-1}}, s_{ip^t_i}(x,f), a_{i+1p^t_{i+1}}, \ldots, a_{dp^t_{d}}\bigr) \\
 &  \quad   - \phi\bigl(a_{1p^{t-1}_1}, \ldots, a_{i-1p^{t-1}_{i-1}}, s_{ip^{t-1}_i}(x,f), a_{i+1p^{t-1}_{i+1}}, \ldots, a_{dp^{t-1}_{d}}\bigr) \Bigr]\\
 & \leq \sum_{t:\omega_t=i}  \frac{\phi\bigl(\Pi(a;p^t)\bigr) - \phi\bigl(\Pi(a;p^{t-1})\bigr)}{a_{i p^t_{i}} - a_{i p^{t-1}_{i}}} \bigl(s_{i p^t_{i}}(x,f) - s_{i p^{t-1}_{i}}(x,f)\bigr) \\
 & = \hat{\BO}^\omega_i(\phi)\bigl(s_i(x,f);a\bigr),
 	\end{aligned}
\]
where the inequality holds because $\phi(\cdot)$ is convex in each argument when others are fixed, and for every $t$ with $w_t=i$, $p^t - p^{t-1} = e_i$,  $s_{ip^{t-1}_i}(x,f) \leq s_{ip^t_i}(x,f)$, $s_{ip^{t-1}_i}(x,f) \leq a_{ip^{t-1}_i}$,  $s_{ip^t_i}(x,f) \leq a_{ip^t_i}$ and $a_{ip^{t-1}_i} \leq a_{ip^t_i}$.\Halmos
\end{compositeproof}
In Appendix~\ref{app:example-convex-each-argument}, we illustrate the improvement of Proposition~\ref{prop:decomposition} over Proposition~\ref{prop:stair-Q} using Example 1 from~\cite{he2021tractable}.

\subsection{Two Generalizations}\label{section:matrix_ineq}
Here, we start with generalizing the idea in Example~\ref{ex:factorable-over} to derive inequalities in the form of matrices. Let $F: X \subseteq  R^m \to \R^{p_1\times p_2}$ and $G:X \subseteq \R^m \to \R^{q_1 \times q_2}$ be two matrices of functions defined as follows:
\[
F(x) = \begin{pmatrix}
	f_{11}(x) & \cdots & f_{1r_2}(x) \\
	\vdots & \ddots & \vdots \\
	f_{r_11}(x) & \cdots & f_{r_1r_2}(x)
\end{pmatrix} \quad \text{and} \quad G(x) = \begin{pmatrix}
	g_{11}(x) & \cdots & g_{1q_2}(x) \\
	\vdots & \ddots & \vdots \\
	g_{q_11}(x)& \cdots & 	g_{q_1q_2}(x)
\end{pmatrix},
\]
where $f_{ij}(\cdot)$ and $g_{tk}(\cdot)$ are functions from $\R^m$ to $\R$. As before, our construction exploits structures of estimators for $F(\cdot)$ and $G(\cdot)$. Let $Q \succeq 0 $ denote that a matrix $Q$ is symmetric and positive semidefinite. For two matrices $A$ and $B$, we say $A \succeq B$ if $A - B \succeq 0$. Now,  we assume that there are matrices of functions $\bigl(U^i(\cdot),A^i(\cdot),V^i(\cdot),B^i(\cdot)\bigr)_{i=1}^n$ so that  for every $x \in X$
\begin{equation}\label{eq:matrix-system}
\begin{aligned}
	U^0(x) \preceq \cdots \preceq U^n(x)=F(x)& \qquad V^0(x) \preceq \cdots \preceq V^n(x) =G(x), \\
	U^0(x) =  A^0 (x) \preceq \cdots \preceq A^n(x)& \qquad V^0(x) = B^0(x) \preceq \cdots \preceq B^n(x),  \\ 
	 U^i(x) \preceq A^i(x)& \qquad U^i(x) \preceq B^i(x) \quad i = 0, \ldots, n.
\end{aligned}
\end{equation}
To define the product of functions $F(\cdot)$ and $G(\cdot)$, we introduce the Kronecker product of two matrices. For two matrices $A \in \R^{p_1 \times p_2}$ and $B \in \R^{q_1 \times q_2}$, the \textit{Kronecker product} $A  \otimes B$ is the $p_1q_1 \times p_2q_2$ matrix defined as
\[
A \otimes B = \begin{pmatrix}
	a_{11} B & \cdots & a_{1m} B\\
	\vdots & \ddots & \vdots \\
	a_{m1}B & \cdots & a_{mn}B
\end{pmatrix}.
\]
Some useful properties of the Kronecker product are summarized in the following result; see Chapter 4 in~\cite{horn1994topics}.
\begin{lemma}\label{lemma:Kron-prod}
Let $A, C \in \R^{p_1 \times p_2}$ and $B,D \in \R^{q_1 \times q_2}$. Then,	$(A + C) \otimes B = A\otimes B + C \otimes B $ and $A \otimes (B + D) = A \otimes B + A\otimes D$.
If $A \succeq 0$ and $B \succeq 0 $ then $A\otimes B \succeq 0$. \Halmos
\end{lemma}

\begin{proposition}\label{prop:Kronecker}
Consider two matrices of functions $F:\R^m \to \R^{p_1 \times p_2}$ and $G:\R^m \to \R^{q_1 \times q_2}$, and assume that there exists $\bigl(U^i(\cdot),A^i(\cdot),V^i(\cdot),B^i(\cdot)\bigr)_{i=1}^n$ satisfying~(\ref{eq:matrix-system}). Then, for a direction vector $\omega$ in the grid given by $\{0, \ldots, n\}^2$,  
\[
\begin{aligned}
F(x) \otimes G(x)& \preceq  A^0(x) \otimes B^0(x) +  \sum_{t:\omega_t = 1} \Bigl(U^{p^t_1}(x) - U^{p^{t-1}_1}(x)\Bigr) \otimes B^{p^t_2}(x)   \\
&  +  \sum_{t:\omega_t = 2} A^{p^t_1}(x) \otimes \Bigl(V^{p^t_2}(x) - V^{p^{t-1}_2}(x)\Bigr) \quad \text{ for every } x \in X.
\end{aligned}
\]

\end{proposition}
\begin{compositeproof}
	See Appendix~\ref{app:Kronecker}  \Halmos
\end{compositeproof}

We next generalize our construction in Theorem~\ref{them:stair-ineq} to the case where the outer-function is \textit{increasing difference} on the Cartesian product of \textit{partial ordered sets}. Before presenting the details of our generalization, we introduce some basic notation from order theory. A partially order set is a set $X$ on which there is a binary relation $\preceq$  that is reflexive, antisymmetric, and transitive. A partially ordered set $X$ is a chain if for any two elements $x'$ and $x''$ of $X$, either $x' \preceq x''$ or $x'' \preceq x'$. If two elements $x'$ and $x''$ of a partially ordered set $X$ have a least upper bound (resp. greatest lower bound) in $X$, it is their join (resp. meet) and is denoted as $x' \vee x''$ (resp. $x' \wedge x''$). A partial ordered set that contains the join and the meet of each pair of its elements is a lattice.  A bounded lattice is a lattice that additionally has a greatest element and a least element. Clearly, the Cartesian product of finite number of partial ordered sets (resp. lattices) is also a partial ordered set (resp. lattice).  Given two partially ordered sets $X$ and $T$, a function $f: X \times T \to \R$ has \textit{increasing differences} in $(x,t)$ if for $x' \succeq x$, $f(x',t) - f(x,t)$ is monotone non-decreasing in $t$. Given a finite number of partially ordered sets $(X_i)_{i=1}^d$, a function $f: \prod_{i=1}^dX_i \to \R$ has increasing differences in $x=(x_1, \ldots, x_d)$ if, for all distinct $i'$ and $i''$ and for all $x_i \in X_i$ for every $i \notin \{i',i''\}$, $f(x)$ has increasing differences in $(x_{i'}, x_{i''})$. Given a lattice $X$, a function $f: X \to \R$ is said to be supermodular if $f(x') + f(x'') \leq f(x' \vee x'') + f(x' \wedge x'')$ for all $x'$ and $x''$ in $X$. The notion of supermodularity and increasing differences are related. By Theorem 2.6.1 and Corollary 2.6.1 of~\cite{topkis2011supermodularity}, a function has increasing differences on the Cartesian product of a finite collection of chains if and only if it is supermodular on that product. However, this property does not hold for the Cartesian product of general lattices. 

To illustrate this difference in our context, consider 
$\phi:[f^L,f^U] \subseteq \R^d \to \R$ defined as $\phi(f) = g\bigl(h_1(f_{D_1}), \ldots, h_k(f_{D_k}) \bigr)$, where $D_1, \ldots, D_k$ is a partition of $\{1, \ldots, d\}$, $h_i: X_i =\prod_{j \in D_i}[f^L_{j}, f^U_{j}] \subseteq \R^{|D_i|} \to \R$ is a non-decreasing function, and $g: \R^k \to  \R$ is a supermodular function. We show next that $\phi(\cdot)$ satisfies increasing differences on the Cartesian product of partially ordered sets $\prod_{i=1}^kX_i$. However, $\phi(\cdot)$ is not, in general, supermodular on $\prod_{i=1}^k X_i$, \textit{e.g.}, for $x\in \R^m_+$ and $y\in \R^n_+$, the function $\|x\|_{\theta}\|y\|_{\rho}$ 
has increasing differences on $\R^m_+\times \R^n_+$ in that $\|x'\|_{\theta}\|y'\|_{\rho} - \|x'\|_{\theta}\|y''\|_{\rho} \ge \|x''\|_{\theta}\|y'\|_{\rho}-\|x''\|_{\theta}\|y''\|_{\rho}$ for $x'\ge x''\ge 0$ and $y'\ge y''\ge 0$ because $(\|x'\|_{\theta}-\|x''\|_{\theta})(\|y'\|_{\rho}-\|y''\|_{\rho})\ge 0$, but the function is not supermodular over $\R^m_+\times \R^n_+$ when $\rho>1$ and at least one of $m$ or $n$ exceeds one.

\begin{lemma}\label{lemma:non-supermodular}
Given finite partially ordered sets $(X_i)_{i=1}^k$, consider $g \mcirc h: \prod_{i=1}^k X_i \to \R$ defined as $(g \mcirc h)(x_1, \ldots, x_k) =g\bigl(h_1(x_1), \ldots, h_d(x_k)\bigr) $, where $g: \R^k \to \R$ is supermodular and $h_i: X_i \to \R$ is non-decreasing. Then, $(g \mcirc h)(\cdot)$ has increasing differences on $\prod_{i=1}^kX_i$.
\end{lemma}
\begin{compositeproof}
	Let $s$ and $t$ be two distinct indexes, and consider two pairs of elements $(x''_s,x'_s)$ and $(x''_t,x'_t)$ in $X_s$ and $X_t$, respectively, so that $x''_s \succeq_s x'_s$ and $x''_t \succeq_t x'_t$. The proof follows since, for every $x_i \in X_i$ where $i \in I: =\{1, \ldots, k\}\setminus \{ s,t\}$,
	\[
	\begin{aligned}
	&g\bigl( h_s(x''_s), h_{t}(x'_t),h_{I}(x_{I})\bigr) - 	g\bigl( h_s(x'_s), h_{t}(x'_t),h_{I}(x_{I})\bigr) \\
&\qquad \leq		g\bigl( h_s(x''_s), h_{t}(x'_t) \vee h_{t}(x''_t),h_{I}(x_{I})\bigr) - 	g\bigl( h_s(x'_s), h_{t}(x'_t) \vee h_{t}(x''_t),h_{I}(x_{I})\bigr) \\
	&\qquad	= g\bigl( h_s(x''_s),  h_{t}(x''_t),h_{I}(x_{I})\bigr) - 	g\bigl( h_s(x'_s),  h_{t}(x''_t),h_{I}(x_{I})\bigr),
\end{aligned}
	\]
	where the inequality follows from the supermodularity of $g(\cdot)$ on $\R^k$ and the monotonicity of $h_s(\cdot)$, and the equality follows from the monotonicity of $h_t(\cdot)$.   \Halmos
\end{compositeproof}

We note that Theorem~\ref{them:stair-ineq} uses increasing differences of the outer-function over pairs of inner functions $\bigl(f_i(\cdot),f_j(\cdot)\bigr)$. Consequently, the argument extends to the case where each $f_i(\cdot)$ represents a vector of functions and its underestimators are organized to form a partial order. In this case, suitably modifying Theorem~\ref{them:stair-ineq} yields an inequality expressible as a summation of lower-dimensional functions, each involving exactly one inner-function. Overestimating these functions then suffices to obtain a concave overestimator of the composite function.  
Formally, we extend Theorem~\ref{them:stair-ineq} to the case where the outer-function $\phi: \prod_{i=1}^d \Upsilon_i \to \R$ has increasing differences and each $\Upsilon_{i}$ is a partially ordered set with an ordering $\preceq_i$. More specifically, given a collection of pairs of elements $(\upsilon_{ij}, \alpha_{ij} )_{j=0}^n$ in $\Upsilon_i$ so that 
\begin{equation}~\label{eq:order-partial}
 \alpha_{i0} \preceq_i   \cdots \preceq_i  \alpha_{in}\; \;\;  \upsilon_{i0} = \alpha_{i0} \quad \upsilon_{ij}  \preceq_i \upsilon_{in}   \text{ and }  \upsilon_{ij} \preceq_i \alpha_{ij} \;\text{for} \; j \in \{1, \ldots, n\},
\end{equation} 
we overestimate $\phi(\upsilon_{1n}, \ldots, \upsilon_{dn})$ in three steps. First, we show that there exists a lattice subset $(\tilde{\upsilon}_{ij})_{j=0}^n$ of $\Upsilon_i$ that satisfies
\begin{equation}~\label{eq:order-updating}
\begin{aligned}
	\tilde{\upsilon}_{i0} \preceq_i \cdots \preceq_i \tilde{\upsilon}_{in} = \upsilon_{in} \quad \text{and} \quad  \upsilon_{ij} \preceq_{i}\tilde{\upsilon}_{ij}  \preceq_i \alpha_{ij} \text{ for  } j \in \{0, \ldots,n \}.
\end{aligned}
\end{equation}
Second, for each staircase $\omega \in \{0, 1, \ldots, n\}^d$, we telescope $\phi(\upsilon_{1 n}, \ldots, \upsilon_{d n})$ into $\mathcal{D}^\omega(\phi)(\tilde{ \upsilon}_1, \ldots, \tilde{\upsilon}_d)$, where $\mathcal{D}^\omega(\phi)(\cdot)$ is defined in~(\ref{eq:str-exp}). Third, for each difference term in the staircase expansion, we replace the non-changing (resp. changing) coordinates of $\tilde{\upsilon}$ with corresponding coordinates of $\alpha$ (resp. $\upsilon$).

\begin{theorem}~\label{them:stair-lattice}
Consider a function $\phi:\prod_{i=1}^d \Upsilon_i \to \R$ which has increasing differences. Given $(\upsilon_{ij}, \alpha_{ij} )_{j=0}^n$ satisfying~(\ref{eq:order-partial}), if for each $i \in \{1, \ldots, d\}$ there exists $(\tilde{\upsilon})_{j=0}^n$ then, for each staircase $\omega \in \{0, 1, \ldots, n\}^d$, $\phi(\upsilon_{\cdot n})  \leq \BO^\omega(\phi)(\upsilon,\alpha)$.
\end{theorem}
\begin{compositeproof}
	See Appendix~\ref{app:stair-lattice}.  \Halmos
\end{compositeproof}
Theorem~\ref{them:stair-lattice} generalizes Theorem~\ref{them:stair-ineq} to handle functions that are not necessarily supermodular, but have increasing differences, as in Lemma~\ref{lemma:non-supermodular}. We remark that $\alpha_{ij}$ can be any overestimator for $\tilde{\upsilon}_{ij}$ so that the overestimation in \eqref{eq:degreereduce} can be seen as a way to reduce the degree of the resulting inequalities. We provide more detail on this aspect in Example~\ref{ex:RLT+} later.

\subsection{Leveraging Composite Relaxations alongside RLT}\label{section:RLT-Q}
In this subsection, we modify Reformulation-Linearization Technique ($\RLT$) for polynomial programs~\cite{sherali1992global,sherali2013reformulation} so that the relaxation produced by Proposition~\ref{prop:CR}, when the outer-function is multilinear, is implicitly derived by RLT at the $d^{\text{th}}$ level. To achieve this, will require that the $Q$ polytope is modeled explicitly in the formulation in a specific way. Later, we will also apply Theorem~\ref{them:stair-ineq} in a recursive fashion to strengthen RLT relaxations for polynomial optimization problems. We begin with a brief review of RLT relaxations. Consider the feasible region $\mX$ of a polynomial optimization in $m$ variables defined as 
\begin{equation}\label{eq:pp}
\mX:  = \bigl\{x \in \R^m \bigm| g_k(x) \geq 0,\ k = 1,\ldots, \kappa \bigr\}.
\end{equation}
Let $\deg(\mX)$ denote the maximum degree of polynomials $g_1, \ldots, g_\kappa$, and let $\gamma$ be a positive integer. To obtain a generic LP relaxation of $\mX$, the $\RLT$ procedure \textit{reformulates} $\mX$ by generating implied constraints using distinct product forms: 
\[
g_1(x)^{\alpha_1}\cdots g_\kappa(x)^{\alpha_\kappa} \geq 0 \qquad \text{for } \sum_{k = 1}^\kappa \alpha_k \leq \deg(\mX) + \gamma -1,\ \alpha_k \in  \mathbb{Z}_+^\kappa. 
\]
After this, $\RLT$ \textit{expands} the left-hand side of each resulting polynomial inequality so that it becomes a weighted sum of distinct monomials. Last, $\RLT$ \textit{linearizes} the resulting polynomial inequalities by substituting a new variable $y_\alpha$ for each monomial term $x^\alpha:= \prod_{i=1}^n x_i^{\alpha_i} $, so as to obtain linear inequalities in terms of introduced $y$ variables. The resulting LP relaxation is called the $\gamma^{\text{th}}$ level RLT relaxation for $\mX$, and will be denoted as $\RLT_\gamma(\mX)$. 

Let $p:[x^L, x^U] \to \R^{\kappa} $ be a vector of polynomials, where $[x^L, x^U]$ is a hypercube in $\R^m$ so that $x_i^L < x_i^U$ for every $i = 1, \ldots, m$. Clearly,  $p(\cdot)$ can be represented as a vector of composite functions. In particular, let $\theta: \R^d \to \R^\kappa$ be a vector of multilinear functions and $f: \R^m \to \R^d$ be a vector of polynomial functions so that, for $k \in \{1, \ldots, \kappa \}$, $p_k(x) = \theta_k\bigl(f_{1}(x), \ldots, f_d(x) \bigr)$ for   $x \in [x^L,x^U]$. Assume that  $f(\cdot)$ is associated with a pair of vectors of polynomials $\bigl(u(x),a(x)\bigr)$ satisfying~(\ref{eq:ordered-oa}):
\begin{equation}\label{eq:RLT-phi-f}
	\begin{aligned}
 x_i - x^L_i &\geq 0 &  x^U_i - x_i &\geq 0, \quad i = 1, \ldots, d ,  \\
 f_{i}(x) - a_{i0}(x) &\geq 0 & a_{in}(x) - f_{i}(x)  &\geq 0, \quad i = 1, \ldots, d ,\\
f_{i}(x) - u_{ij}(x) &\geq 0 &  a_{ij}(x) - u_{ij}(x) &\geq 0, \quad i = 1, \ldots, d, \; j  = 0, \ldots, n. 
	\end{aligned}
\end{equation}
The $\gamma^{\text{th}}$ level RLT over~(\ref{eq:RLT-phi-f}) yields a relaxation for the graph of $p$ that is tighter than the standard $\gamma^{\text{th}}$ level RLT relaxation of the graph of $p$, since we have introduced additional inequalities relating $f_i(x)$, $a_{ij}$, and $u_{ij}(x)$. It turns out that RLT relaxation, improved by adding the inequalities in \eqref{eq:RLT-phi-f} is still not sufficient to derive the inequalities of Theorem~\ref{them:stair-ineq}. This is because Theorem~\ref{them:stair-ineq} exploits the inequality $f_i(x) \geq \max\{f_i^L,u_{ij}(x)\}$ to derive estimating functions.  However, RLT does not exploit this inequality. Example 3 in~\cite{he2021new} provides a concrete setting where RLT over~(\ref{eq:RLT-phi-f}) does not imply inequalities obtained using Theorem~\ref{them:stair-ineq}.

Now, we apply RLT to the polytope $P$. This relaxation, together with the linearization of $\theta = \theta(u_{1 n}, \ldots, u_{dn})$, yields a linear relaxation for $\Theta^P: = \bigl\{(u,\theta) \bigm| \theta = \theta(u_{1 n}, \ldots, u_{dn}),\ u \in P \bigr\}$, but does not suffice to generate its convex hull even when we utilize the $d^{\text{th}}$ level RLT of $P$, as we illustrate in the next example. 
\begin{example}\label{ex:CR-vs-RLT-2}
Consider a bilinear term $f_1f_2$ over a polytope
$P: = \bigl\{(f,u) \bigm| 0\leq u_{i} \leq \min \{3, f_i\} \text{ and } 0 \leq f_i \leq 4,\ i = 1, 2 \bigr\}$.
Let $\phi$ denote $f_1f_2$. The second level RLT of $P$ fails to generate $\phi \geq 3u_1 +f_1 + 3u_2 + f_2-15$, which is obtained using Corollary~\ref{cor:bilinear-stair}. Moreover, minimizing $\phi- (3u_1 +f_1 + 3u_2 + f_2-15)$ over the second level RLT of $P$ yields a solution smaller than $-0.52$, while it has been shown that $\phi - (3u_1 +f_1 + 3u_2 + f_2-15) \geq 0 $ is valid for the graph of $f_1f_2$ over $P$.  \Halmos
\end{example}
Fortunately, we can adapt RLT to exploit the structure of the product of simplices $Q$ and, thereby, generate the convex hull of $\Theta^Q$, introduced in Proposition~\ref{prop:CR}, at the $d^{\text{th}}$ level. 
In fact, we establish a convexification result in a more general setting. 
\begin{theorem}~\label{them:hull-Lambda}
Let $\Lambda: = \prod_{i=1}^d\Lambda_i$, where $\Lambda_i:=\bigl\{\lambda_i \in \R^{n+1} \bigm| 1- \sum_{j=0}^n \lambda_{ij}=0,\ \lambda_i \geq 0\bigr\}$, and let  $M^\Lambda: = \bigl\{(\lambda, w) \bigm| \lambda \in \Lambda,\ w_j = \prod_{i=1}^d\lambda_{ij_i},\  j \in \{0, \ldots, n \}^d \bigr\}$. Then, $\conv(M^\Lambda) = \proj_{(\lambda,w)}\bigl(\RLT_d(\Lambda)\bigr) = R$, where 
\begin{equation*}\label{eq:hull-lambda}
\begin{aligned}
R:=\Biggl\{ (\lambda, w) \Biggm|	 &w \geq 0,\	\sum_{j \in \{0, \ldots, n \}^d} w_j = 1,\ \\  &  \lambda_{ij_i} = \sum_{p \in \{0, \ldots, n \}^d:p_i =j_i}w_p \; \text{ for }  i \in \{ 1, \ldots, d\}, \;  j \in  \{0, \ldots, n \}^d \Biggr\}. 
\end{aligned}
\end{equation*}
\end{theorem}
\begin{compositeproof}
First, we show  $\conv( M^\Lambda) = R$. Consider a set $M^{\vertex(\Lambda)}: = \bigl\{(\lambda, w)\in M^{\Lambda} \bigm| \lambda \in \vertex(\Lambda)\bigr\}$. Then, we obtain $ \conv(M^\Lambda) = \conv\bigl(M^{\vertex(\Lambda)}\bigr) = R$, where the first equality follows from Corollary 2.7 in~\cite{tawarmalani2010inclusion} and the second equality can be established by  using disjunctive programming~\cite{balas1998disjunctive}. Now, the proof is complete if we show $\proj_{(\lambda, w)}\bigl(\RLT_d(\Lambda)\bigr) \subseteq R$ since the $d^{\text{th}}$ level RLT over $\Lambda$ yields a convex relaxation for $M^\Lambda$. 
Let $D=\{1, \ldots, d\}$ and $E=\{0, \ldots, n\}^d$, and, for $I \subseteq D$ and $j \in E$, define $E(I,j) = \{p \in E \mid p_i = j_i \; \forall i \in I \}$. Then, the $d^{\text{th}}$ level RLT constraints, with $ \prod_{i \in I}\lambda_{ij_i}$ linearized by $y_{(I,j)}$, are $y_{(D,j)}  \geq 0$ and $\sum_{I' \supseteq I, j' \in E(I,j)} (-1)^{|I'\setminus I|}y_{(I',j')} =0$, where $I \subset \{1, \ldots, d\}$ and $j \in E$. For $I \subseteq D$ and $j \in E$, let $w_{(I,p)}$ denote $ \prod_{i \in I}\lambda_{ij_i} \prod_{i \notin I}(1- \sum_{j=0}^n\lambda_{ij})$. Since, $y(I,j)$ and $w(I,p)$ each form a basis of the space of multilinear functions, it follows that there is an invertible transformation relating the two set of variables. This transformation is: $y_{(I,j)} = \sum_{I' \supseteq I, j' \in E(I,j)}w_{(I',j')}$ and $w_{(I,j)} = \sum_{I' \supseteq I, j' \in E(I,j)} (-1)^{|I'\setminus I|}y_{(I',j')}$. Then, expressing $y_{(I,j)}$ in terms of $w_{(I,j)}$ variables, the $d^{\text{th}}$ level RLT relaxation can be represented as 
\[
\begin{aligned}
	&w_{(D,j)} \geq 0 \qquad w_{(I,j)} = 0  && \text{for } I \subset D,\; j \in E, \\
	&y_{(I,j)} = \sum_{I' \supseteq I, j' \in E(I,j)}w_{(I',j')}  && \text{for }I \subseteq D,\; j \in E. 
\end{aligned}
\]
In particular, using the relation $ y_{(D,j)} = w_j$, $y_{(\{i\},j)} = \lambda_{ij_i}$ and $y_{(\emptyset,j)} = 1$, this system implies constraints in $R$. \Halmos
\end{compositeproof}
This result can be used to derive the convex hull of a set of points $M^Q$ satisfying the following multilinear monomial equations over  $Q$,
\[
\biggl\{ (s,m) \biggm|	s \in Q,\ m_{(I,j)} = \prod_{i \in I}s_{ij_i}, \;  I \subseteq\{1, \ldots, d\} \text{ with } |I| \geq 2, j \in \{1, \ldots, n \}^d \biggr\}.
\]
It can be verified that the affine mapping $\mathcal{L}: (\lambda, w) \mapsto (s, m)$ that maps $M^\Lambda$ to $M^Q$ is defined so that for $i \in \{1, \ldots, d \}$ $s_i = \sum_{j = 0}^nv_{ij} \lambda_{ij}$, where $\{v_{10}, \ldots, v_{in}\}$ is the vertex set of $Q_i$ defined in~(\ref{eq:Q-V}) so that $v_{ijk} = a_{ik\wedge j}$, where $k\wedge j$ is used to denote $\min\{k,j\}$. 
Then, for $I \subseteq \{1, \ldots, d\}$ with $|I| \geq 2$ and $j \in \{1, \ldots, n\}^d$,  $m_{(I,j)} = \prod_{i\in I}s_{ij_i}\prod_{i\notin I}\sum_{p =0}^n \lambda_{ip}= \sum_{p \in \{0, \ldots, n\}^d} \bigl( \prod_{i \in I} a_{i  p_i\wedge j_i }\bigr) w_{p}$. The inverse $\mathcal{L}^{-1}$ of $\mathcal{L}$ maps a point $(s,m) \in M^Q$ to a point $(\lambda, w) \in  M^\Lambda$ and is given by $\lambda_{i0} = 1- \frac{s_{i1} - s_{i0}}{a_{i1} - s_{i0}}$, $\lambda_{in} = \frac{s_{in}-s_{in-1}}{a_{in} - a_{in-1}}$ and $\lambda_{ij} = \frac{s_{ij} - s_{ij-1}}{a_{ij} - a_{ij-1}} - \frac{s_{ij+1} - s_{ij}}{a_{ij+1} - a_{ij}}$ otherwise. 
Then, $w_j = \prod_{i = 1}^d \lambda_{ij_i}$ can be written as a linear function of $m_{(I,j)}$ for $I\subseteq\{1,\ldots,d\}$ and $j\in \{1,\ldots,n\}^d$, after substituting $\lambda_{ij_i}$ using their definitions and expanding the resulting multilinear form. Therefore, we obtain that $\conv(M^Q) = \conv(\mathcal{L}(M^\Lambda)) =\mathcal{L}(\conv(M^\Lambda))$, where the second equality holds since affine transformation commutes with convexification. In particular, we obtain an explicit description of the convex hull of $\Theta^Q$. 

\begin{corollary}\label{cor:hull-ml-Q}
Let $\theta: \R^d \to \R^\kappa$	 be a vector of multilinear functions, \textit{i.e.}, for each $k \in \{1, \ldots, \kappa\}$, $\theta_k(s_{1n}, \ldots, s_{dn}) = \sum_{I \in \mathcal{I}_k} c^k_I\prod_{i \in I}s_{in} $, where $\mathcal{I}_k$ is a collection of subsets of $\{1, \ldots, d\}$. Then, the convex hull of $\Theta^Q$ is given by 
\begin{equation*}\label{eq:hull-Q}
\left\{ (s,\theta) \left|\; 
\begin{aligned}
 &\theta_k = \sum_{I \in \mathcal{I}_k}  c^k_I\sum_{j \in \{0, \ldots, n\}^d}\Bigl(\prod_{i \in I} a_{ij_i}\Bigr) w_j,\ w \geq 0 \text{ for } k  \in \{1, \ldots, \kappa \} \\
 & \sum_{j \in \{0, \ldots, n\}^d} w_j = 1,\  s_i = \sum_{p = 0}^{n} v_{ip} \sum_{j \in \{0, \ldots, n\}^d:j_i = p}w_j \text{  for } i \in \{ 1, \ldots, d  \}
 \end{aligned}
\right.
\right\}.
\end{equation*}
 Moreover, the convex hull of $\Theta^P$ is obtained by introducing variables $u$ and the constraints $u\in P$, $u \leq s$, and $u_{\cdot n} = s_{\cdot n}$.
\end{corollary}
\begin{compositeproof}
The second statement follows from the discussion above while the third statement follows from Proposition~\ref{prop:CR}.\Halmos 
\end{compositeproof}

Next, we discuss how RLT relaxations of polynomial optimization problems can be strengthened using insights from Theorems~\ref{them:stair-ineq} and \ref{them:stair-lattice}. Assume that we are given the $\gamma^{\text{th}}$ level RLT relaxation of the feasible region $\mX$ in~(\ref{eq:pp}), we will derive valid inequalities for $\mX$ that involve existing variables in the $\gamma^{\text{th}}$ level RLT relaxation in two steps. First, we use staircase expansions to telescope a composition of polynomials, which defines the feasible region $\mX$, in various forms. Then, we interpret the termwise relaxation step in Theorem~\ref{them:stair-ineq} as a degree reduction procedure. More specifically, we derive over- and under- estimators for the polynomial by replacing difference terms, which appear in staircase expansions, with polynomials of degree up to $\gamma$. As a result, we obtain valid polynomial inequalities of degree up to $\gamma$, which can be linearized using existing variables in the $\gamma^{\text{th}}$ RLT relaxation. 
%
%
%
The next example illustrates how this procedure can strengthen RLT relaxations.  

\begin{example}\label{ex:RLT+}
Consider the set $M_{(2,4)}$ defined by all monomials up to degree $4$ in $2$ variables, that is, $M_{(2,4)}: = \bigl\{(x,m) \bigm| x \in [0,1]^2,\  m_{(\gamma_1, \gamma_2)} = x_1^{\gamma_1} x_2^{\gamma_2},\  2 \leq \gamma_1 + \gamma_2 \leq 4,\ \gamma_i \in \mathbb{N}\bigr\}$. We derive a valid linear inequality for the set $M_{(2,4)}$ without using additional variables and show that those inequalities are not implied by the $6^{\text{th}}$ level RLT relaxation of $x_i \geq 0$ and $1-x_i \geq 0$ where $i =1,2$. 

Consider a polynomial with a degree of $6$ defined as $q(x)= (1 - x_1)^3(1 -x_2)^3$, which will be treated as a composite function with a bilinear outer-function. We begin with two staircase expansions for the polynomial $q(\cdot)$. Then, we derive an overestimator using one expansion while the other is used to under-estimate the function. In particular, estimators are obtained by replacing difference terms of degree greater than $4$ with polynomials of degree up to $4$. As a result, the resulting inequality can be expressed in terms of variables defining the set $M_{(2,4)}$.  For the first staircase expansion, we use functions $\bigl(1,-3x_i+1,(1-x_i)^3\bigr)$ for each $x_i$ in the order they are specified. We remark that this sequence of underestimating functions does not satisfy the inequality system \eqref{eq:ordered-oa}. In particular, $1\not\le (1-x_i)^3$ although $-3x_i+1\le (1-x_i)^3$.
The staircase expansion and degree reduction on $q(\cdot)$ is then as follows:
\begin{equation*}\label{eq:hierarchy-ex-1}
\begin{aligned}
q(x) &= 1 \cdot 1 + 1 \cdot (-3x_2+1 -1) + (-3x_1 + 1 -1) \cdot  (-3x_2 + 1)  \\
& \quad \quad+ (-3x_1 + 1) \cdot (-x_2^3 + 3x_2^2) + (-x_1^3 + 3x_1^2) \cdot (1-x_2)^3 \\
&  	\leq 1-3x_2 + (-3x_1)(-3x_2 + 1) + (-3x_1 + 1) (-x_2^3 + 3x_2^2) + (-x_2^3 + 3x_2^2),
\end{aligned}
\end{equation*}
where the last difference term, the only term with a degree greater than $4$ in the expansion, is relaxed to a degree $3$ polynomial by exploiting the suerpmodularity of the bilinear term and the relation $-3x_1 + 1 \leq (1-x_1)^3$ and $(1-x_2)^3 \leq 1$. On the other hand, since the bilinear term is submodular after switching the second coordinate, using a vector of underestimators of $(1-x_i)^3$ defined as $ u_{i0}(x) = 0$, $u_{i1}(x) = -\frac{3}{4}x_1 + \frac{1}{2}$, and $u_{i2}(x) = (1-x_i)^3 $, and bounds $a_i := (0,0.5,1)$, Corollary~\ref{cor:stair-switching} leads to the following  a degree $4$ underestimator for $q(\cdot)$:
\begin{equation*}\label{eq:hierarchy-ex-2}
\begin{alignedat}{3}
&q(x)&&={}&&0\cdot 1 + \Bigl(-\frac{3}{4}x_1 + \frac{1}{2}\Bigr)\cdot 1 + \Bigl(-\frac{3}{4}x_1 + \frac{1}{2}\Bigr)\Bigl((1-x_2)^3+\frac{3}{4}x_2 -1\Bigr) +{}\\
&&&&& \Bigl((1-x_1)^3+ \frac{3}{4}x_1-\frac{1}{2}\Bigr)\Bigl((1-x_2)^3+\frac{3}{4}x_2\Bigr) + (1-x_1)^3\Bigl(-\frac{3}{4}x_2\Bigr)\\
&&&\ge && \Bigl(-\frac{3}{4}x_1 + \frac{1}{2}\Bigr)\cdot 1 + \Bigl(-\frac{3}{4}x_1 + \frac{1}{2}\Bigr)\Bigl((1-x_2)^3+\frac{3}{4}x_2 -1\Bigr) + \\
&&&&& \Bigl((1-x_1)^3+ \frac{3}{4}x_1-\frac{1}{2}\Bigr)\frac{1}{2} + (1-x_1)^3\Bigl(-\frac{3}{4}x_2\Bigr),
\end{alignedat}
\end{equation*}
where the second last difference is the only term with a degree greater than $4$ in the expansion and is thus relaxed to a degree $3$ polynomial by replacing $(1-x_2)^3+\frac{3}{4}x_2$ with its lower bound $\frac{1}{2}$. Therefore, by Fourier-Motzkin elimination, we obtain  a valid degree-4 inequality for $M_{(2,4)}$, namely,
\[
\begin{aligned}
1 & -3x_2 + (-3x_1)(-3x_2 + 1) + (-3x_1 + 1) (-x_2^3 + 3x_2^2) + (-x_2^3 + 3x_2^2) \geq \\
&\quad (-0.75x_1 + 0.5)\bigl((1-x_2)^3+0.75x_2 \bigr) + \bigl((1-x_1)^3+ 0.75x_1-0.5\bigr)0.5 \\
& \qquad  + (1-x_1)^3(-0.75x_2).
\end{aligned}
\]
Let $l(x,m) \geq 0 $ denote the resulting linear inequality obtained by replacing monomials in the above polynomial inequality with the corresponding $m$ variables. It follows that that $l(x,m) \geq 0$ is not implied by the $6^{\text{th}}$ level RLT relaxation of $[0,1]^2$ since $\min\bigl\{ l(x,m) \bigm| (x,m,w) \in \RLT_6([0,1]^2)\bigr\} = -0.125$.\Halmos

\end{example}

\section{MIP relaxations for composite functions}\label{section:discrete-relaxations}
In this section, we derive \textit{mixed-integer programming (MIP) relaxations} for the hypograph of a composite function $\phi \mcirc f: X \to \R$. For a set $x \in S \subseteq \R^{n_1}$, we say that $E:=R \cap \bigl(\R^{n_1} \times \R^{n_2}  \times \{0,1\}^{n_3} \bigr)$, where $R$ is a set in the space of variables $(x, y, \delta)$, is an MIP relaxation of $S$ if $S\subseteq \proj_{x}(E)$, and we say that $R$ is the \textit{continuous relaxation} of $E$. In particular, we might refer to $E$ as a mixed-integer linear programming (MILP) and mixed-integer convex programming (MICP) relaxation if $R$ is a polyhedron and convex set, respectively.  Moreover, we call $E$  an MIP formulation of $S$ if $S = \proj_{x}(E)$. We will discuss the \textit{quality} of a given MIP relaxation $E$ of $S$ in terms of the size $(|y|, |\delta|)$ as well as the strength of the associated continuous relaxation $R$. We say that an MIP relaxation $E$ is \textit{ideal} if $\proj_\delta(\vertex(R)) \subseteq \{0,1\}^{n_3}$~\cite{vielma2015mixed}. An ideal formulation is desirable since solving its continuous relaxation yields an optimal solution that is integer and, therefore, optimal for the original MIP. Moreover, branching on any binary variables restricts $E$ to a face, and, thus, ideality is retained upon branching.

To construct MIP relaxations for the hypograph of $\phi \mcirc f$, we use a vector $a$ to discretize the hypercube $[f^L,f^U]$, which contains the range of the inner-function $f(\cdot)$ over  $X$. 
Formally, let $a := (a_1, \ldots, a_d) \in \R^{d \times (n+1)}$ so that, for  $i \in \{1, \ldots, d\}$, $f^L_i = a_{i0} < \cdots < a_{in} = f^U_i$. Then, the $i^{\text{th}}$ coordinate is discretized at a subsequence $a_{i\tau(i,0)}, \ldots, a_{i\tau(i,l_i)}$ of $a_{i0}, \ldots, a_{in}$, where $0=\tau(i,0)< \ldots < \tau(i,l_i) = n$. As a result, we obtain a subdivision $\mathcal{H}$ of $[f^L,f^U]$ as a collection of hypercubes, 
\[
\mathcal{H}: = \Biggl\{ \prod_{i = 1}^d\bigl[a_{i\tau(i,t_i-1)},a_{i\tau(i,t_i)}\bigr] \Biggm| t=(t_1, \ldots, t_d ) \in \prod_{i=1}^d \{1, \ldots, l_i  \} \Biggr\}. 
\]
Henceforth, this pair $(a, \mathcal{H})$ will be referred to as a discretization scheme for a composite function $\phi \mcirc f$. Given such a discretization $(a,\mathcal{H})$, an MIP relaxation for the hypograph of $\phi \mcirc f$ can obviously be constructed by outer-approximating the graph of the inner-function $f(\cdot)$ with a polyhedron $W$, and replacing the following disjunctive constraints by its MIP formulation 
\begin{equation}\label{eq:MIP-standard}
	 (f,\phi) \in \bigcup_{H \in \mathcal{H}} \conv\bigl( \hypo(\phi|_H)\bigr),
\end{equation}
where $\phi|_H(\cdot)$ denotes the restriction of $\phi(\cdot)$ on $H$. To derive an MIP formulation for~(\ref{eq:MIP-standard}), we can always treat each $\conv(\phi|_H)$ separately and then use disjunctive programming~\cite{balas1998disjunctive}. Using this approach, the standard formulation for~(\ref{eq:MIP-standard}) from~\cite{balas1998disjunctive,jeroslow1984modelling} requires a binary variable $\delta^H$ and a copy of $(f^H, \phi^H)$ for $\conv(\phi^H)$, thus introducing $|\mathcal{H}|$ Al. Although this formulation of~(\ref{eq:MIP-standard}) is ideal, the formulation and its continuous relaxation are typically intractable since $|\mathcal{H}|$ is exponential in $d$. In Section~\ref{section:MIP-relaxation-incremental}, we provide a new MIP relaxation framework for the hypograph of $\phi \mcirc f$ which seamlessly integrates the incremental model~\cite{markowitz1957solution,dantzig1960significance} into the composite relaxation framework given in~\cite{he2021new}. This scheme exploits underestimating functions associated with the discretization points, yielding an MIP relaxation that is tighter than the above mentioned MIP relaxation obtained using the exponentially sized disjunctive formulation. Moreover, our scheme exploits properties of the outer function $\phi(\cdot)$, and, in particular, convexification results for its hypograph. This has the advantage that when compact convex hull descriptions are available, the resulting MIP relaxations are also compact. For example, under certain conditions, we use the staircase inequalities from Theorem~\ref{them:stair-ineq} to derive an MIP formulation for~(\ref{eq:MIP-standard}) that requires $d(n+1)$ auxiliary continuous variables and $\sum_{i=1}^d(l_i-1)$ additional binary variables, and has a continuous relaxation that is tractable. To the best of our knowledge, for such setting, no ideal MIP formulation was previously known that uses polynomially many continuous variables or has a tractable continuous relaxation.  Section~\ref{section:comparision} will provide geometric insights into the quality of formulations, paving the way for tightening our MIP relaxations. Section~\ref{section:MIP-log} will be devoted to reduce the number of binary variables of our MIP relaxations to $\mathcal{O}(\sum_{i=1}^d \log l_i)$  without increasing the size of continuous variables.

\subsection{Exploiting inner-function structure}
\label{section:MIP-relaxation-incremental}
Let $\Delta_i: = \{ z_i \in \R^{n+1} \mid 1 = z_{i0} \geq z_{i1} \geq \cdots \geq z_{in} \geq 0\}$. Then, a model of selecting subcubes from $\mathcal{H}$ is given as follows:
\begin{subequations}\label{eq:Inc}
\begin{align}
		& \begin{aligned}
			z_i \in \Delta_i,\ &\delta_{it} \in \{0,1\}, \\
			 &z_{i\tau(i,t)} \geq \delta_{it} \geq z_{i\tau(i,t)+1} \qquad  i = 1, \ldots,d ,\ t = 1, \ldots,l_i-1,			
		\end{aligned} 
 \label{eq:Inc-1}  \\
		&	f_i =  a_{i0}z_{i0} + \sum_{j = 1}^{n}(a_{ij} - a_{ij-1})z_{ij} =: F_i(z_i) \qquad i =1, \ldots, d \label{eq:Inc-2}. 
\end{align}	
\end{subequations}
In order to analyze our MIP relaxations, we will use the following lemma that relates the facial structure of $\Delta:= \prod_{i=1}^d \Delta_i $ to the collection of subcubes $\mathcal{H}$ introduced above. 
For $H = \prod_{i=1}^d[a_{i\tau(i,t_i-1)},a_{i\tau(i,t_i)}] \in \mathcal{H}$, we denote by $\Delta_H$ the face of $\Delta$ such that, for $i \in \{1, \ldots,d \}$, $z_{ij} = 1$ for $j \leq \tau(i,t_i-1)$  and $z_{ij} = 0$ for $j> \tau(i,t_i)$. Then, $\{ \Delta_H\}_{H \in \mathcal{H}}$ represents a collection of faces of $\Delta$.
\begin{lemma}\label{lemma:Inc}
The constraints from~(\ref{eq:Inc-1}) yield an MIP formulation for $\{ \Delta_H\}_{H \in \mathcal{H}}$. Moreover, for $H \in \mathcal{H}$, the function $F(z):=\bigl(F_1(z_1), \ldots, F_d(z_d)\bigr)$ maps $\Delta_H$ to $H$.
\end{lemma}
\begin{compositeproof}
See Appendix~\ref{app:Inc}.
\Halmos
\end{compositeproof}
We remark that when $d=1$ and $\{\tau(i,t)\}_{t = 0}^{l_i} = \{0, \ldots, n\}$, the formulation~(\ref{eq:Inc}) is the  \textit{incremental formulation}~\cite{markowitz1957solution,dantzig1960significance} of selecting intervals from $\bigl\{[a_{1j-1}, a_{1j}]\bigr\}_{j = 1}^n$:
\[
f_1 = F_1(z_1),\ z_1 \in \Delta_1,\   \delta_1 \in \{0,1\}^{n-1} ,\ z_{1j} \geq \delta_{1j} \geq z_{1j+1}\quad \text{for } j = 1, \ldots, n-1 .
\]
Henceforth,~(\ref{eq:Inc}) will be referred to as the incremental formulation that selects subcubes from $\mathcal{H}$. Our usage of incremental formulation will be different in two ways. First, we will show in Theorem~\ref{them:MIP-phi-z} that, under certain technical conditions, convexification of $\phi$ over $(f,z)$ naturally yields an ideal formulation for \eqref{eq:MIP-standard}. This is useful since we do not need to consider $\delta$ variables during the convexification. Second, we will relate the $z_{ij}$ variable in the incremental formulation to the slope of line connecting $(a_{ij-1}, s_{ij-1})$ to $(a_{ij},s_{ij})$. Here, $s_{ij}$ is the evaluation, at any $x$, of the best underestimator that can be obtained by taking convex combinations of $u_i(x)$ and is bounded by $a_{ij}$. We use this insight, in Theorem~\ref{them:DCR} to tighten the relaxation by using the inner-function structure via its underestimators.

Next, we provide an ideal formulation for~(\ref{eq:MIP-standard}). Consider a function $\phi \mcirc F: \Delta \to \R$ defined as $(\phi \mcirc F)(z) = \phi\bigl(F_1(z_1), \ldots, F_d(z_d) \bigr)$, where $F_i(\cdot)$ is introduced in the incremental formulation~(\ref{eq:Inc-2}). The main idea of our derivation is to utilize the concave envelope of $\phi \mcirc F$ over $\Delta$. More specifically, we will show that
\begin{equation}\label{eq:MIP-z}
	\Bigl\{(f,\phi,z,\delta) \Bigm|  \phi \leq \conc_\Delta (\phi \mcirc F)(z)  ,\ (f,z,\delta) \in(\ref{eq:Inc}) \Bigr\}
\end{equation}
is an MIP formulation for~(\ref{eq:MIP-standard}). Moreover, if $\phi \mcirc F$ is concave-extendable from $\vertex(\Delta)$ then (\ref{eq:MIP-z}) is ideal. To prove this, we need the following lemma, regarding ideality of MIP formulations.  
\begin{lemma}\label{lemma:ideal}
Let $S$ be a polyhedron in $\R^{n+m}$ such that for $1\le j,k\le m$ and all $(x,y)$ in $\vertex(S) \cup \vertex\bigl(S \cap\{(x,y) \mid y_{j} = y_k\} \bigr)$ we have $y \in \mathbb{Z}^m$. Then, for $\hat{S} = \{(x,y,\delta)\in S \times \R\mid y_{j} \leq \delta\leq y_k \} \neq \emptyset$, we have $  \proj_{(y,\delta)}\bigl(\vertex(\hat{S})\bigr) \subseteq \mathbb{Z}^{m+1}$.
\end{lemma}

\begin{compositeproof}
Consider an extreme point $(\hat{x},\hat{y}, \hat{\delta})$ of $\hat{S}$. It is easy to see that either $\hat{\delta} = \hat{y}_j$ or $\hat{\delta} = \hat{y}_k$. Therefore, it suffices to show that $\hat{y} \in \mathbb{Z}^m$. Without loss of generality, assume that $\hat{\delta} =\hat{y}_j$. There are two cases to consider, either  $\hat{y}_j < \hat{y}_k$ or $\hat{y}_j = \hat{y}_k$. If $\hat{y}_j < \hat{y}_k$ then $(\hat{x}, \hat{y})$ must be an extreme point of $S$ since any constraint that is tight at $(\hat{x},\hat{y})$ is valid for $S$, in which case it follows from the hypothesis that $\hat{y} \in \mathbb{Z}^m$. 
If $\hat{y}_j = \hat{y}_k$, we claim that $(\hat{x}, \hat{y})$ is an extreme point of $S \cap \{ (x,y) \mid y_j = y_k \}$, and, thus, by the hypothesis $\hat{y} \in \mathbb{Z}^m$. If not, $(\hat{x},\hat{y})$ is expressible as a convex combination of two distinct points $( x',y')$ and $(x'',y'')$ of $S \cap\{ (x,y) \mid y_j = y_k \}$. Then, $(\hat{x}, \hat{y}, \hat{\delta})$ is expressible as a convex combination of two distinct points $( x',y',y'_j)$ and $( x'',y'',y''_j)$ of $\hat{S}$, yielding a contradiction. \Halmos
\end{compositeproof}


\begin{theorem}\label{them:MIP-phi-z}
An MIP formulation for~(\ref{eq:MIP-standard}) is given by~(\ref{eq:MIP-z}). Moreover, if $\phi \mcirc F: \Delta \to \R $ is concave-extendable from $\vertex(\Delta)$ then~(\ref{eq:MIP-z}) is an ideal MILP formulation.  
\end{theorem}
\begin{compositeproof}
First, we show that~(\ref{eq:MIP-z}) is an MIP formulation of~(\ref{eq:MIP-standard}).  For a subset $S \subseteq \Delta$, let $E_{S}:= \bigl\{ (f,\phi,z) \bigm| z \in S,\ f = F(z),\ \phi \leq \phi(f) \bigr\}$. By the linearity of $F(\cdot)$ and the definition of $\phi \mcirc F$, it follows that $\conv(E_\Delta) = \bigl\{(f,\phi,z)\bigm| z \in \Delta,\ f = F(z),\ \phi \leq \conc_{\Delta}(\phi \mcirc F )(z)\bigr\}$.
This, together with the first statement in Lemma~\ref{lemma:Inc}, implies that projecting out the binary variables from~(\ref{eq:MIP-z}) yields $\bigl\{ \conv(E_\Delta) \cap (\R^{d+1} \times \Delta_H)\bigr\}_{H \in \mathcal{H} }$. Therefore, the proof is complete since projection commutes with set union and for $H \in \mathcal{H}$
\[
\begin{aligned}
\proj_{(f,\phi)}\bigl( \conv(E_\Delta) \cap (\R^{d+1} \times \Delta_H) \bigr)& = \proj_{(f,\phi)}\bigl(\conv(E_{\Delta_H})\bigr) \\
& = \conv\bigl(\proj_{(f,\phi)}(E_{\Delta_H})\bigr) = \conv\bigl(\hypo(\phi|_H)\bigr),
\end{aligned}
\]
where the first equality holds since $\Delta_H$ is a face of $\Delta$, the second equality holds as the projection commutes with convexification, the last equality holds since, for $S \subseteq \Delta$, $\proj_{(f,\phi)}(E_S) = \hypo(\phi|_{F(S)})$, and, by Lemma~\ref{lemma:Inc}, $F(\Delta_{H}) =H$. 

Next, we prove the second statement. Suppose that $\phi \mcirc F$ is concave-extendable from $\vertex(\Delta)$. It follows readily that $\conc_\Delta(\phi \mcirc F)(\cdot)$ is a polyhedral function, and, thus, the MIP formulation~(\ref{eq:MIP-z}) is linear. Now, we apply Lemma~\ref{lemma:ideal} recursively to show that the formulation is ideal. For $t = (t_1, \ldots, t_d) \in \prod_{i=1}^d\{0, \ldots, l_i-1\}$, let
\[
\begin{aligned}
R^t: = \Bigl\{&(f, \phi,z, \delta) \Bigm|  (f, \phi, z) \in \conv(E_\Delta),\ \delta \in [0,1]^{\sum_{i=1}^d(l_i-1)},  \\
 & \qquad \qquad \qquad z_{i\tau(i,j_i)+1} \leq \delta_{ij_i} \leq z_{i\tau(i,j_i)},\  i = 1, \ldots, d,\ 1 \leq j_i \leq t_i \Bigr\},
\end{aligned} 
\]
and let $y^t: = \bigl((\delta_{11}, \ldots, \delta_{1t_1} ), \ldots, (\delta_{d1}, \ldots, \delta_{dt_d} )\bigr)$. Since  $\phi \mcirc F$ is concave-extendable from $\vertex(\Delta)$, $\proj_{z}\bigl(\vertex(R^0)\bigr) = \vertex(\Delta)$, a set consisting of binary points. Let $\iota \in \{1, \ldots, d\}$ and $t' \in \prod_{i=1}^d\{0, \ldots, l_i-1 \}$ so that $t'_\iota < l_\iota-1$, and assume that the points in $ \proj_{(z,y^{t'})}\bigl(\vertex(R^{t'})\bigr)$ are binary-valued. Since $z_{\iota\tau(\iota,t'_\iota+1)}= z_{\iota\tau(\iota,t'_\iota+1)+1}$ defines a face $F$ of $R^{t'}$, $\vertex(F) \subseteq \vertex(R^{t'})$, and thus, by hypothesis, $\proj_{(z,y^{t'})}\bigl( \vertex(F) \bigr)$ is binary. It follows from Lemma~\ref{lemma:ideal} that $\proj_{(z,y^{t''})}\bigl(\vertex(R^{t''})\bigr)$ is binary, where $t''_i = t'_i$ for $i \neq \iota$ and $t''_i = t'_i+1$ otherwise.  \Halmos
\end{compositeproof}
Notice that the proof of Theorem~\ref{them:MIP-phi-z} can be used to show that $\bigl\{(f,\phi,z,\delta) \bigm|  \phi = \conv\bigl(\text{graph}(\phi \mcirc F)\bigr),\ (f,z,\delta) \in(\ref{eq:Inc}) \bigr\}$ is an ideal formulation of $(f,\phi) \in \bigcup_{H \in \mathcal{H}} \conv\bigl(\text{graph}(\phi|_H)\bigr)$ if $\phi \mcirc F$ is convex-and concave-extendable from  $\vertex(\Delta)$.

We remark that the concave-extendability condition cannot be dropped from the statement of Theorem~\ref{them:MIP-phi-z}. Consider for example a function $\phi(\cdot)$ defined as $\phi(f) = \sqrt{f}$ for $0\leq f \leq 1$ and $\phi(f) = f$ for $1<f \leq 2$, 
Now, with $(a_0,a_1,a_2) = (0,1,2)$, projecting $f$ out from~(\ref{eq:MIP-z}) we get 
\begin{equation}\label{eq:MIP-z-ex}
1 = z_0 \geq z_1 \geq \delta \geq z_2 \geq 0,\ \phi \leq \conc_\Delta(\psi)(z),
\end{equation}
where $\psi(z) = \phi(z_1 + z_2)$. It can be verified that $(\phi,z,\delta) = (\frac{1}{\sqrt{2}},1,\frac{1}{2},0, \frac{1}{2} )$ is extremal.  
It is, however, possible to obtain an ideal formulation by constructing the concave envelope in $(z,\delta)$ space as shown in the next result.

\begin{proposition}\label{prop:zdeltaenv}
Let $\bar{\Delta}: = \prod_{i=1}^d \bar{\Delta}_i$, where $\bar{\Delta}_i: = \bigl\{(z_i,\delta_i) \bigm| z_i \in \Delta_i, z_{i\tau(i,t)} \geq \delta_{it} \geq z_{i\tau(i,t)+1} \text{ for } t \in\{1, \ldots, l_i\}  \bigr\}$, and let $\bar{\Delta}': =\bar{\Delta} \cap \R^{d \times (n+1)} \times \{0,1\}^{\sum_{i=1}^dl_i-1}$. Consider an extension $\bar{F}:\bar{\Delta}' \to \R$ of $F(\cdot)$ defined as  $\bar{F}(z,\delta) = F(z)$ for every $(z,\delta) \in \bar{\Delta}'$. Then, an ideal MIP formulation for~(\ref{eq:MIP-standard}) is given by
\begin{equation}\label{eq:zdeltaformulation}
\Bigl\{(f,\phi,z,\delta) \Bigm|  \phi \leq \conc_{\bar{\Delta}} (\phi \mcirc \bar{F})(z,\delta)  ,\ (f,z,\delta) \in(\ref{eq:Inc}) \Bigr\}. 
\end{equation}
\end{proposition}
\begin{compositeproof}
Clearly, $\bar{\Delta}  = \conv(\bar{\Delta}')$ and $\bar{\Delta}' = \{(z,\delta) \in \bar{\Delta}_H\mid H \in \mathcal{H}\}$, where for each hypercube $H = \prod_{i=1}^d[a_{i\tau(i,t_i-1)},a_{i\tau(i,t_i)}] \in \mathcal{H}$, $\bar{\Delta}_H$ is the face of $\bar{\Delta}$ such that, for $i \in \{1, \ldots,d \}$,  $\delta_{ij} = 1$ for $j < t_i$ and $\delta_{ij} = 0$ for $j \geq t_i$ and  $z_{ij} = 1$ for $j \leq \tau(i,t_i-1)$  and $z_{ij} = 0$ for $j> \tau(i,t_i)$. Then, the validity of this formulation follows from that of~(\ref{eq:MIP-z}) by observing that $(z,\delta)$ satisfies~(\ref{eq:Inc-1}) if only if there exists a hypercube $H \in \mathcal{H}$ so that $(z,\delta) \in \bar{\Delta}_{H}$ and, for any $H \in \mathcal{H}$ and $(z,\delta) \in \bar{\Delta}_{H}$,
\[
\conc_{\bar{\Delta}}( \phi \mcirc \bar{F})(z,\delta) = \conc_{\bar{\Delta}_H}( \phi \mcirc \bar{F})(z,\delta) = \conc_{\Delta_H}( \phi \mcirc F)(z) = \conc_{\Delta}( \phi \mcirc F)(z),
\]
where the first and third equalities hold because $\bar{\Delta}_H$ (resp. $\Delta_H$) is a face of $\bar{\Delta}$ (resp. $\Delta$), and the second equality follows since, given $H$, $\delta$ is fixed and $\bar{F}(z,\delta) = F(z)$.

To show the ideality of \eqref{eq:zdeltaformulation}, consider a vertex $(f,\phi,z,\delta)$ of its LP relaxation, and assume that $\delta$ is not binary. Without loss of generality, assume $\phi =\conc_{\bar{\Delta}} (\phi \mcirc \bar{F})(z,\delta)$.  Notice that $(z,\delta) \notin \bar{\Delta}'$ and thus $(z,\delta, \phi)$ is expressible as a convex combination of distinct points $\bigl(z^k,\delta^k, \phi^k\bigr)$, where $(z^k,\delta^k) \subseteq \Delta'$ and $\phi^k = \conc_{\bar{\Delta}}(\phi\mcirc \bar{F})(z^k,\delta^k)$. Therefore, $(f,\phi,z,\delta)$ can be expressed as a convex combination of distinct points $\bigl(F(z^k),z^k, \delta^k,\phi^k \bigr)$ feasible to \eqref{eq:zdeltaformulation}, a contradiction to extremality of $(f,\phi,z,\delta)$. \Halmos
\end{compositeproof}
As a result, Proposition~\ref{prop:zdeltaenv} yields an ideal formulation for the previous example: 
\[
\phi \leq \sqrt{(1-\delta)(z_1 - \delta)} +z_2,\ 1 = z_0 \geq z_1 \geq \delta \geq z_2 \geq 0,\ \delta \in \{0,1\}, \ f = z_1 +z_2.
\]
It can be verified that $(\phi,z,\delta,f) = (\frac{1}{\sqrt{2}},1,\frac{1}{2},0, \frac{1}{2}, \frac{1}{2})$ is not feasible to the above constraints, certifying that formulation~(\ref{eq:MIP-z-ex}) is not ideal. For the remainder of this paper, we focus on functions $\phi \mcirc F$ that are concave-extendable, and therefore, it will suffice to construct the concave envelope description in $z$-space.  

The discretization points were chosen to be a subset of the grid partition $\prod_{i=1}^d\{a_{i0}, \ldots, a_{in}\}$. The remaining grid points will be exploited using ideas developed in Section~\ref{section:exp}. More specifically, these auxiliary variables $z \in \Delta$ can be related, via an invertible affine transformation, to $s \in Q:=\prod_{i=1}^dQ_i$, and thereby to the underestimators of the inner-functions $f(\cdot)$. More specifically, the affine transformation $Z(s)= \bigl(Z_1(s_1), \ldots, Z_d(s_i)  \bigr)$, where $Z_i : \R^{n_i+1} \to \R^{n_i+1}$ relates $z_i$ to $s_i$ so that $Z_i(s_i) = z_i$, where
\begin{equation}\label{eq:Z_trans}
	\begin{aligned}
	 z_{i0} = 1 \qquad	\text{and} \qquad z_{ij} = \frac{s_{i j} - s_{i j-1}}{a_{ij} - a_{ij-1}} \quad \text{for } j = 1, \ldots , n.
	\end{aligned}
\end{equation}
The inverse of $Z$ is then defined as $Z^{-1}(z):= \bigl(Z_1^{-1}(z_1), \ldots, Z_d^{-1}(z_d) \bigr)$, where $Z_i^{-1}$ recovers $s_i$ given $z_i$ as follows:
\begin{equation}\label{eq:Z_inv}
s_{ij} = a_{i0}z_{i0}  + \sum_{k = 1}^j (a_{ ik} - a_{i k-1})z_{ik} \quad \text{for } j = 0, \ldots, n. 
\end{equation}
Recall that $\vertex(Q_i): = \{v_{i0}, \ldots, v_{in}\}$, where $v_{ij} = (a_{i0}, \ldots, a_{ij-1}, a_{ij} \ldots, a_{ij})$. It is easy to verify that $Z_i(v_{ij}) = \zeta_{ij}$, where $\zeta_{ij} = \sum_{j'=0}^j e_{ij}$ and $e_{ij}$ is the $j^{\text{th}}$ principal vector in the space of variables $(z_{i0}, \ldots,  z_{in} )$. Conversely, $Z_i^{-1}(\zeta_{ij}) = v_{ij}$. More generally, for each $H : = \prod_{i=1}^d[a_{i\tau(i,t_i-1)}, a_{i\tau(i,t_i)}] \in \mathcal{H}$, we obtain that 
\begin{equation*}\label{eq:Q-face}
\begin{aligned}
Z^{-1}(\Delta_{H}) &= Z^{-1}\biggl(  \conv \Bigl( \prod_{i=1}^d  \{\zeta_{i\tau(i,t_i-1)}, \ldots, \zeta_{i\tau(i,t_i)} \} \Bigr)\biggr)\\ 
&= \conv\biggl( Z^{-1}  \Bigl( \prod_{i=1}^d \{\zeta_{i\tau(i,t_{i}-1)}, \ldots, \zeta_{i\tau(i,t_i)} \} \Bigr)\biggr)	\\
& = \conv\biggl( \prod_{i=1}^d \{v_{i \tau(i,t_i-1)}, \ldots, v_{i \tau(i,t_i)} \} \biggr) =:Q_{H},
\end{aligned}
\end{equation*}
where the first equality holds by the definition of $\Delta_H$, the second equality holds because convexification commutes with affine maps, and the third equality holds because $Z_i^{-1}$ maps $\zeta_{ij}$ to $v_{ij}$. Conversely, $Z(Q_H) = \Delta_H$. 
\begin{theorem}\label{them:DCR}
Consider a discretization scheme $(\mathcal{H},a)$, and a vector of convex function $u(\cdot)$ such that the pair $\bigl(u(\cdot),a\bigr)$ satisfies~(\ref{eq:ordered-oa}). An MICP relaxation for the hypograph of $\phi \mcirc f$ is given by:
\begin{equation*}\label{eq:DCR}
\Bigl\{ (x, \phi, s, \delta) \Bigm|
\phi \leq \conc_Q(\ephi)(s), s=Z^{-1}(z), \bigl(z,\delta \bigr) \in (\ref{eq:Inc-1}),  u(x) \leq s, (x,s_{\cdot n}) \in W
\Bigr\}
\end{equation*}
where $\ephi(s) = \phi(s_{1n}, \ldots, s_{dn})$ and $W$ is a convex outer-approximation of $\bigl\{(x,s_{\cdot n}) \bigm| s_{\cdot n} = f(x),\ x\in X\bigr\}$.
\end{theorem}

\begin{compositeproof}
Let $(x, \phi) \in\hypo(\phi \mcirc f )$. To establish the validity of formulation~(\ref{eq:DCR}), it suffices to construct a pair $(s,\delta)$ so that $(x,\phi,s,\delta)$ belongs to the formulation. For all $i$ and $j$, let $s_{ij}: = \min\bigl\{f_i(x), a_{ij} \bigr\}$. Since $s_{in} = f_i(x)$, it follows from the hypothesis regarding $W$ that $(x,s_{\cdot n}) \in W$. Our construction of underestimators requires, via (\ref{eq:ordered-oa}), that $u(x) \leq s$. It remains to show that there exists a $\delta$ so that $\bigl(Z(s),\delta\bigr)\in (\ref{eq:Inc-1})$ and that $\phi \leq \conc_Q(\ephi)(s)$. Choose $t$, and thereby $H$, so that $s_{\cdot n} = \bigl(f_1(x),\ldots,f_d(x)\bigr) \in H = \prod_{i=1}^d [a_{i \tau(i,t_i-1)},a_{i \tau(i,t_i)}]$. We will show that $s\in Q_H$. This directly implies that $Z(s)\in \Delta_H$. Then, by Lemma~\ref{lemma:Inc}, there is a binary vector $\delta$ so that $\bigl(Z(s),\delta\bigr)\in (\ref{eq:Inc-1})$. This also shows that $\phi \leq \phi(s_{1n}, \ldots, s_{dn}) = \ephi(s) \leq \conc_Q(\ephi)(s)$, where the first inequality holds since $(x,\phi) \in \hypo(\phi \mcirc f)$ and $s_{in} = f_i(x)$, and the second inequality holds as $s \in Q$. It only remains to show that $s\in Q_H$. Let  $j'_i$ be such that $s_{in} \in [a_{ij'_i-1},a_{ij'_i}]$. Then, it is easy to see that $\tau(i,t_i-1)<j'_i \leq \tau(i,t_i)$ and $s_i = (a_{i1},\ldots,a_{ij'_i-1},s_{in},\ldots,s_{in})$. We can write $s$ as a convex combination of two vertices, $v_{ij'_i-1}$ and $v_{ij'_i}$, of $Q_H$. More specifically, recall that $v_{ ij'_i-1} = (a_{i 0}, \ldots, a_{ ij'_i-1}, a_{ j'_i-1}, \ldots, a_{ ij'_i-1})$ and $ v_{i j'_i} = (a_{ i0}, \ldots, a_{ ij'_i-1}, a_{i j'_i}, \ldots, a_{i j'_i})$. Then, $s_i = \lambda  v_{ ij'_i-1} + (1-\lambda)  v_{i j'_i}$, where $\lambda =  \frac{ a_{i j'_i} - s_{in}}{a_{i j'_i} - a_{i j'_i-1}}$ and, so, $s \in \prod_{i} \conv\bigl(\{v_{ ij'_i-1}, v_{i j'_i}\}\bigr)\subseteq Q_H$.
\Halmos
\end{compositeproof}
\begin{remark}\label{rmk:hull-Q}
As mentioned above, one advantage of this approach is that we can take advantage of existing descriptions of $\conc_Q(\ephi)(\cdot)$. Here, we summarize a few such cases. Recall that Proposition~\ref{prop:stair-Q} provides an explicit description for $\conc_Q(\ephi)(\cdot)$ in the space of $s$ variables if $\phi(s_{1 n}, \ldots, s_{dn})$ is supermodular over $[f^L,f^U]$ and $\ephi(\cdot)$ is concave-extendable from $\vertex(Q)$. In particular, Corollary~\ref{cor:bilinear-stair} yields both convex and concave envelopes of a bilinear term over $Q$. More generally, Theorem 5 in~\cite{he2021tractable} characterizes a family of bilinear functions for which both convex and concave envelopes over $Q$ can be obtained by convexifying each term in the bilinear function separately. We remark that the number of facet-defining inequalities of the hypograph of $\conc_Q(\ephi)(\cdot)$, as described in Proposition~\ref{prop:stair-Q}, is exponential in $d$ and $n$. Nevertheless, for any point $\s \in Q$, the Algorithm 1 in~\cite{he2021tractable} finds a facet-defining inequality valid for $\hypo\bigl(\conc_Q(\ephi)\bigr)$ and tight at $\s$ in $\mathcal{O}(dn \log d)$ time. 

For general $\phi(\cdot)$, describing $\conc_Q(\ephi)(\cdot)$ in the space of $s$ variables is NP-hard. In such a case, the envelope can often be formulated as the projection of a higher dimensional set, if $\ephi(\cdot)$ is concave-extendable from a collection of convex subsets of $Q$ over which $\ephi(\cdot)$ is concave. More concretely, Corollary~\ref{cor:hull-ml-Q} gives an exponential size formulation for the convex hull of the graph of a multilinear function over $Q$. Although this formulation requires exponentially many variables in $d$,  it is polynomial in size when $d$ is a fixed. 
\Halmos
\end{remark}

Last, we remark that constructions in Theorems~\ref{them:MIP-phi-z} and~\ref{them:DCR} can be generalized to the context of a vector of composite functions $\theta \mcirc f:X   \to \R^\kappa$ defined as $(\theta \mcirc f)(x) = \bigl((\theta_1 \mcirc f)(x), \ldots, (\theta_\kappa \mcirc f)(x)\bigr)$, where $\theta: \R^d \to \R^\kappa$ is defined as $\theta(f) = \bigl( \theta_1(f), \ldots, \theta_\kappa(f)\bigr)$. More precisely, the formulation in Theorem~\ref{them:MIP-phi-z} can be easily generalized to model the following disjunctive constraints: 
\begin{equation}\label{eq:MIP-standard-simu}
	\bigcup_{H \in \mathcal{H}} \conv\bigl( \hypo(\theta|_H)\bigr),
\end{equation}
where $\hypo(\theta|_H): = \bigl\{(f,\theta) \mid \theta \leq \theta(f), f \in H\bigr\}$. 
\begin{proposition}\label{prop:MIP-phi-z-simu}
Let $\theta \mcirc F:\Delta \to \R^\kappa$ be a vector of functions so that $(\theta \mcirc F)(f) = \bigl( (\theta_1\mcirc F)(f), \ldots, (\theta_\kappa \mcirc F)(f)\bigr)$. Then, an extended formulation for~(\ref{eq:MIP-standard-simu}) is given by
\begin{equation}\label{eq:MIP-z-simu}
	\Bigl\{(f,\theta,z,\delta) \Bigm|  (z,\theta) \in  \conv\bigl( \hypo( \theta \mcirc F)\bigr)
 ,\ (f,z,\delta) \in(\ref{eq:Inc}) \Bigr\}.
\end{equation}
Moreover, if $\conv\bigl(\hypo(\theta \mcirc F|_{\vertex(\Delta)})\bigr) = \conv\bigl(\hypo(\theta \mcirc F)\bigr) $ then~(\ref{eq:MIP-z-simu}) is an ideal MILP formulation. 
\end{proposition}
\begin{compositeproof}
This result follows from the proof of Theorem~\ref{them:MIP-phi-z} by letting $E_S:=\{(f,\theta,z,\delta) \mid (z,\theta) \in \hypo(\theta\mcirc F),\ f = F(z),\ z \in S\}$ for $S \subseteq \Delta$.\Halmos 
\end{compositeproof}
In particular, Proposition~\ref{prop:MIP-phi-z-simu} yields an ideal extended formulation when  $\theta\mcirc F$ is either a collection (i) of arbitrary multilinear functions, or (ii) of concave-extendable supermodular functions. It can be easily shown that $\conv\bigl(\hypo(\theta \mcirc F|_{\vertex(\Delta)})\bigr) = \conv\bigl(\hypo(\theta \mcirc F)\bigr)~$\cite{tawarmalani2010inclusion}. In the former case, the formulation is exponential in size (see~Corollary~\ref{cor:hull-ml-Q}). In the latter case, by Corollary 7 in~\cite{he2021tractable}, $\conv\bigl(\hypo(\theta\mcirc F) \bigr) = \cap_{k=1}^\kappa \conv\bigl(\hypo(\theta_k\mcirc F ) \bigr)$, where $\conv\bigl(\hypo(\theta_k\mcirc F) \bigr)$ can be obtained as an affine transformation of $\conv\bigl(\hypo(\theta_k\mcirc F\mcirc Z^{-1}) \bigr)$ using Proposition~\ref{prop:stair-Q}. Then, for
a vector of convex function $u(x)$ such that the pair $\bigl(u(x),a\bigr)$ satisfies~(\ref{eq:ordered-oa}), we obtain an MICP relaxation for the hypograph of $\theta \mcirc f$ given by
\begin{equation*}
\Bigl\{ (x, \theta, s, \delta) \Bigm|
(s, \theta) \in  \conv\bigl(\hypo(\bar{\theta}|_Q)\bigr), \bigl(Z(s),\delta \bigr) \in (\ref{eq:Inc-1}),  u(x) \leq s, (x,s_{\cdot n}) \in W
\Bigr\},
\end{equation*}
where $\bar{\theta}(s) = (\theta\mcirc F \mcirc Z^{-1})(s)$, $\hypo(\bar{\theta}|_{Q}) := \{(s,\theta) \mid \theta \leq \bar{\theta}(s),\ s\in Q\}$, and $W$ outer-approximates $\{(x,s_{\cdot}) \mid s_{\cdot n} = f(x),x \in X\}$. A similar result can be used to relax the graph of $\phi \mcirc f$. 
%
%
\subsection{Geometric insights and strengthening the relaxation}\label{section:comparision}
We start with comparing the strength of relaxations given by Proposition~\ref{prop:CR}, the disjunctive constraint in~(\ref{eq:MIP-standard}), and Theorem~\ref{them:DCR}. Since they are not defined in the same space, we consider three functions $\varphi_{\mathcal{H}}(\cdot)$, $\varphi_{\mathcal{H}-}(\cdot)$, and $\varphi(\cdot)$ defined as follows:
\[
\begin{aligned}
	\varphi(x,s_{\cdot n}) &: = \max \Bigl\{ \conc_Q(\ephi)(s) \Bigm|    u(x) \leq s,\ (x,s_{\cdot n}) \in W \Bigr\},\\
		\varphi_{\mathcal{H}-}(x,s_{\cdot n}) &: = \max\Bigl\{ \conc_Q(\ephi)(s) \Bigm|  \bigl(Z(s), \delta \bigr) \in(\ref{eq:Inc-1}),\ (x,s_{\cdot n}) \in W \Bigr\}, \\
	\varphi_{\mathcal{H}}(x,s_{\cdot n}) &: = \max \Bigl\{ \conc_Q(\ephi)(s) \Bigm|
\bigl(Z(s),\delta \bigr) \in (\ref{eq:Inc-1}) ,\  u(x) \leq s ,\ (x,s_{\cdot n}) \in W
\Bigr\}. 
\end{aligned}
\]
Clearly, for every $(x,s_{\cdot n}) \in W$, $\varphi_{\mathcal{H}}(x,s_{\cdot n}) \leq \varphi_{\mathcal{H}-}(x, s_{\cdot n})$ and $\varphi_{\mathcal{H}}(x,s_{\cdot n}) \leq \varphi(x,s_{\cdot n})$. More specifically, $\varphi_{\mathcal{H}-}(x,s_{\cdot n})$ is derived by using the reformulation \eqref{eq:MIP-z} of \eqref{eq:MIP-standard} and outer-approximating the inner functions using $(x,s_{\cdot n})\in W$. On the other hand, $\varphi_{\mathcal{H}}(x,s_{\cdot n})$ tightens this relaxation by, additionally appending the inequalities $u(x)\le s$ implicit in our construction~\eqref{eq:ordered-oa}.
Next, we show that the three formulations evaluate a non-increasing function at three different points which yields geometric insights into the quality of relaxation. To do so, we introduce a discrete  function $\xi_{i,a_i}$ of $u_i$ defined as:
\begin{equation*}\label{eq:2-d-rep}
	\xi_{i,a_i}(a;u_i) = \left\{ \begin{aligned}
		&u_{ij} && a=a_{ij} \quad \text{for } j \in \{0, \ldots, n \} \\
		&-\infty && \text{otherwise}. 
	\end{aligned}	
	\right.
\end{equation*}
For a given $\bar{x}$, assume that $u_{ij}(\bar{x})$ underestimates $\min\{a_{ij},f_i(x)\}$ and let $s^{u}_i := \bigl(\conc(\xi_{i,a_i})(a_{i0};u_i(\bar{x})), \ldots, \conc(\xi_{i,a_i})(a_{in};u_i(\bar{x})) \bigr)$, where each argument denotes the concave envelope of $\xi_{i,a_i}\bigl(\cdot;u_i(\bar{x})\bigr)$ over $[a_{i0}, a_{in}]$. Since $u_{ij}(\bar{x})$ underestimates $\min\{a_{ij},f_i(x)\}$, the constructed $s^{u}_i$ belongs to $Q_i$ (see Proposition 4 in~\cite{he2021new}). We show next that the three relaxations $\varphi$, $\varphi_{\mathcal{H}-}$, and $\varphi_{\mathcal{H}}$ use different $u(\cdot)$ to construct $s^{u}$ which is then used to evaluate $\conc_Q(\bar{\phi})(s^{u})$. Since $\xi_{i,a_i}$ is non-decreasing with $u_i(\bar{x})$, $s^{u}_i$ non-decreasing with $u_i(\bar{x})$. Moreover, it was shown in \cite{he2021new} that $\conc_Q(\bar{\phi})$ is  a non-increasing function. Then, as the result shows, the $u(\cdot)$ used by $\varphi_{\mathcal{H}}$ is the largest, which reveals why it produces the best bound.
\begin{proposition}\label{prop:eval}
Assume the same setup as in Theorem~\ref{them:DCR}. Let $(\bar{x},\bar{f}) \in W$, and let $\bar{t}:=(\bar{t}_1, \ldots, \bar{t}_d)$ is a vector of indexes so that $ \bar{f} \in \bar{H}:=\prod_{i=1}^d[a_{i \tau(i,\bar{t}_i-1) }, a_{i \tau(i,\bar{t}_i) } ]$. For $i \in \{1, \ldots, d\}$, consider vectors $\u_i$, $\hat{u}_i$ and $u^*_i$ in $\R^{n+1}$ defined as 
\begin{equation*}~\label{eq:updating-underestimation}
\u_{ij} = \begin{cases}
\bar{f}_i & j = n \\
	u_{ij}(\bar{x}) & j <n 
\end{cases},\ \hat{u}_{ij} = \begin{cases}
	a_{ij} & j \leq \tau(i,\bar{t}_i-1) \\
	\bar{f}_i& j \geq \tau(i,\bar{t}_i) \\
   -\infty & \text{otherwise}
\end{cases},\ \text{and } u^*_i =  \u_i \vee \hat{u}_i, 
\end{equation*}
where $\vee$ is the component-wise maximum of two vectors. Then,  $\varphi(\bar{x},\bar{f})$, $\varphi_{\mathcal{H}-}(\bar{x},\bar{f})$, and $\varphi_{\mathcal{H}}(\bar{x},\bar{f})$ equal to $\conc_Q(\ephi)(\s)$, $\conc_Q(\ephi)(\hat{s})$, and $\conc_Q(\ephi)(s^*)$, respectively,
where, for all $i$ and $j$, $\s_{ij} = \conc(\xi_{i,a_i})(a_{ij};\u_i)$, $\hat{s}_{ij}= \conc(\xi_{i,a_i})(a_{ij};\hat{u}_i)$ and $s^*_{ij}= \conc(\xi_{i,a_i})(a_{ij};u_i^*)$. Moreover, $\varphi(\bar{x},\bar{f}) \leq \varphi_{\mathcal{H}}(\bar{x},\bar{f})$ and $\varphi_{\mathcal{H}-}(\bar{x},\bar{f}) \leq \varphi_{\mathcal{H}}(\bar{x},\bar{f})$. 
\end{proposition}

\begin{compositeproof}
See Appendix~\ref{app:eval}. \Halmos
\end{compositeproof}

\begin{example}\label{ex:MIP-ex-1}
Consider the monomial $x_1^2x_2^2$ over $[0,3] \times [0,2]$. Let $\phi(f_1, f_2) = f_1f_2$. Let $a_{1} = (0,5,8,9)$ and $a_2 = (0,4)$, and consider a vector of functions $u:[0,3] \times [0,2] \to \R^{4 + 2}$ defined as $u_1(x) := \bigl(0, 2x_1-1,4x_1-4, x_1^2\bigr)$ and $u_2(x):= (0, x_2^2)$. Here, $\conv_Q(\ephi)(s)$ can be obtained using Proposition~\ref{prop:stair-Q}, that is, $\conv_Q(\ephi)(s) = \max\{0,\ 4s_{11} + 5s_{21} -20,\ 4s_{12} + 8 s_{21} - 32,\ 4s_{13} + 9s_{21} - 36\}$. Thus, we obtain a convex underestimator $\varphi(x): = \conc_Q(\ephi)\bigl(u(x)\bigr)$ for $x_1^2x_2^2$ over $[0,3] \times [0,2]$. We now discretize the range of $x_1^2$ as $[0,5] \cup [5,9]$. Observe that $5$ was one of the points in $a_1$ and the associated underestimator of $x_1^2$ was $2x_1-1$. With this discretization, we can underestimate $x_1^2x_2^2$ over $[0,3] \times [0,2]$ by $\varphi_{\mathcal{H}}(x)$, where
\[
\varphi_{\mathcal{H}}(x) : = \min \left\{ \conv_Q(\ephi)(s) \left| \begin{aligned}
	&	1 \geq \frac{s_{11} - 0}{5-0} \geq \delta \geq \frac{s_{12}-s_{11}}{8-5}  \geq \frac{s_{13}-s_{12}}{9-8} \geq 0  \\
& \delta \in \{0,1\},\ u(x) \leq s ,\ x \in [0,3] \times [0,2] 
\end{aligned}
\right.
\right\}.
\]
It can be verified that $\varphi(2.5,1.5) = \conv_Q(\ephi)\bigl((0,4,6,6.25),(0,2.25)\bigr) = 10$. In contrast, since $u_{13}(2.5) \in [5,9]$, it follows that $s_1 = (0,4,6,6.25)$ is not feasible and  $\varphi_{\mathcal{H}}(2.5,1.5)= \conv_Q(\phi)\bigl((0,5,6,6.25),(0,2.25)\bigr) = 11.25$. In other words, for $x = (2.5,1.5)$ we obtain that $\varphi(x) < \varphi_{\mathcal{H}}(x) < x_1^2x_2^2 =14.0625$. \Halmos
\end{example}

\newcommand{\BoundFunc}[1]{\setsepchar{:}\readlist*\ZZ{#1}{\cal L}_{\ZZ[1]\ZZ[2]}(\ZZ[3])}
\newcommand{\bb}[1]{\setsepchar{:}\readlist*\ZZ{#1}b_{\ZZ[1]\ZZ[2]\ZZ[3]}}
\newcommand{\bd}[1]{\setsepchar{:}\readlist*\ZZ{#1}\mathop{\text{bd}}_{\ZZ[1]\ZZ[2]\ZZ[3]}}
\newcommand{\idx}{k}
\newcommand{\jtoidx}[1]{\setsepchar{:}\readlist*\ZZ{#1}\theta(\ZZ[1],\ZZ[2])}
The remainder of this subsection is focused on deriving Theorem~\ref{them:DCR+} which tightens the relaxation of Theorem~\ref{them:DCR}. Recall that
Theorem~\ref{them:DCR} yields a valid relaxation as long as $(u,a,f)$ satisfy \eqref{eq:ordered-oa}. To tighten the relaxation, we will use, for each underestimator, the best local bounds available when the $\delta$ variable are fixed. Since these bounds can be tighter than the global bounds, $a$, we will show that the resulting relaxation is also tighter than that of Theorem~\ref{them:DCR}. Towards this end, we define $a'_{ij} = \BoundFunc{i:j:\delta_i} = a_{i0} + \sum_{\idx=0}^{l_i} \bb{i:j:\idx}\delta_{i\idx}$. In particular, we require that $b_{ink}=a_{i\tau(i,k+1)}- a_{i\tau(i,k)}$ and $b_{i0k} = a_{i\tau(i,k)} - a_{i\tau(i,k-1)}$. For $0\le j' \le j\le n$, we assume that for all $\idx\in\{1,\ldots,l_i\}$, $\sum_{\idx'=0}^\idx \bb{i:j:\idx'} \ge  \sum_{\idx'=0}^\idx \bb{i:j':\idx'}$. The second condition shows that for $\delta_i$ satisfying $1=\delta_{i0}\ge\cdots \ge\delta_{il_i}\ge \delta_{il_i+1}=0$, we have \begin{equation*}
    \BoundFunc{i:j:\delta_i} - \BoundFunc{i:j':\delta_i} = \sum_{\idx'=0}^\idx  \bigl(\delta_{i\idx'}-\delta_{i\idx'+1}\bigr)\biggl(\sum_{\idx''=0}^{\idx'} \bigl(\bb{i:j:\idx''}-\bb{i:j':\idx''}\bigr)\biggr) \ge 0.
\end{equation*}
It follows that $a'_{ij}\ge a'_{ij-1}$ for all $j\in \{1,\ldots,n\}$. Moreover, we require that $b_{ijk}$ are chosen so that $u_{ij}(x)\le a_{i0} +\sum_{\idx'=0}^{\idx-1} b_{ijk'}$ whenever, $1=\delta_{ik-1}>\delta_{ik}=0$, {\it i.e.}, $f(x)\in (a_{i\tau(i,k-1)},a_{i\tau(i,k)}]$. The setting of Theorem~\ref{them:DCR} is obtained when $\bb{i:j:0}=a_{ij}-a_{i0}$ and $\bb{i:j:k}=0$ for $k\ge 1$. Let $\jtoidx{i:j} = \arg\min\{\idx\mid \tau(i,\idx)\ge j\}$. Now, if $(z_i,\delta_i)$ satisfies~(\ref{eq:Inc-1}) then we can rewrite \eqref{eq:Z_trans} as:
\begin{equation*}
\begin{split}
    s_{ij}-s_{ij-1} &=  z_{ij} \bigl(\BoundFunc{i:j:\delta_i}-\BoundFunc{i:j-1:\delta_i}\bigr)\\
    &=\sum_{\idx=0}^{l_i} (\bb{i:j:\idx}-\bb{i:,j-1:,\idx})\delta_{i\idx}z_{ij} \\
    &= 
    z_{ij}\sum_{\idx=0}^{\jtoidx{i:j}-1} (\bb{i:j:\idx}-\bb{i:,j-1:,\idx}) + \sum_{\idx=\jtoidx{i:j}}^{l_i} (\bb{i:j:\idx}-\bb{i:,j-1:,\idx})\delta_{i\idx},
\end{split}
\end{equation*}
where the last equality holds since $\delta_{i\idx}z_{ij} = z_{ij}$ if $\idx < \jtoidx{i:j}$ and $\delta_{i\idx}$ otherwise. Moreover, as before, we define $s_{i0} = a'_{i0} = a_{i0} + \sum_{\idx=0}^{l_i} b_{i0\idx}\delta_{i\idx}$. Therefore, 
\begin{equation*}
  \begin{split}
  s_{ij} &= a_{i0} + \sum_{\idx=0}^{l_i} \bb{i:0:\idx}\delta_{i\idx} + \sum_{j'=1}^j \biggl\{ z_{ij'}\sum_{\idx=0}^{\jtoidx{i:j'}-1} (\bb{i:j':\idx}-\bb{i:,j'-1:,\idx}) \\
  &\qquad\qquad + \sum_{\idx=\jtoidx{i:j'}}^{l_i} (\bb{i:j':\idx}-\bb{i:,j'-1:,\idx})\delta_{i\idx}\biggr\}\\
  &=a_{i0} + \sum_{j'=1}^j z_{ij'}\sum_{\idx=0}^{\jtoidx{i:j'}-1} (\bb{i:j':\idx}-\bb{i:j'-1:\idx})
  +\sum_{\idx=0}^{l_i}\delta_{i\idx}\bb{i:,\min\{\tau(i,\idx),j\}:,\idx}\\
  &=a_{i0} + \sum_{j'=1}^j z_{ij'}\sum_{\idx=0}^{\jtoidx{i:j'}-1} (\bb{i:j':\idx}-\bb{i:j'-1:\idx})\\
  &\qquad\qquad+\sum_{\idx=0}^{\jtoidx{i:j}-1}\delta_{i\idx}\bb{i:,\tau(i,\idx):,\idx} +\sum_{\idx=\jtoidx{i:j}}^{l_i}\delta_{i\idx}\bb{i:,j:,\idx}:=G_{ij}(z,\delta)
  \end{split}
\end{equation*}
Observe that, in the above expression, the coefficients of $z_{ij'}$ for all $j'$ are non-negative because we have assumed that $\sum_{\idx=0}^{\jtoidx{i:j'}-1} (\bb{i:j':\idx}-\bb{i:j'-1:\idx})\ge 0$.
If we additionally assume that $\bb{i:j:\idx}\ge 0$ for all $i\in \{1,\ldots,d\}$, $j\in \{1,\ldots,n\}$ and $\idx\in \{-1,\ldots,l_i\}$ all the coefficients in the above definition of $s_{ij}$ are non-negative. Let $\delta$ be binary so that for each $i$, there is a $\idx_i$ so that $1 = \delta_{i\idx_i-1} > \delta_{i\idx_i}=0$. Then, it follows that $a'_{ij}=\BoundFunc{i:j:\delta_i} = a_{i0} +\sum_{\idx=0}^{{\idx_i}-1}\bb{i:j:\idx}$ and the above calculation shows that $s_{ij} = a'_{i0} + \sum_{j'=1}^j z_{ij'}(a'_{ij'}-a'_{ij'-1})$. 


\begin{theorem}~\label{them:DCR+}
 Assume the same setup as in Theorem~\ref{them:DCR}. Then, an MICP relaxation for the hypograph of $\phi \mcirc f$ is given by
\begin{equation}\label{eq:DCR+}
\left\{ (x, \phi, s, z, \delta) \left| \;
\begin{aligned}
&\phi \leq \conc_{\bar{\Delta}}(\ephi \mcirc G )(z,\delta),\ s= G(z,\delta),\ (z,\delta ) \in (\ref{eq:Inc-1}) \\
&u(x) \leq s,\ (x,s_{\cdot n}) \in W	
\end{aligned}
\right.
\right\}.
\end{equation}
Let $\phi_{\mathcal{H}_+}(x,s_{\cdot n}): = \max\bigl\{ \phi \bigm| (x,\phi,s,z,\delta) \in(\ref{eq:DCR+}) \bigr\}$. For $(\bar{x},\bar{f}) \in W$, $\phi_{{\cal H}_+}(\bar{x},\bar{f})=\conc_{Q'}(\bar{\phi})(s')$ and $\phi_{{\cal H}_+}(\bar{x},\bar{f})\leq \phi_{{\cal H}}(\bar{x},\bar{f})$, where $Q'$ is formed using $a'$ and $s'_{ij} = \conc(\xi_{i,a'_i})(a'_{ij};\tilde{u}_i)$, where $\tilde{u}_i$ is the same as $u^*_i$ defined in Proposition~\ref{prop:eval} except that $a_{i}$ is replaced with $a'_{i}$. 
\end{theorem}
\begin{compositeproof}
See Appendix~\ref{app:DCR+}. \Halmos
\end{compositeproof}
We now specify a particular choice of $\bb{i:j:k}$ that satisfies the requirements. As before, to have the same breakpoints as in Theorem~\ref{them:DCR}, recall that $\bb{i:n:\idx} := a_{i\tau(i,\idx+1)}-a_{i\tau(i,\idx)}$ where $\idx\in \{0,\ldots,l_i\}$. With this notation, $\delta_{ik-1}-\delta_{ik}$ can be interpreted as an indicator of whether $f_{i}\in (a_{i\tau(i,\idx-1)},a_{i\tau(i,\idx)}]$ or not. Since $u_{ij}(x)\le a'_{ij}$, we must choose $\bb{i:j:k}$ so that they satisfy $u_{ij}(x)\le a_{i0}+\sum_{\idx'=0}^{\idx-1} \bb{i:j:\idx'}$ whenever $f_{i}\in (a_{i\tau(i,\idx-1)},a_{i\tau(i,\idx)}]$. Let $\bd{i:j:\idx} := \inf\bigl\{u\bigm|  u\ge u_{ij'}(x)\forall j'\le j, u\ge a_{i\tau(i,\idx)}, f(x)\in (a_{i\tau(i,\idx)},a_{i\tau(i,\idx+1)}]\bigr\}$ and $\bb{i:j:\idx}=\bd{i:j:\idx}-\bd{i:j:\idx-1}$. Observe that the definition of $\bb{i:n:\idx}$  given above is consistent with this definition. Since, for $j>j'$, $\bd{i:j:\idx}\ge \bd{i:j':\idx}$, it follows that $\sum_{\idx'=0}^{\idx}\bb{i:j:\idx'} = \bd{i:j:\idx}-a_{i0}\ge \bd{i:j':\idx}-a_{i0} = \sum_{\idx'=0}^{\idx}\bb{i:j':\idx'}$. If we further require that $\bb{i:j:\idx}$ is non-negative, we may instead define $\bb{i:j:\idx}=\max_{\idx'\le \idx}\bd{i:j:\idx'} - \max_{\idx'\le \idx-1}\bd{i:j:\idx'}$.

\begin{remark}
Theorem~\ref{them:DCR+} requires the concave envelope of $\ephi\mcirc G$ over $\bar{\Delta}$. In this remark, we show that this envelope is readily available under certain conditions. Assume $\phi(\cdot)$ is supermodular and $\bar{\phi}\circ G$ is concave extendable from $\vertex(\bar{\Delta})$. We provide an explicit description of the concave envelope  $\ephi \mcirc G$ over $\bar{\Delta}$ for this case.  First, observe that $\vertex(\bar{\Delta}) = \prod_{i=1}^d \vertex(\bar{\Delta}_i)$, and $\vertex(\bar{\Delta}_i)$ forms a chain with joint (resp. meet) defined as component-wise maximum (resp. minimum). Moreover, if $\phi(\cdot)$ is supermodular over $[f^L,f^U]$ then $\ephi \mcirc G$ is supermodular when restricted to $\vertex(\bar{\Delta})$, \textit{i.e.}, for $y':=(z',\delta')$ and $y'':=(z'',\delta'')$ in $\vertex(\bar{\Delta})$, 
\[
\begin{aligned}
(\ephi \mcirc G)(y' \vee y'')  &+(\ephi \mcirc G)(y' \wedge y'')  \\
&=  \phi \bigl( G_{1n}(y'_1) \vee G_{1n}(y''_1), \ldots,G_{dn}(y'_d) \vee G_{dn}(y''_d) \bigr) + \\
& \qquad \qquad \phi \bigl( G_{1n}(y'_1) \wedge G_{1n}(y''_1), \ldots,G_{dn}(y'_d) \wedge G_{dn}(y''_d) \bigr) \\
&\geq \phi \bigl(G_{\cdot n}(y') \bigr) + \phi \bigl(G_{\cdot n}(y'') \bigr) = (\ephi \mcirc G)(y') + (\ephi \mcirc G)(y''),
\end{aligned}
\]
where the first equality holds since $G_{in}$ is non-decreasing and $\vertex(\bar{\Delta}_i)$ forms a chain, the inequality holds by supermodularity of $\phi(\cdot)$, and the last equality holds by definition. Therefore, if $ \ephi \mcirc G$ is concave-extendable from $\vertex(\bar{\Delta})$ and $\phi(\cdot)$ is supermodular on $[f^L,f^U]$, the concave envelope of $\ephi \mcirc G$ over $\bar{\Delta}$ is explicitly described by Proposition 3 in~\cite{he2021tractable}--a consequence of Corollary 3.4 in~\cite{tawarmalani2013explicit}. \Halmos
\end{remark}

\subsection{Small MIP formulations for a vector of functions}\label{section:MIP-log}

We have shown that the number of continuous variables in our formulation can be significantly smaller than the disjunctive prsformogramming counterparts. Here, we show that our formulations can continue to exploit inner function structure while also leveraging recent work on reducing the number of binary variables.
In particular, significant work has been done in the literature on modeling functions that are expressible as a disjunction of linear pieces, using binary variables that are lograthmic in the number of disjunctions (see, for example, ~\cite{ibaraki1976integer,vielma2010mixed,vielma2011modeling,misener2011apogee,misener2012global,gupte2013solving,huchette2019combinatorial,nagarajan2019adaptive}). More precisely, for a $d$-dimensional multilinear function, with $n$ discretization points along each dimension, the number of binary variables in these formulations grows at the rate of $\mathcal{O}\bigl(d\log(n)\bigr)$. Here, we show that  our formulations can be readily adapted so as to require the same number of binary variables. More importantly, a key insight is that, in certain cases, we  require  $\mathcal{O}(dn)$ continuous variables, while the previously mentioned  schemes required $\mathcal{O}(n^d)$ continuous variables. Although we focus on new formulations with fewer continuous variables, we remark that both formulations (existing and new ones with $\mathcal{O}(n^d)$ and $\mathcal{O}(dn)$ continuous variables, respectively) can utilize Theorem~\ref{them:DCR} to improve the quality of their relaxation by exploiting the structure of the inner functions via underestimating functions. 

Instead of using the incremental formulation, we select an interval using SOS2 constraints~\cite{beale1970special} on continuous variables $\lambda \in \Lambda_i:= \bigl\{\lambda \geq 0 \bigm| \sum_{j=0}^n \lambda_{ij}=1\bigr\}$, which allows at most two variables, $\lambda_{ij}$ and $\lambda_{ij'}$, to be non-zero and requires that $j$ and $j'$ are adjacent. For a Gray code~\cite{savage1997survey} or a sequence of distinct binary vectors $\{\eta^t\}_{t=1}^n \subseteq \{0,1\}^{\log_2(n)}$, where each adjacent pair $(\eta^{(t)}, \eta^{(t+1)})$ differ in at most one component,~\cite{vielma2011modeling} proposes the following ideal MILP formulation for $\lambda_i\in \Lambda_i$ satisfying SOS2 constraints, using binary variables, $\delta$, that are logarithmically many in the number of continuous variables $\lambda_i$:
\begin{equation}\label{eq:SOS2-log}
\begin{aligned}
\lambda_i \in \Lambda_i,\ \delta_{i} \in \{0,1\}^{\lceil \log_2(n) \rceil},\ &\sum_{j \notin  L_{ik}} \lambda_{ij} \leq \delta_{ik} \leq 1-  \sum_{j \notin R_{ik}} \lambda_{ij},\ \\
&  \qquad \qquad \qquad    \text{ for } k =1, \ldots, \lceil\log_2(n) \rceil,
\end{aligned}
\end{equation}
where $\eta^{(0)}:= \eta^{(1)}$ and $\eta^{(n+1)}: = \eta^{n}$, $L_{ik}: = \bigl\{j \in \{0, \ldots, n\} \bigm| \eta^{(j)}_k = 1 \text{ or } \eta^{(j+1)}_k = 1  \bigr\}$, and $R_{ik}: = \bigl\{j \in \{0, \ldots, n\} \bigm| \eta^{(j)}_k = 0 \text{ or } \eta^{(j+1)}_k = 0  \bigr\}$. Let $\Lambda: = \prod_{i=1}^d \Lambda_i$, and let $A: \Lambda \to \R^d$ be a function defined as $A(\lambda): = \bigl(A_1(\lambda), \ldots, A_d(\lambda_d)\bigr)$, where $A_i(\lambda_i) := \sum_{j = 0}^n a_{ij} \lambda_{ij}$. As in Proposition~\ref{prop:MIP-phi-z-simu}, we  obtain an ideal MILP formulation for~(\ref{eq:MIP-standard-simu}) by convexifying the hypograph of a vector of composite functions $(\theta \mcirc A): \Lambda \to \R^k$, which is defined as $ (\theta \mcirc A) (\lambda) = \bigl((\theta_1 \mcirc A)(\lambda), \ldots, (\theta_k \mcirc A)(\lambda)\bigr)$. 
\begin{proposition}\label{prop:MIP-phi-lambda-log}
Assume that $\conv\bigl(\hypo(\theta \mcirc A)\bigr) =\conv\bigl( \hypo(\theta \mcirc A|_{\vertex(\Lambda)})\bigr)$. Then, an ideal MIP formulation for~(\ref{eq:MIP-standard-simu}) is given by 
\begin{equation*}\label{eq:MIP-phi-z}
\Bigl\{(f,\theta,\lambda,\delta) \Bigm|(\lambda,\theta) \in \conv\bigl( \hypo(\theta \mcirc A) \bigr),\ f = A(\lambda),\ (\lambda_i,\delta_i) \in(\ref{eq:SOS2-log}),\; i  = 1, \ldots,d  \Bigr\}. 	
\end{equation*}
\end{proposition} 
\begin{compositeproof}
See Appendix~\ref{app:MIP-phi-lambda-log}. \Halmos
\end{compositeproof}

Clearly, the formulation in Proposition~\ref{prop:MIP-phi-lambda-log} requires $d (n+1)$ additional continuous variables and $d \lceil \log_2(n) \rceil $ binary variables, provided that an explicit description of $\conv\bigl(\hypo(\theta \mcirc A) \bigr)$ is available in the space of $(\theta,\lambda)$ variables. Here, we remark that $\conv\bigl(\hypo(\theta \mcirc A) \bigr)$ can be described using existing descriptions of $\conc_Q(\ephi)(\cdot)$ in Remark~\ref{rmk:hull-Q}. Consider an invertible affine mapping $V = Z^{-1} \mcirc T^{-1}$ and $V^{-1} = T\mcirc Z$, where $T$ is the transformation introduced in the proof of Proposition~\ref{prop:MIP-phi-lambda-log}. Here, $V$ (resp. $V^{-1}$) maps $\Lambda$ to $Q$ (resp. $Q$ to $\Lambda$). Now, observe that, for every $s \in Q$, $(\theta \mcirc A)\bigl(V^{-1}(s)\bigr) = \theta(s_{1n}, \ldots, s_{dn})$. Since convexification commutes with affine transformation, it follows readily that $\conv\bigl(\hypo(\theta \mcirc A)\bigr) = \bigl\{(\lambda, \theta) \bigm| \bigl(V(\lambda),\theta\bigr) \in \conv\bigl(\hypo(\theta'|_Q)\bigr) \bigr\}$,
where $\theta'(s) = \theta(s_{1n}, \ldots, s_{dn})$. In particular, $\conc_\Lambda(\phi \mcirc A)(\lambda) = \conc_Q(\ephi)\bigl(V(\lambda)\bigr)$. By Proposition~\ref{prop:stair-Q}, if $\phi(\cdot)$ is supermodular over $[f^L,f^U]$ and $\ephi(\cdot)$ is concave extendable from $\vertex(Q)$ then we obtain that
\begin{equation}\label{eq:envelope-lambda}
\begin{aligned}
\conc_{\Lambda}(\phi \mcirc A)(\lambda)&: =  \min_{\omega \in \Omega}\Biggl\{ \phi\bigl(\Pi(a;p^0) \bigr)  + \sum_{i=1}^d\sum_{t:\omega(t)=i} \Bigr[ \phi\bigl(\Pi(a;p^t)\bigr) \\ 
& \qquad  - \phi\bigl(\Pi(a;p^{t-1})\bigr)\Bigr](\lambda_{ip^t_i}+ \cdots + \lambda_{in} ) \Biggr\},
\end{aligned}
\end{equation}
where $\Omega$ the set of all direction vectors in $\{0, \ldots, n \}^d$. 
\begin{corollary}\label{cor:MIP-sup-simu}
Assume that, for every $k \in \{1, \ldots, \kappa \}$, function $\theta_k(s_{1n}, \ldots, s_{dn})$ is supermodular over $[f^L,f^U]$ and is concave-extendable from $\vertex(Q)$. Then, an ideal MILP formulation for~(\ref{eq:MIP-standard-simu}) is given by 
\[
\begin{aligned}\Bigl\{(f,\theta,\lambda,\delta) \Bigm| 
	 \theta_k &\leq \conc_\Lambda(\theta_k \mcirc A)(\lambda),\; k=1, \ldots,\kappa, \\
	 & \qquad \qquad\qquad f = A(\lambda),\ (\lambda_i, \delta_i) \in (\ref{eq:SOS2-log}),\;  i=1, \ldots,d\Bigr\},
\end{aligned}
\]
where, for each $k$, $\conc_{\Lambda}(\theta_k \mcirc A)(\cdot)$ can be described by~(\ref{eq:envelope-lambda}).
\end{corollary}
\begin{compositeproof}
By Corollary 7 of~\cite{he2021tractable}, $\conv\bigl(\hypo( \theta'|_Q)\bigr) = \cap_{k=1}^\kappa \conv\bigl(\hypo(\theta'_{k}|_Q)\bigr)$. Thus, this result follows from  Proposition~\ref{prop:MIP-phi-lambda-log} and the invertible affine transformation $V$ defined above. \Halmos
\end{compositeproof}

Next, we specialize our result to the bilinear case. Discretization of the bilinear term is a common technique used to MINLPs. By recursively applying this technique, factorable MINLPs can be approximated arbitrarily closely~\cite{misener2011apogee,misener2012global,nagarajan2019adaptive}. In other words, the following result is useful in this context.
\begin{corollary}\label{cor:MIP-phi-bilinear}
let $\phi(f_1, f_2):= f_1f_2$, and let $\check{b}$ and $\hat{b}$ be functions defined as in~(\ref{eq:bilinear-stair-1}) and~(\ref{eq:bilinear-stair-2}), respectively. Then, an ideal  formulation for $\{ \conv(\graph(\phi|_H))\}_{H \in \mathcal{H}}$ is  
\[
\Bigl\{(f,\phi,\lambda,\delta) \Bigm| \check{b}\bigl(V(\lambda)\bigr) \leq	\phi \leq \hat{b}\bigl(V(\lambda)\bigr),\ f  = A(\lambda),\ (\lambda_i, \delta_i) \in (\ref{eq:SOS2-log}), \;  i = 1,2\Bigr\}.
\]
\end{corollary}
\begin{compositeproof}
	This result follows from Corollary~\ref{cor:bilinear-stair} and Proposition~\ref{prop:MIP-phi-lambda-log}. \Halmos
\end{compositeproof}
Observe that if there are $n-1$ discretization points along each axis, formulations from Corollary~\ref{cor:MIP-phi-bilinear} require $2(n+1)$ continuous variables and $2 \lceil \log n \rceil$ (resp. $2n$) binary variables if logarithmic (resp. incremental) formulation is used. 
This is in contrast to existing ideal formulations from~\cite{huchette2019combinatorial} which require $(n+1)^2$ continuous variables. The following example illustrates this difference. We remark that the formulation in Corollary~\ref{cor:MIP-phi-bilinear} yields tighter relaxations than those in \cite{misener2011apogee,misener2012global,nagarajan2019adaptive}  because of the presence of constraints $u(x)\le s$ (see Proposition~\ref{prop:eval}).

\begin{example}
Consider a bilinear term $f_1f_2$ over $[0,4]^2$ and a partition of the domain $\mathcal{H}:=\bigl\{(f_1,f_2) \bigm| f_i \in [0,3] \cup [3,4],\; i = 1, 2 \bigr\}$.  Let $(a_{i0}, a_{i1}, a_{i2}) = (0,3,4)$. Here, the number of discretization points, $n-1$, equals $1$. The formulation from~\cite{huchette2019combinatorial} or Corollary~\ref{cor:MIP-phi-mul-log} yields the following formulation that introduces $9$ ({\it{i.e.}}, $(n+1)^2$) continuous variables to convexify $\bigl\{(f_1, f_2, f_1f_2) \bigm| f_i \in \{a_{i0},a_{i1},a_{i2}\},\; i =1,2 \bigr\}$
 \[
\begin{aligned}
&(f_1,f_2,\phi) = \sum_{j=0}^2 \sum_{k = 0}^2 w_{jk}(a_{1j}, a_{2k}, a_{1j}a_{2k}),\ \sum_{j = 0}^2 \sum_{k=0}^2 w_{jk}  = 1,\ w \geq 0, \\
	&\sum_{j=0}^2w_{j0} \leq \delta_1 ,\ \sum_{j=0}^2w_{j2} \leq 1-\delta_1,\ \delta_1 \in \{0,1\}, \\
	&\sum_{k=0}^2w_{0k} \leq \delta_2 ,\ \sum_{k=0}^2w_{2k} \leq 1-\delta_2,\ \delta_2 \in \{0,1\}.
\end{aligned}
\]
In contrast, the formulation from Corollary~\ref{cor:MIP-phi-bilinear} introduces $6$ ({\it{i.e.}}, $2(n+1)$) continuous variables $\lambda$  and yields: 
\[
\begin{aligned}
&\max\left\{\!
\begin{aligned}
&12\lambda_{11}+16\lambda_{12} + 12\lambda_{21} 
+16\lambda_{22} - 16\\
& 9 \lambda_{11} + 9\lambda_{12} + 9 \lambda_{21} + 9\lambda_{22} - 9 \\
&12\lambda_{11} + 12\lambda_{12} + 9 \lambda_{21} + 12\lambda_{22} - 12 \\ 
&9\lambda_{11} +  12 \lambda_{12} + 12 \lambda_{21} + 12 \lambda_{22} - 12 \\ 
& 12 \lambda_{11} + 15\lambda_{12}  + 12 \lambda_{21} 
+ 15\lambda_{22} -15\\
&0 
\end{aligned}
\right\} \leq \phi \leq \min\left\{\!
\begin{aligned}
&12 \lambda_{21} + 16\lambda_{22}\\
& 3\lambda_{12} + 9 \lambda_{21} +13\lambda_{22} \\
&  4\lambda_{12} + 9 \lambda_{21} + 9\lambda_{22}\\
&12 \lambda_{11} + 16\lambda_{12}\\
& 9 \lambda_{11} +13\lambda_{12} + 3\lambda_{22} \\
&  9 \lambda_{11} + 9\lambda_{12}  4\lambda_{22}
\end{aligned}
\right\} \\
	& \lambda_i \in \Lambda_i,\ \lambda_{i0} \leq \delta_i ,\ \lambda_{i2} \leq 1-\delta_i,\ \delta_i \in \{0,1\} \quad \text{for } i = 1,2 ,
\end{aligned}
\]
where the $12$ inequalities are derived using Corollary~\ref{cor:bilinear-stair} and  the affine map $V$. 
Although both formulations model SOS2 constraints in the same way, we highlight that earlier discretization schemes were obtained by using disjunctive programming while our construction relies on convex hull construction over $Q$. When the convex hull is constructed using disjunctive programming, our scheme recovers the earlier models. However, the new formulation can directly use convex hull formulations when they are available in a lower-dimensional setting as in Corollary~\ref{cor:MIP-phi-bilinear}. In this case, our formulation requires significantly fewer variables in comparison to disjunctive programming techniques. In fact, our formulation can be seen as an affine transformation of the earlier formulation.

Moreover, if $f_1=x_1^2$, we have shown in Theorem~\ref{them:DCR} that we may additionally require that $3\lambda_{11}\ge \max\{2x_1-1, \frac{3}{4}x_1^2\}$ because $2x-1$ and $\frac{3}{4}x_1^2$ are convex underestimators of $x_1^2$ bounded by $3$ and the transformation $V$ shows that $s_{11}=3\lambda_{11}$. One of the key advantages of Theorem~\ref{them:DCR} is that it provides a mechanism for the discretized formulation to exploit the structure of the inner functions $f_i$.\Halmos
\end{example}

For the general case when $\theta(\cdot)$ is a  vector of multilinear functions, describing the convex hull of the graph of $\theta \mcirc A$ in the space of $(\lambda, \theta)$ variables is NP-hard. In this case, additional variables are introduced to describe the convex hull. 
\begin{corollary}\label{cor:MIP-phi-mul-log}
Let $\theta: \R^d \to \R^\kappa$	 be a vector of multilinear functions, \textit{i.e.}, for  $k \in \{1, \ldots, \kappa\}$, $\theta_k(s_{1n}, \ldots, s_{dn}) = \sum_{I \in \mathcal{I}_k} c^k_I\prod_{i \in I}s_{in} $, where $\mathcal{I}_k$ is a collection of subsets of $\{1, \ldots, d\}$. Then, an ideal MILP formulation for $\{\conv(\graph(\theta|_H))\}_{H \in \mathcal{H}}$ is given by 
\begin{subequations}\label{eq:MIP-mul-log}
\begin{align}
	 	&\theta_k = \sum_{I \in \mathcal{I}_k}  c^k_I\biggl(\sum_{e \in E}\Bigl(\prod_{i \in I} a_{ie_i}\Bigr) w_e \biggr),& &k = 1, \ldots, \kappa, \label{eq:MIP-mul-log-1} \\
	& w \geq 0,\ \sum_{e \in E}w_e = 1,\ 
		 \lambda_{ij} = \sum_{e \in E:e_i = j}w_e,& & i =1, \ldots, d ,\; j  =0, \ldots, n, \label{eq:MIP-mul-log-2} \\
 &f = A(\lambda),\ (\lambda_i, \delta_i) \in(\ref{eq:SOS2-log}),& &  i  = 1, \ldots, d.  \label{eq:MIP-mul-log-3}
\end{align}
\end{subequations}

\end{corollary}
\begin{compositeproof}
By Corollary~\ref{cor:hull-ml-Q} and the invertible affine mapping $V$ between $Q$ and $\Lambda$, constraints in~(\ref{eq:MIP-mul-log-1}) and~(\ref{eq:MIP-mul-log-2}) describe the convex hull of the graph of $\theta \mcirc A$. Therefore, this result follows from Proposition~\ref{prop:MIP-phi-lambda-log}.  \Halmos
\end{compositeproof}
Note that when $\theta(\cdot)$ is a multilinear monomial, \textit{i.e.}, $\theta(f) = \prod_{i=1}^df_i$, formulation~(\ref{eq:MIP-mul-log}) reduces to the formulation from Corollary 3 in~\cite{huchette2019combinatorial}. Instead of utilizing the convex hull the graph of $\theta \mcirc A$,~\cite{huchette2019combinatorial} constructs the formulation using independent branching constraints. 


\section{Conclusion}
In this paper, we showed how to exploit inner-function structure to derive improved continuous and discrete relaxations of composite functions. These relaxations were obtained using staircase expansions, which are relaxed to give inequalities for the hypograph of the composite function. The inequalities exploit knowledge of underestimators and their bounds for each inner-function and assume that the outer-function is supermodular. The paper made several advances. First, staircase inequalities led to an alternative and much simpler derivation of cuts obtained in~\cite{he2021new,he2021tractable}. Second, this new interpretation led to relaxations for a wider class of functions than those treated in prior literature~\cite{tawarmalani2013explicit,he2021tractable} and to tighter relaxations when prior techniques are applicable.
Third, the staircase expansion allowed for generalizing our results to the case where inner-function underestimators form a partial order and the outer-function has increasing differences thereby allowing the derivation of new inequalities that tighten RLT relaxations.
Finally, we developed a series of discretization schemes for composite functions that are tighter than prevalent schemes and, in certain cases, require exponentially fewer continuous variables. For example, we give a tractable discretized ideal formulation for the hypograph of a supermodular multilinear function where no such formulation was available before.

%
%
%
\newpage
\appendix
\section{Appendix}~\label{app:proof}

\subsection{Proof of Theorem~\ref{them:stair-ineq}}\label{app:stair-ineq}
	Consider a vector of functions $s: X \to \R^{d \times (n+1)} $ so that, for $i \in \{1, \ldots, d\}$ and $j \in \{0, \ldots, n \}$, $s_{ij}(x):= \max\{ u_{ij'}(x) \mid j'\leq j \}$. Let $x \in X$, and define $u:=u(x)$, $a := a(x)$ and $s:=s(x)$. Clearly, $(\phi \mcirc f)(x) = \DO^{\omega}(\phi)(s)$, and, thus, the proof is complete if we show that $\DO^{\omega}(\phi)(s) \leq \BO^{\omega}(\phi)(s,a) \leq \BO^{\omega}(\phi)(u,a)$.  To establish the first inequality, we observe that $\phi(s_{10}, \ldots, s_{d0}) = \phi(a_{10}, \ldots, a_{dn})$, and for $t \in \{1, \ldots, dn\}$ and $\iota := \omega_t$
\begin{equation}~\label{eq:them1-1}
\begin{aligned}
\phi\bigl(\Pi(s;p^t) \bigr) - \phi\bigl(\Pi(s;p^{t-1}) \bigr) &= \phi\bigl(\Pi(s;p^{t-1}) + e_\iota (s_{\iota p^t_{\iota}} - s_{\iota p^{t-1}_{\iota}} ) \bigr) - \phi\bigl(\Pi(s;p^{t-1}) \bigr)\\ 
& \leq \phi\Bigl(\Pi\bigl((a_{-\iota}, s_{\iota});p^{t-1} \bigr) + e_\iota (s_{\iota p^t_{\iota}} - s_{\iota p^{t-1}_{\iota}} ) \Bigr) \\
& \qquad \qquad - \phi\Bigl(\Pi\bigl((a_{-\iota}, s_{\iota} );p^{t-1}\bigr) \Bigr) \\
& = \phi\Bigl(\Pi\bigl((a_{-\iota}, s_{\iota});p^{t} \bigr)  \Bigr) - \phi\Bigl(\Pi\bigl((a_{-\iota}, s_{\iota} );p^{t-1}\bigr) \Bigr),
 \end{aligned}
\end{equation}
where the two equalities hold since, for $i \neq \iota$, $p^t_{i} = p^{t-1}_i$, and the inequality follows from the supermodularity of $\phi(\cdot)$ as $s_{-\iota} \leq a_{-\iota}$ and, by construction, $s_{\iota p^{t-1}_{\iota}} \leq s_{\iota p^{t}_{\iota}}$. Next, we show  that $\BO^\omega(\phi)(s,a) \leq \BO^\omega(\phi)(u,a)$. For $i \in \{1, \ldots, d\}$, let $\{\tau_{i1} ,\ldots, \tau_{in} \}:= \{ t \mid \omega_t = i \}$ so that $\tau_{i1} < \ldots < \tau_{in}$. Then, it follows that for $v \in \{u,s\}$ 
\begin{equation*}
\begin{aligned}
   \BO^\omega(\phi)(v,a) &= \phi\bigl(\Pi(v;p^0)\bigr) + \sum_{i = 1}^d \sum_{j = 1}^{n}\biggl[ \phi\Bigl( \Pi\bigl((a_{-i},v_i); p^{\tau_{ij}} \bigr) \Bigr) - \phi\Bigl( \Pi\bigl((a_{-i}, v_i);p^{\tau_{ij}-1}  \bigr) \Bigr) \biggr] \\
	&  = \phi\bigl(\Pi(a;p^0)\bigr)  + \sum_{i = 1}^d \Biggl[ - \phi\Bigl( \Pi\bigl((a_{-i},s_i); p^{\tau_{i1}-1}\bigr) \Bigr)   \\
	& \qquad \qquad + \sum_{j = 1}^{n-1} \biggl(  \phi\Bigl( \Pi\bigl((a_{-i},v_i); p^{\tau_{ij}}\bigr) \Bigr)   -  \phi\Bigl( \Pi\bigl((a_{-i},v_i); p^{\tau_{ij+1}-1}  \bigr) \Bigr) \biggr) \\
	& \qquad \qquad +\phi\bigl( \Pi\bigl((a_{-i},s_i); p^{\tau_{in}}\bigr) \bigr) \Biggr],
\end{aligned} 
\end{equation*}
where the first equality holds by definition of $\tau$, and the second equality holds by rearranging terms and observing that, for $i \in \{1, \ldots, d\}$, $u_{i0} = s_{i0} = a_{i0}$, $(p^{\tau_{i1}-1})_i=0$, $u_{in} = s_{in}$, and $(p^{\tau_{in}})_i=n$. Therefore, the proof is complete by noticing that, for $i \in \{1, \ldots, d\}$ and for $j \in \{1, \ldots, n-1\}$, the supermodularity of $\phi(\cdot)$ implies that 
\begin{equation}~\label{eq:them1-2}
 \begin{aligned}
  \phi\Bigl( \Pi\bigl((a_{-i},s_i); p^{\tau_{ij}}\bigr) \Bigr)   &- \phi\Bigl( \Pi\bigl((a_{-i},u_i); p^{\tau_{ij}}\bigr) \Bigr) \leq  \phi\Bigl( \Pi\bigl((a_{-i},s_i); p^{\tau_{ij+1}-1}\bigr) \Bigr) \\
 & \qquad \qquad   - \phi\Bigl( \Pi\bigl((a_{-i},u_i); p^{\tau_{ij+1}-1}\bigr) \Bigr), 
\end{aligned}
\end{equation}
because $u_i \leq s_i$, $a_{i0} \leq \cdots \leq a_{in}$, $p^{\tau_{ij}} \leq p^{\tau_{ij+1}-1}$ and $(p^{\tau_{ij}})_i = (p^{\tau_{ij+1}-1})_i$. \Halmos

\subsection{Proof of Corollary~\ref{cor:stair-switching}}\label{app:stair-switching}
Clearly, $ \BO^{\omega}(\phi)\bigl(\tilde{u}(x),\tilde{a}(x)\bigr) =  \BO^{\omega}\bigl(\phi(T)\bigr)\bigl(\tilde{u}'(x),\tilde{a}'(x)\bigr)$, where for each $i \notin T$ $\tilde{u}'_i(x) := \tilde{u}_i(x)$ and $\tilde{a}'_i(x) := \tilde{a}_i(x)$ otherwise $\tilde{u}'_i(x) := f_i^U \mathbf{1} - \tilde{u}_i(x)$ and $\tilde{a}'_i(x) := f_i^U  \mathbf{1} -  \tilde{a}_i(x)$, where $\mathbf{1}$ denotes the all-ones vector. Moreover, for $x \in X$, $(\phi \mcirc f)(x) = \phi(T)\bigl(\tilde{u}'_{1n}(x), \ldots, \tilde{u}'_{dn}(x)  \bigr) \leq \BO^{\omega}\bigl(\phi(T)\bigr)\bigl(\tilde{u}'(x),\tilde{a}'(x)\bigr)$, where the equality holds by definition, and the inequality follows from Theorem~\ref{them:stair-ineq} because the transformed pair $\bigl(\tilde{u}'(x), \tilde{a}'(x)\bigr)$ satisfies the requirement in~(\ref{eq:ordered-oa}) and $\phi(T)(\cdot)$ is supermodular. \Halmos

\subsection{Convex underestimators for the root node in Example~\ref{ex:recursive}}\label{app:inequalities}
\[
x_1^2x_2^2x_3^3 \geq \max \left\{
\begin{aligned}
&e_6(x)  +   u_{32}(x) - 1\\
&2e_1(x)+ e_6(x)  +  6u_{31}(x)  +  u_{32}(x) - 21\\
&3 e_1(x) +e_6(x)  +   7u_{32}(x) - 28\\
&e_2(x)  + e_6(x)  +  10u_{31}(x)  +  u_{32}(x) - 33\\
&\max\{e_3(x),e_4(x)\}  +   e_6(x)  +  12u_{31}(x)  +  u_{32}(x) - 39\\
&e_1(x) +  2e_2(x)  +e_6(x)  +  4u_{31}(x)  +  7u_{32}(x) - 40\\
&3e_2(x)  +  e_6(x)  +  11u_{32}(x) - 44\\ 
&2e_5(x)  +  e_6(x)  +  14u_{31}(x)  +  u_{32}(x) - 45\\
&e_1(x) +   2\max\{e_3(x),e_4(x)\}  +    e_6(x)  +  6u_{31}(x)  +  7u_{32}(x) - 46\\
&3e_6(x)  + 15 u_{31}(x)  +  u_{32}(x)-48\\
&e_2(x)  + 2\max\{e_3(x),e_4(x)\}   +  e_6(x)  +  2u_{31}(x)  +  11u_{32}(x)-50\\
&e_1(x)  +  2e_5(x)  +  e_6(x)  +  8u_{31}(x)  +  7u_{32}(x) -52 \\
&3\max\{e_3(x),e_4(x)\}   +  e_6(x)  +   13u_{32}(x) - 52 \\
&e_1(x) + 3e_6(x)  +  9u_{31}(x)  +  7u_{32}(x) -55 \\
&e_2(x)   +  e_5(x)  +  e_6(x)  +  4u_{31}(x)  +  11u_{32}(x) - 56 \\
&\max\{e_3(x),e_4(x)\}  +  2e_5(x)  +  e_6(x)  +  2u_{31}(x)  +  13u_{32}(x) - 58 \\
&    e_2(x) + 3e_6(x)  +  5u_{31}(x)  + 11u_{32}(x) - 59 \\
&  3e_5(x)  +  e_6(x)  +  15u_{32}(x)-60\\
&    \max\{e_3(x),e_4(x)\}  +  3 e_6(x)  +  3u_{31}(x)  +  13u_{32}(x) -61 \\
&	e_5(x) + 3e_6(x) + u_{31}(x) + 15u_{32}(x)  - 63 \\
&  4e_6(x)  + 16u_{32}(x) -64
\end{aligned}
\right\} .
\]

\subsection{Additional convex underestimators for inner-functions tighten the relaxation}\label{app:example-more-underestimators}
Consider $(x_1-\frac{1}{x_2})^2 - \frac{x_1}{x_2}$ over $[1,2]\times [\frac{1}{4},4]$, and consider $(f_1-f_2)^2-f_1f_2$ and $(x_1, \frac{1}{x_2})$ as its outer-function and inner-functions, respectively. Notice that the outer-function is submodular over but not convex-extendable from $[1,2] \times [0.25,4]$. Consider a convex underestimator $l(x):=\max\{-x_2+2,0.15+\frac{0.4}{x_2} \}$ of $\frac{1}{x_2}$, which is bounded from above by $1.75$ over $[0.25,4]$. By Theorem~\ref{them:stair-ineq}, we obtain the following underestimators for $(x_1-\frac{1}{x_2})^2-\frac{x_1}{x_2}$:
\begin{equation*}
	\begin{aligned}
		\varphi_1(x)&:=\Bigl(x_1 - \frac{1}{4}\Bigr)^2 - \frac{1}{4}x_1 +\Bigl(2 - \frac{1}{x_2}\Bigr)^2 -\frac{2}{x_2} - \Bigl(\frac{7}{4}\Bigr)^2 +\frac{1}{2}, \\
		\varphi_2(x)&:= (x_1-4)^2 - 4x_1 + \Bigl(1-\frac{1}{x_2}\Bigr)^2 - \frac{1}{x_2} -5, \\
		\varphi_3(x)&:=  \Bigl(x_1 - \frac{7}{4}\Bigr)^2 -\frac{7}{4}x_1  + \Bigl(2 - \frac{1}{x_2}\Bigr)^2 - \frac{2}{x_2} + 3l(x) - \Bigl(\frac{3}{4}\Bigr)^2 - \frac{5}{4}.
	\end{aligned}
\end{equation*}
Observe that only $\varphi_3(\cdot)$ exploits the structure of the underestimator $l(\cdot)$, and its convex envelope is not dominated by that of $\varphi_1(\cdot)$ and $\varphi_2(\cdot)$ since, at the point $(x_1,x_2)=(1.8,0.3)$, $\conv(\varphi_3)(x) >-4.75$ while $\conv(\varphi_1)(x)< \varphi_1(x) < -5.49$ and $\conv(\varphi_2)(x)< \varphi_2(x) < -5.24$. In particular, it can be easily verified that the convex function $g(x)$ underestimates $\varphi_3(x)$ with $g(1.8,0.3)>-4.75$, where
$
g(x):=\bigl(x_1 - \frac{7}{4}\bigr)^2 -\frac{7}{4}x_1 + h(x_2) +3(-x_2+2)- \left(\frac{3}{4}\right)^2 -\frac{5}{4}$,
and $h(x_2)= 4-\frac{6}{x_2}+\frac{1}{x_2^2}$ when $x_2 < \frac{1}{3}$ and $-5$ otherwise.
\Halmos

\subsection{Proof of Proposition~\ref{prop:stair-Q}}\label{app:stair-Q}
We start with showing the second statement. First, we prove the validity of $\hat{\BO}^\omega(\phi)(s;a)$. Observe that for every $s \in \vertex(Q)$ we have $\ephi(s) = \phi(s_{1n}, \ldots, s_{dn} )  \leq \BO^\omega(\phi)(s;a) = \hat{\BO}^\omega(\phi)(s;a)$, where the first inequality follows from Theorem~\ref{them:stair-ineq}, and the second equality holds because $\hat{\BO}^\omega(\phi)(s;a)$ is obtained by interpolating $B^\omega_i(\phi)(s;a)$ over $\vertex(Q_i)$. Therefore, we obtain that $\conc_Q(\ephi)(s) = \conc_Q(\ephi|_{\vertex(Q)})(s) \leq \hat{\BO}^\omega(\phi)(s;a)$, where the equality holds by the concave-extendability of $\phi(\cdot)$, and the inequality holds because $\hat{\BO}^\omega(\phi)(s;a)$ is a concave overestimator of $\ephi|_{\vertex(Q)}(s)$. Next, we show that $\min_{\omega \in \Omega}\hat{\BO}^\omega(\phi)(s;a) \leq \conc_{Q}(\ephi)(s)$ for $s \in Q$. For $\omega \in \Omega$, let $\Upsilon_\omega$ be the simplex defined as the convex hull of $\bigl\{ (v_{1p^{t}_1}, \ldots, v_{dp^t_d} )\bigr\}_{t=0}^{dn}$. It follows from Section 2.2 in~\cite{he2021tractable} or Theorem 6.2.13 in~\cite{de2010triangulations} that $\{\Upsilon_\omega\}_{\omega \in \Omega}$ is a triangulation of $Q$, \textit{e.g.}, $Q = \cup_{\omega\in \Omega} \Upsilon_{\omega}$ and for $\omega', \omega'' \in \Omega$ $\Upsilon_{\omega'} \cap \Upsilon_{\omega''}$  is a face of both $\Upsilon_{\omega'}$ and $\Upsilon_{\omega''}$. Given any $\s \in Q$, there exists $\bar{\omega} \in \Omega$ and a convex multiplier $\lambda$ so that $\s= \sum_{t=0}^{dn} \lambda_t\bigl(v_{1p^t_1}, \ldots, v_{dp^t_d}\bigr)$, where $p^0 = 0$ and $p^t = p^{t-1} + e_{\bar{\omega}_t}$. Moreover, we obtain that 
\[
\begin{aligned}
\min_{\omega \in \Omega}\hat{\BO}^\omega(\phi)(\s;a) \leq \hat{\BO}^{\bar{\omega}}(\phi)(\s;a)  &=  \sum_{t = 0}^{dn}\lambda_t\hat{\BO}^{\bar{\omega}}(\phi)\bigl(v_{1p^t_1}, \ldots, v_{dp^t_d};a\bigr)  \\
& = \sum_{t = 0}^{dn}\lambda_t \phi\bigl(\Pi(a;p^t)\bigr) \leq \conc_Q(\ephi)(\s),
\end{aligned}
\]
where the first equality holds as $\bar{\omega} \in \Omega$,  the first equality holds by the linearity of $\hat{\BO}^{\bar{\omega}}(\phi)(\cdot;a)$ and the convex combination representation of $\s$, and the last inequality follows since $\conc_Q(\ephi)$ is concave and, for $t \in \{0, \ldots, dn\}$, $\phi\bigl(\Pi(a;p^t)\bigr) = \conc_Q(\ephi)(v_{1p^t_1}, \ldots, v_{dp^t_d})$. To see the second equality, we observe that  $\hat{\BO}^{\bar{\omega}}(\phi)\bigl(v_{1p^0_1}, \ldots, v_{dp^0_d};a\bigr) = \phi\bigl(\Pi(a,p^0)\bigr)$ and for $t \in \{1, \ldots, dn\}$
\[
\begin{aligned}
\hat{\BO}^{\bar{\omega}}(\phi)\bigl(v_{1p^t_1}, \ldots, v_{dp^t_d};a\bigr) &= \hat{\BO}^{\bar{\omega}}(\phi)\bigl(v_{1p^{t-1}_1}, \ldots, v_{dp^{t-1}_d};a\bigr)  + \phi\bigl(\Pi(a;p^t)\bigr) - \phi\bigl(\Pi(a;p^{t-1})\bigr). 
\end{aligned}
\]
In other words, for $t \in \{0, \ldots, dn\}$, $\hat{\BO}^{\bar{\omega}}(\phi)\bigl(v_{1p^t_1}, \ldots, v_{dp^t_d};a\bigr) = \phi\bigl(\Pi(a;p^t)\bigr)$. 

Now, the fourth statement follows the same argument as the first one by observing that $\hat{\BO}(\phi)(\tilde{s}; \tilde{a})$ the affine interpolation of $\BO(\phi)(\tilde{s};\tilde{a})$ over $\vertex(\tilde{Q})$, which, by Corollary~\ref{cor:stair-switching}, is an overestimator of $\phi(\tilde{s}_{1n}, \ldots, \tilde{s}_{dn})$ over $\tilde{Q}$, and observing that $\bigl\{U(T)(\Upsilon_\omega)\bigr\}_{\omega \in \Omega}$ is a triangulation of $\tilde{Q}$ and, for $\omega \in \Omega$ and for $\tilde{s} \in \vertex\bigl(U(T)(\Upsilon_\omega)\bigr)$, $\hat{\BO}(\phi)(\tilde{s};\tilde{a}) = \phi(\tilde{s}_{1n}, \ldots, \tilde{s}_{dn} )$.\Halmos

\subsection{Illustration that Theorem~\ref{them:improve_mor} improves over Proposition~\ref{prop:stair-Q}}\label{app:example-convex-each-argument}
Consider a function $\sqrt{x_1+x_2^2}$ over $[0,5] \times [0,2]$, and let $\sqrt{f_1+f_2}$ be its outer-function, which is submodular over $[0,5] \times [0,4]$. Let $s_{1}(x):=(0,x_1)$ and $s_2(x) :=\bigl(0,\ \max\{\frac{3}{4}x_2^2, 2x_2-1 \},\ x_2^2\bigr)$, and consider their upper bounds $a_1 := (0,5)$ and $a_2 := (a_{20}, a_{21}, a_{22}) = (0,3,4)$. Observe that there are $3$ staircases, and for the staircase $\bigl((0,0),(1,0), (1,1), (1,2)\bigr)$, Proposition~\ref{prop:decomposition} yields:
\[
\begin{aligned}
\sqrt{x_1+x_2^2} &= \sqrt{0+0} + \bigl(\sqrt{x_1+0} - \sqrt{0+0}\bigr) + \bigl(\sqrt{x_1+s_{21}(x)}- \sqrt{x_1+0}\bigr) \\
&\quad \qquad + \bigl(\sqrt{x_1+s_{22}(x)}- \sqrt{x_1+s_{21}(x)}\bigr) \\
	&\geq \sqrt{x_1} + \sqrt{5+s_{21}(x)} - \sqrt{5} + \sqrt{5+s_{22}(x)} - \sqrt{5+s_{21}(x)} \\
	& \geq \frac{\sqrt{5}}{5}x_1+\sqrt{5+x_2^2} - \sqrt{5}=:\varphi_1(x),
\end{aligned}
\]
where the first inequality holds by the submodularity of $\sqrt{f_1+f_2}$, and the second inequality follows because $\frac{\sqrt{5}}{5}x_1$ is the affine interpolation of the concave function $\sqrt{x_1}$ over $[0,5]$. In contrast, Example 1 of~\cite{he2021tractable} yields
$
\sqrt{x_1+x_2^2} \geq \frac{\sqrt{5} }{5}x_1 + \Bigl( \frac{\sqrt{8} - \sqrt{5}}{3} - \frac{\sqrt{9} - \sqrt{8}}{1}\Bigr)s_{21}(x) + \frac{\sqrt{9} - \sqrt{8}}{1}s_{22}(x)$,
which can be obtained by relaxing $\sqrt{5+x_2^2} - \sqrt{5}$ in $\varphi_1(x)$ as follows:
\[
\begin{aligned}
\sqrt{5+x_2^2} - \sqrt{5} &= \sqrt{5+s_{21}(x)} - \sqrt{5} + \sqrt{5+s_{22}(x)} - \sqrt{5+s_{21}(x)} \\
&\geq  \frac{\sqrt{8} - \sqrt{5}}{3} s_{21}(x) + \frac{\sqrt{9} - \sqrt{8}}{1}\bigl(s_{22}(x) - s_{21}(x)\bigr),
\end{aligned}
\]
where the inequality holds by the concavity of $\sqrt{5+f_2}$ over $[0,4]$. \Halmos

\subsection{Proof of Proposition~\ref{prop:Kronecker}}\label{app:Kronecker}
	For a direction vector $\omega$ in the grid given by $\{0, \ldots, n\}^2$, we proceed as follows:
\[
\begin{aligned}
F(x) \otimes G(x) &= U^n(x) \otimes V^n(x) \\
&=  U^0(x) \otimes V^0(x) + \sum_{t=1}^{2n} \biggl(U^{p^t_1}(x) \otimes V^{p^t_2}(x) - U^{p^{t-1}_1}(x) \otimes V^{p^{t-1}_2}(x) \biggr) \\
& = U^0(x) \otimes V^0(x) +  \sum_{t:\omega_t = 1} \Bigl(U^{p^t_1}(x) - U^{p^{t-1}_1}(x)\Bigr) \otimes V^{p^t_2}(x)   \\
& \qquad \qquad \qquad \qquad \qquad +  \sum_{t:\omega_t = 2} U^{p^t_1}(x) \otimes \Bigl(V^{p^t_2}(x) - V^{p^{t-1}_2}(x)\Bigr) \\
& \preceq A^0(x) \otimes B^0(x) +  \sum_{t:\omega_t = 1} \Bigl(U^{p^t_1}(x) - U^{p^{t-1}_1}(x)\Bigr) \otimes B^{p^t_2} (x)  \\
& \qquad \qquad \qquad \qquad \qquad +  \sum_{t:\omega_t = 2} A^{p^t_1}(x) \otimes \Bigl(V^{p^t_2}(x) - V^{p^{t-1}_2}(x)\Bigr),
\end{aligned}
\]
where the first equality follows from the first requirement in~(\ref{eq:matrix-system}), the second equality follows from staircase expansion, the third equality follows from the second statement in Lemma~\ref{lemma:Kron-prod}, and the inequality follows from the second and third statement in Lemma~\ref{lemma:Kron-prod} and~(\ref{eq:matrix-system}). \Halmos

\subsection{Proof of Theorem~\ref{them:stair-lattice}}\label{app:stair-lattice}
The proof is similar to that of Theorem~\ref{them:stair-ineq} except that inequalities~(\ref{eq:them1-1}) and~(\ref{eq:them1-2}) need some justifications. To generalize (\ref{eq:them1-1}), let the $t^{\text{th}}$ movement be along $k^{\text{th}}$ coordinate. Then,
\begin{equation}\label{eq:degreereduce}
\phi(\Pi(\tilde{\upsilon}, p^t)) - \phi(\Pi(\tilde{\upsilon}, p^{t-1})) 
\leq \phi\Bigl(\Pi\bigl((\alpha_{-k}, \tilde{\upsilon}_k), p^t\bigr) \Bigr)  - \phi\Bigl(\Pi\bigl((\alpha_{-k}, \tilde{\upsilon}_k), p^t\bigr) \Bigr),
\end{equation}
where the inequality follows readily by recursively replacing coordinates of $\tilde{v}$ with those of $\alpha$ using $\tilde{\upsilon}_{ij} \preceq_i \alpha_{ij}$ for all $i$,  $\tilde{\upsilon}_{kj-1} \preceq_k \tilde{\upsilon}_{kj}$ for all $j \neq n$, and $p^t - p^{t-1} = e_k$. Now, we generalize (\ref{eq:them1-2}). For any two points $p'$ and $p''$ of the grid $\{0, \ldots, n\}^d$ such that $p'_k = p''_k$ and $p'_j\leq p''_j$ for $j \neq k$ and for any permutation $\sigma$ of $\{1, \ldots, d\}\setminus \{k\}$, let $p^0 = p'$ and $p^i = p^{i-1} + \bigl(p''_{\sigma(i)} - p'_{\sigma(i)}\bigr)e_{\sigma(i)}$ for $i =1, \ldots, d-1$. Observe that 
\[
\begin{aligned}
\phi\Bigl(\Pi\bigl((\alpha_{-k}, \tilde{\upsilon}_k), p'\bigr) \Bigr)& -\phi\Bigl(\Pi\bigl((\alpha_{-k}, \tilde{\upsilon}_k), p''\bigr) \Bigr)  \\
&=  \sum_{i =1 }^{d-1}  \phi\Bigl(\Pi\bigl((\alpha_{-k}, \tilde{\upsilon}_k), p^{i-1}\bigr) \Bigr) -  \phi\Bigl(\Pi\bigl((\alpha_{-k}, \tilde{\upsilon}_k), p^i\bigr) \Bigr) \\
& \leq   \sum_{i=1}^{d-1}  \phi\Bigl(\Pi\bigl((\alpha_{-k}, \upsilon_k), p^{i-1}\bigr) \Bigr)- \phi\Bigl(\Pi\bigl((\alpha_{-k}, \upsilon_k), p^i\bigr) \Bigr)\\
&=\phi\Bigl(\Pi\bigl((\alpha_{-k}, \upsilon_k), p'\bigr) \Bigr) -  \phi\Bigl(\Pi\bigl((\alpha_{-k}, \upsilon_k), p''\bigr) \Bigr),
\end{aligned}
\]
where both equalities from the staircase expansion, and the inequality follows from $\upsilon_{ij} \preceq_i \tilde{\upsilon}_{ij}$ and increasing differences of $\phi(\cdot)$. \Halmos

\subsection{Proof of Lemma~\ref{lemma:Inc}}\label{app:Inc}
We start with proving the first statement. Let $(z,\delta)$ be a point that satisfies~(\ref{eq:Inc-1}). Then, there exists a vector of indices $(t_1, \ldots, t_d) \in \prod_{i=1}^d\{1, \ldots, l_i\}$ such that, for each $i \in \{1, \ldots, d\}$, $\delta_{ij} = 1$ for $j < t_i$ and $\delta_{ij} = 0$ for $j \geq t_i$. Thus, for every $i \in \{1, \ldots, d\}$, $z_{i \tau(i,t_i-1)} = 1$ and $z_{ij} = 0$ for $j>  \tau(i,t_i)$. In other words, $z \in \Delta_H$. Conversely, let $z \in \Delta_H$, where $H = \prod_{i=1}^d[a_{\tau(i,t_i-1)},a_{i\tau(i,t_i)}]$. Then, for $i \in \{1, \ldots, d\}$, $z_{i\tau(i,t_i-1)} = 1$ and $z_{ij} = 0$ for $j > \tau(i,t_i)$. We can define a binary vector $\delta$ such that, for every $i \in \{1, \ldots, d\}$, $\delta_{ij}= 1$ for $j < t_i$ and $\delta_{ij} = 0$ for $j \geq t_i$ so that $(z,\delta)$ satisfies~(\ref{eq:Inc-1}). Now, we prove the second statement.  Let $\zeta_{ij} :=\sum_{j'=0}^je_{ij'} $, where $e_{ij}$ is the $j^{\text{th}}$ principal vector in the space spanned by variables $(z_{i0}, \ldots,  z_{in} )$. Then, the result follows by observing that the affine function $F_i$ is non-decreasing and maps $\zeta_{i\tau(i,t_i-1)} $ and $\zeta_{i\tau(i,t_i)}$  to $a_{i\tau(i,t_i-1)}$ and $a_{i\tau(i,t_i)}$, respectively. 

If $z\in \Delta_{\bar{H}}\backslash\Delta_{H}$, $Z^{-1}(z)_{\cdot n}\not\in H$. Let $\delta^H$ be the setting of $\delta$ variables that force $z\in \Delta_H$. Then, it follows that $(z,\delta^{\bar{H}}) \in \eqref{eq:Inc-1}$ while $(z,\delta^{H}) \not\in \eqref{eq:Inc-1}$. Assume without loss of generality, that there is an $(i,t)$ such that $\delta^{\bar{H}}_{it} \le z_{i\tau(i,t)} < \delta^H_{it}$. Then, it follows that $\delta^{\bar{H}}_{it} = 0$, $\delta^{\bar{H}}_{it} = 1$, $z_{ik} = 0$ for $k > \tau(i,t)$ and the left end-point of $H_i$ is no less that $a_{i\tau(i,t)}$. Moreover, $F(z) = a_{i0} + \sum_{j=1}^{\tau(i,t)} (a_{ij}-a_{ij-1})z_{ij} < a_{i\tau(i,t)}$ showing that $F_i(z_i)\not\in H_i$. \Halmos

\subsection{Proof of Proposition~\ref{prop:eval}}\label{app:eval}
We first show that $\varphi_{\mathcal{H}}(\bar{x}, \bar{f}) = \conc_Q(\ephi)(s^*)$. By the first statement of Lemma~\ref{lemma:Inc} and the definition of $Z$, $\varphi_{\mathcal{H}}(\bar{x},\bar{f}) = \max\bigl\{ \conc_Q(\ephi)(s) \bigm| \u \leq s,\ s_{\cdot n} = \bar{f},\ s \in \cup_{H \in \mathcal{H}}Q_H \bigr\}$, where the feasible region of the right hand side will be denoted as $\mathcal{L}$. We will show that $s^*$ is the smallest point in $\mathcal{L}$, \textit{i.e.}, $s^* \in \mathcal{L}$ and, for every $s' \in \mathcal{L}$, $s^* \wedge s' = s^*$, where $\wedge$ denotes the component-wise minimum of two vectors. Then, $\max\{\conc_Q(\ephi)(s) \mid s \in \mathcal{L}\} \geq  \conc_Q(\ephi)(s^*) \geq \max\{\conc_Q(\ephi)(s) \mid s \in \mathcal{L}\}$, where the first inequality holds as $s^* \in \mathcal{L}$, and the second inequality holds because, by Lemma 8 in~\cite{he2021new}, $\conc_Q(\ephi)$ is non-increasing in $s_{ij}$ for all $i$ and $j \neq n$, and, for every point $s' \in \mathcal{L}$ we have $s' \geq s^*$ and $s'_{\cdot n} = s^*_{\cdot n}$. In other words, $\varphi_{\mathcal{H}}(\bar{x},\bar{f}) = \conc_Q(\phi)(s^*)$.

Here, we show that $s^*$ is the smallest point in $\mathcal{L}$. It follows from Proposition 4 in~\cite{he2021new} that $s^*_i \in  Q_i$ and $s^*_i \geq u^*_i$. As $u_i^* \geq \u_i$, we obtain that $s^*_i \geq \u_i$. Moreover, since $s^*_{ij}=a_{ij}$ for $j \leq \tau(i,\bar{t}_i-1)$ (resp. $s^*_{ij} = \bar{f}_i$ for $j > \tau(i,\bar{t}_i)$), it follows that $Z(s^*)_{ij} = 1$ for $j \leq \tau(i,\bar{t}_i-1)$ (resp. $Z(s^*)_{ij} = 0$ for $j > \tau(i,\bar{t}_i)$). In other words, $Z(s^*) \in \Delta_{\bar{H}}$ and, so, $s^* \in Q_{\bar{H}}$. Therefore, we can conclude that $s^*\in \mathcal{L}$. Now, we show that, for $s' \in \mathcal{L}$, $s' \wedge s^* = s^*$. For any $s' \in \mathcal{L}$, we have $s'_{\cdot n}  \in \bar{H}$ and $Z(s') \in \cup_{H \in \mathcal{H}}\Delta_H$.
Then, by the second statement of Lemma~\ref{lemma:Inc} and $ s'_{\cdot n} = F\bigl(Z(s')\bigr)$, we obtain that $Z(s') \in \Delta_{\bar{H}}$.  Therefore, $s' \in Q_{\bar{H}}$, and in particular, $s'_{ij} = a_{ij}$ for $j \leq \tau(i,\bar{t}_i-1)$ and $s'_{ij} = \bar{f}_i$ for $j \geq \tau(i,\bar{t}_i)$.  This, together with $\u \leq s'$, implies that $ u^* \leq s'$, and thus $\xi_{i,a_i}(a;u^*_i) \leq \xi_{i,a_i}(a;s'_i) $ for every $a \in [a_{i0}, a_{in}]$. It turns out that for  $i \in \{1, \ldots, d\}$ and for  $j \in \{0, \ldots, n \}$, $(s'_i \wedge s^*_i)_j = \min\bigl\{ \conc(\xi_{i,a_i})(a_{ij};s'_i), \conc(\xi_{i,a_i})(a_{ij};u^*_i) \bigr\} = \conc(\xi_{i,a_i})(a_{ij};u^*_i) = s^*_{ij}$,
where the first equality holds because it follows from Lemma 5 in~\cite{he2021new} that, for any $s_i \in Q_i$, we have $s_{ij} = \conc(\xi_{i,a_i})(a_{ij};s_i)$, and the second equality follows from $\xi_{i,a_i}(a;u^*_i) \leq \xi_{i,a_i}(a;s'_i)$. 

Similarly, $\varphi(\bar{x}, \bar{f}) = \conc_Q(\ephi) (\s)$ and $\varphi_{\mathcal{H}-}(\bar{x}, \bar{f}) = \conc_Q(\ephi) (\hat{s})$. Hence, the proof is complete since for each $i\in \{1, \ldots, d\}$ we have $s^*_i \geq \s'_i$ and $s^*_i \geq \s''_i$ with $s^*_{in} =\s'_{in} = \s''_{in}$, and by Lemma 8 in~\cite{he2021new},  $\conc_Q(\bar{\phi})$ is non-increasing in $s_{ij}$ for all $i$ and $j \neq n$. \Halmos

\subsection{Proof of Theorem~\ref{them:DCR+}}\label{app:DCR+}
\newcommand{\ba}{{\bar{a}}}
\newcommand{\baa}{{\breve{a}}}
\newcommand{\bn}{{\bar{n}}}
\newcommand{\ban}{{\breve{n}}}

\begin{OldVersion}
\begin{lemma}\label{lemma:lifta}
Let $\baa_{iK_i} = \ba_i$, $\bn_i=|K_i|-1$, and $\ban_i=|\baa_i|-1$. We lift \eqref{eq:ordered-oa} defined with $(a,u)=(\ba,u^{\ba})$ to that with $(a,u)=(\baa,u^{\baa})$ by defining $u^{\baa}_{ij}$ as follows:
\begin{equation*}
u^{\baa}_{ij} = \left\{\begin{alignedat}{2}
&u^{\ba}_{i,r-1} &\quad& j\text{ is the $r^{th}$ largest entry in } K_i\\
&\baa_{i0}-\epsilon && j\not\in K_i, j > 0, j < \ban_i \\
&\baa_{i0} && j\not\in K_i, j=0\\
&u^{\ba}_{i,\bn_i} &&j\not\in K_i, j=\ban_i.
\end{alignedat}\right.  
\end{equation*}
Let $Q_{\ba}$ (resp. $Q_{\baa}$) be the simplex defined with $\ba$ (resp. $\baa$). Let $s^{\baa}_i = \conc(\xi_{i\baa})(\cdot,u^{\baa})$ and $s^{\ba}_i = \conc(\xi_{i\ba})(\cdot,u^{\ba})$, where $\xi_{i\ba}(y,u^{\ba})=u^{\ba}_{ij}$ if $y=\ba_{ij}$. Then, $\conc_{Q_{\baa}}(\phi)(s^{\baa}) =
\conc_{Q_{\ba}}(\phi)(s^{\ba})$.
\end{lemma}
\begin{proof}
Let
\begin{equation}
    h(y) = \left\{\begin{alignedat}{2}
    &y &\quad& y\le \ba_{i0}\\
    &\conc(\xi_{\ba})(y,u^{\ba}) && y\in [\ba_{i0},\ba_{i\bn_i}]\\
    &u^{\ba}_{i\bn_i} && y\ge \ba_{i\bn_i}.\\
    \end{alignedat}\right.
\end{equation}
We show that $h(\cdot) = \conc(\xi_{i,\baa})(\cdot;u^{\baa})$. To begin, we show that $h(y)$ is concave and $h(\baa_{ij})\ge u^{\baa}_{ij}$, where the inequality is strict if $j\not\in K_i\cup\{0\}\cup\{j'\mid \nexists j'\}$. The concavity of $h$ follows since $1\ge \frac{s^{\ba}_{i1}-s^{\ba}_{i0}}{\ba_{i1}-\ba_{i0}}\ge \frac{s^{\ba}_{i\bn_i}-s^{\ba}{i,\bn_i-1}}{\ba_{i,\bn_i}-\ba_{i,\bn_i-1}}\ge 0$. We now show that $h(\baa_{ij})\ge u^{\baa}_{ij}$. If $j\in K_i$, the result follows since $h(y) = \conc(\xi_{\ba})(y,u^{\ba})$ and the concave envelope construction used a point $(\baa_{ij}, u^{\baa}_{ij})$. Now, let $j\not\in K_i$. If there does not exist $j'\in K_i$ such that $j' < j$, then $h(\baa_{ij}) = \baa_{ij}\ge \baa_{i0}$, where the last ineqaulity is strict if $j > 0$. Now, assume that there exist $j'\in K_i$ such that $j' < j$ and $j<\ban_i$. Then, $h(\baa_{ij})\ge \bar{a}_{i0} > u^{\baa}_{ij}$. Finally, if $j=\ban_i$, we have $h(\baa_{ij})=u^{\ba}_{i\bn_i}=u^{\baa}_{ij}$. Clearly, $\conc(\xi_{i\baa})(\cdot,u^{\baa})\le h(\cdot)$ since $h(\cdot)$ is concave and overestimates the left-hand-side. Combining this with the observation that $\conc(\xi_{i\baa})(\cdot,u^{\baa})$ is constructed by including all points that are extremal in the hypograph of $h(y)$, it follows that $\conc(\xi_{i\baa})(\cdot,u^{\baa})= h(\cdot)$. Consequently, $s^{\baa}_{iK_i} = s^{\ba}_i$. It follows from Proposition 5 of~\cite{he2021new} that $s^{\baa}_i$ belongs to the face $F_{K_i\cup\{0\}\cup\{|\baa_i|\}}$ of $Q_{\baa}$. Moreover, if $0\not\in K_i$ (resp. $|\baa_i|\not\in K_i$), we have
$\frac{s^{\baa}_{i1}-s^{\baa}_{i0}}{\baa_{i1}-\baa{i0}}=1$ (resp. $\frac{s^{\baa}_{i\bn_i}-s^{\baa}_{i,\bn_i-1}}{\baa_{i\bn_i}-\baa{i,\bn_i-1}}=0$ showing that $s^{\baa}_i$ belongs to the face $F_{K_i}$ of $Q_{\baa}$. However, this face is isomorphic to $Q_{\ba}$. By Lemma 8 of~\cite{he2021new}, the hypographs of the concave envelope are also isomorphic and therefore, $\conc_{Q_{\baa}}(\phi)(s^{\baa}) =
\conc_{Q_{\ba}}(\phi)(s^{\ba})$.\qed
\end{proof}

\begin{lemma}\label{lemma:moveu}
Let $j' < j$, $j\not= 0$, and $\baa_{ij'}\ge u^\baa_{ij}$. Define $u'^{\baa}$ so that 
$u'^{\baa}_{ij} = a_{i0}-\epsilon$ if $j<|\baa_i|$, $u'^{\baa}_{i|\baa_i|} = u^{\baa}_{i|\baa_i|}$, $u'^{\baa}_{ij'} = \max\{u^{\baa}_{ij},u^{\baa}_{ij'}\}$.
Then, $\conc(\xi_{i\baa})(\cdot,u'^\baa)\ge \conc(\xi_{i\baa})(\cdot,u^\baa)$.
\end{lemma}
\begin{proof}
Let $h'(\cdot) = \conc_{\xi_{i\baa}}(\cdot,u'^\baa)$ and $h(\cdot) = \conc_{\xi_{i\baa}}(\cdot,u^\baa)$. It suffices to show that  $h'(\baa_{ij})\ge u^\baa_{ij}$ and $h'(\baa_{ij'})\ge u^\baa_{ij'}$.  These follow from $h'(\baa_{ij}) \ge h'(\baa_{ij'})\ge u'(\baa_{ij'}) =  \max\{\max\{u^{\baa}_{ij},u^{\baa}_{ij'}\}\ge u^\baa_{ij}$, where the first inequality is because $h'(\cdot)$ is non-decreasing by Lemma 5 of~\cite{he2021new}, the second inequality is by concave envelope construction, and the equality is by definition. \qed
\end{proof}

Now, we prove Lemma~\ref{lemma:improvement-bdd}. We may assume without loss of generality that $a$ and $a'$ are strictly increasing. If not, let $j$ be such that $a_{ij} = a_{i,j+1}$. We assume that $u_{i,j+2} = \max\{u_{ij},u_{ij+1}\}$ since $\max\{u_{ij}(x),u_{i,j+1}(x)\}$ is a valid underestimator for $f_i(x)$. Then, the inequalities for $Q$ can be written as $(s_{ij}-s_{i,j-1})(a_{i,j+1}-a_{ij}) \ge (s_{i,j+1}-s_{ij})(a_{ij}-a_{i,j-1})$ and $(s_{i,j+1}-s_{ij})(a_{i,j+2}-a_{ij+1}) \ge (s_{i,j+2}-s_{i,j+1})(a_{i,j+1}-a_{ij})$. Consider the system where $u_{ij}(x)$ and $u_{i,j+1}(x)$ were not present. In other words, let $(\tilde{a},\tilde{u})$ be defined so that $\tilde{a}_i = a_{iK_i}$ and $\tilde{u}_i=u_{iK_i}$, where $K_i = \{0,\ldots,n\}\backslash \{j,j+1\}$. We let $(\tilde{a}_{i'}, \tilde{u}_{i'}) = (a_{i'},u_{i'})$ for $i'\ne i$. Then, we can lift any feasible point $\tilde{s} \in Q_{\tilde{a}}$ to one in $s\in Q_{a}$ by defining $s_{ij'}=\tilde{s}_{ij'}$ for $j' \le j$, $s_{i,j+1} =\tilde{s}_{ij}$, and $s_{ij'} = \tilde{s}_{i,j'-2}$ for $j' \ge j+2$. Moreover, if $\tilde{s}_{ij'} \ge \tilde{u}_{ij'}$ for $0\le j'\le n-2$, $\tilde{s}_{i0}=\tilde{u}_{i0}$, and $\tilde{s}_{i,n-2}=\tilde{u}_{i,n-2}$ then it follows easily that $s_{ij'}\ge u_{ij'}$ for $0\le j'\le n$, $s_{i0}=u_{i0}$, and $s_{in}=\tilde{s}_{i,n-2} = \tilde{u}_{i,n-2} = u_{in}$. Therefore, $s\in Q_a$.  Moreover, if $s\in Q_a$, it follows easily by setting $\tilde{s}_i=s_{iK_i}$ that $\tilde{s}\in Q_{\tilde{a}}$. In other words, the inequalities obtained with $a$, when projected, yield the same relaxation as if we had used $\tilde{a}$ instead. Since $\tilde{a}$ satisfies at least one less equality than $a$ (even if the procedure required adding $\max\{u_{ij},u_{i,j+1}\}$ as a new underestimator), it follows, by induction, that we can remove repetitions from $a$ without affecting the relaxation. The argument is similar for the case when $a'$ is not strictly increasing.

We construct $\baa$ by merge sorting $a$ and $a'$ and removing any replicates. Then, we lift $(a,u)$ (resp. $(a',u)$) to $(\baa, u^{a})$ (resp. $(\baa, u^{a'})$) as in Lemma~\ref{lemma:lifta}. Let $s^a_{ij} = \conc(\xi_{ia_i})(a_{ij},u_{ij})$, $s^{a'}_{ij} = \conc(\xi_{i{a'}_i})({a'}_{ij},u_{ij})$, $s^1_{ij} = \conc(\xi_{i\baa_i})(\baa_{ij},u^a_{ij})$, and $s^2_{ij} = \conc(\xi_{i\baa_i})(\baa_{ij},u^{a'}_{ij})$.
It follows from Lemma~\ref{lemma:lifta} that $\conc_{Q_\baa}(\phi)(s^1) = \conc_{Q_a}(\phi)(s^a)$ and $\conc_{Q_\baa}(\phi)(s^2) = \conc_{Q_{a'}}(\phi)(s^{a'})$. Now, we show that $s^1\le s^2$.  We carry out the following procedure. Intialize $k=0$, $a^0=a$, and let $u^0$ be obtained by lifting $(a,u)$ to $(\baa,u^0)$ as in Lemma~\ref{lemma:lifta}. Let $J^k=\{\bar{j}\mid a'_{i\bar{j}} < a^k_{i\bar{j}}\}$. As long as $J^k$ is not empty, let $j''$ be the smallest index in $J^k$. Then, there exist $j,j'$ with $j' < j$ and $j\ne 0$ such that $\baa_{ij} = a^k_{ij''}$ and $\baa_{ij'}=a^k_{ij''}$. We have $u^k_{ij}=u_{ij''}\le a'_{ij''} = \baa_{ij'}$. Use the construction in Lemma~\ref{lemma:moveu} to obtain $u^{k+1}$ and observe that $\conc(\xi_{i\baa_i})(\cdot,u^{k+1})\ge \conc(\xi_{i\baa_i})(\cdot,u^{k})$. Let $a^{k+1}_{i\bar{j}} = a^k_{i\bar{j}}$ for $j\ne j''$ and $a^{k+1}_{ij''} = a'_{ij''}$, increase $k$ by $1$ and iterate. Since, $|J^{k+1}|<|J^{k}|$ because $j''\in J^{K}\backslash J^{k+1}$, it follows that the procedure terminates in $t$ iterations. Moreover, $s^1_i$ (resp. $s^2_i$) is obtained by evaluating $\conc(\xi_{i\baa_i})(\cdot,u^{0})$ (resp. $\conc(\xi_{i\baa_i})(\cdot,u^{t})$) at $\baa_i$. Therefore, $s^1\le s^2$. Now, 
\begin{equation*}
    \conc_{Q_a}(\phi)(s^a) =  \conc_{Q_\baa}(\phi)(s^1) \le \conc_{Q_\baa}(\phi)(s^2) = \conc_{Q_{a'}}(\phi)(s^{a'}),
\end{equation*}
where the inequality follows from Proposition~8 in \cite{he2021new} and $s^1\le s^2$.

---------------------------------------
\end{OldVersion}

\newcommand{\SSb}{\tilde{S}}
\newcommand{\SSa}{S}
 We may assume without loss of generality that $a$ and $a'$ are strictly increasing. We treat $a'$ is a vector of values rather than a function of $\delta$ because our arguments will assume that $\delta$ is fixed at binary values. We provide the argument for $a'$, which is similar to that for $a$. If entries in $a'$ repeat, let $(i,j)$ be the lexicographically smallest pair such that $a'_{ij} = a'_{i,j+1}$. Without loss of generality and, for ease of notation, we assume that $i=1$. Define
\begin{equation}\label{eq:DCRLemma+}
\SSa:=\left\{ (x, \phi, s, z, \delta) \left| \;
\begin{aligned}
&\phi \leq \phi(s_{1n},\ldots,s_{dn}),\ s= G(z,\delta),\ (z,\delta ) \in (\ref{eq:Inc-1}) \\
&u(x) \leq s,\ (x,s_{\cdot n}) \in W	
\end{aligned}
\right.
\right\},
\end{equation}
where $W$ is a convex outer-approximation of $\{(x,s_{\cdot n})\mid s_{\cdot n} = f(x), x\in X\}$.  The set $\SSb$ is now obtained in a manner similar to $\SSa$ except that the construction is performed as if the underestimators $u_{ij}(x)$ and $u_{i,j+1}(x)$ are replaced with an underestimator $\max\{u_{ij}(x),u_{i,j+1}(x)\}$ which has a bound of $a'_{ij}$. Observe that this replacement can be done because $\max\{u_{ij}(x),u_{i,j+1}(x)\}$ is a valid underestimator of $f_i(x)$ and $\max\{u_{ij}(x), u_{i,j+1}(x)\}\le a'_{i,j+1}=a'_{ij}$, {\it i.e.}\/, $\bar{a}$ is obtained from $a'$ by projecting out $a'_{1,j+1}$. Therefore, $\SSb$ can be used to construct MICP reformulations as in Theorem~\ref{them:DCR+}. Let $t = \arg\max\{t'\mid \tau(i,t') < j+1\}$ and $\bar{a}_{i'j'} = a'_{i'j'}$ for $i'> 1$ or $j'\le j$ and $\bar{a}_{1j'}=a'_{1,j'+1}$ for $j'>j$. Similarly, let $\tau'(i',t') = \tau(i',t')$ for $i'>1$, $\tau'(1,t')=\tau(1,t')$ for $t'\le t$, and $\tau'(1,t')=\tau(1,t')-1$ for $t' > t$. Then, $\SSb$ is defined as follows:
\begin{equation}\label{eq:DCRLemmaRemoveReplicate}
\SSb:=\left\{ (x, \phi, s, z, \delta) \left| \;
\begin{aligned}
&\phi \leq \phi(s_{1,n-1},\ldots,s_{dn}),\ s= G'(z,\delta),\\
&1\ge z_{i'1}\ge\cdots\ge z_{i'n}\ge 0, \forall i' > 1\\ 
&1\ge z_{11}\ge\cdots\ge z_{1,n-1}\ge 0\\ &z_{i'\tau'(i',t')}\le \delta_{i't'}\le z_{i'\tau'(i',t')+1} \forall i', t=1,\ldots,l_{i'}-1\\
&u_{i'j'}(x) \leq s_{i'j'},\forall i' > 1, 0\le j'\le n\\
&u_{1j'}(x) \leq s_{1j'},\text{ for } j' = 0,\ldots,j,\\
&u_{1,j'+1}(x) \leq s_{1j'},\text{ for } j'=j,\ldots,n-1\\
&(x,s_{1,n-1},s_{2n},\ldots,s_{dn}) \in W	
\end{aligned}
\right.
\right\},
\end{equation}
where $G'_{i'j'}(z,\delta)= \bar{a}_{i'0} + \sum_{j''=1}^{j'} z_{ij''}(\bar{a}_{ij''}-\bar{a}_{i,j''-1})$. 
The main changes to the formulation of $\SSb$ relative to that for $\SSa$ are that (i) the bound $a'_{i,j+1}$ and the corresponding variable $s_{i,j+1}$ are dropped along with associated constraints, (ii) the constraint $s_{ij}\ge u_{i,j+1}(x)$ is added, (iii) $(x,s_{\cdot n})\in W$ is replaced with $(x,s_{1n-1},s_{2n},\ldots,s_{dn})\in W$, and (iv) $\phi \leq \phi(s_{1n},\ldots,s_{dn})$ is replaced with $\phi \leq \phi(s_{1,n-1},s_{2n},\ldots,s_{dn})$.
We argue that the relaxation \eqref{eq:DCR+} created using $\SSa$ projects to that created using $\SSb$. 
Let $(\tilde{x},\tilde{\phi},\tilde{s},\tilde{z},\tilde{\delta})\in \SSa$. Observe that, since $\tilde{s}=Z^{-1}\tilde{z}$, $a'_{ij}=a'_{i,j+1}$, and $\tilde{s}\ge \tilde{u}$, it follows that $\tilde{s}_{ij}=\tilde{s}_{i,j+1}\ge \max\{\tilde{u}_{ij}(x),\tilde{u}_{i,j+1}\}$. Let $p=(\tilde{x},\tilde{\phi},s',z',\tilde{\delta})$ be the vector obtained by projecting variables $(z_{i,j+1},s_{i,j+1})$ out of $(\tilde{x},\tilde{\phi},\tilde{s},\tilde{z},\tilde{\delta})$. Then, it is easy to see that $p\in \SSb$. In particular, since $s'_{i'n} = \tilde{s}_{i'n}$ for all $i' > 1$ and $s'_{1,n-1}=\tilde{s}_{1n}$, we have $\phi(s'_{1,n-1},s'_{2n},\ldots,s'_{dn}) = \phi(\tilde{s}_{1n},\tilde{s}_{2n},\ldots,\tilde{s}_{dn})\le \phi'$. On the other hand, let $(x',\phi',s',z',\delta')\in \SSb$. Let $K_i = \{0,\ldots,n\}\backslash\{j+1\}$ and recall that $t = \arg\max\{t'\mid \tau(i,t') < j+1\}$. Then, we define $(x',\phi',\tilde{s},\tilde{z},\delta')\in \SSa$, where $\tilde{z}_{i'}=z'_{i'}$ and $\tilde{s}_{i'} = s'_{i'}$ for $i'>1$, $\tilde{z}_{i,j+1}=\min\{\delta'_{it},z'_{ij}\}$, and $\tilde{s}_{i,j+1}=s'_{ij}$. Again,  $\phi'\le \phi(s'_{1,n-1},s'_{2n},\ldots,s'_{dn}) = \phi(\tilde{s}_{1,n-1},\tilde{s}_{2n},\ldots,\tilde{s}_{dn})$. It follows that  $\proj_{s,\phi}\SSa$ is affinely isomporhic to $\proj_{s,\phi}\SSb$ where the forward mapping is by projecting out $s_{i,j+1}$ and the reverse mapping is by setting. Observe that $\bar{a}$ has strictly fewer repetitions than in $a'$ which in turn has no more than $nd$ repetitions. Then, by induction, it follows that $\SSa$ models the same relaxation as if entries in $a'$ do not repeat and a maximum of the corresponding underestimators is used instead. Here onwards, whenever entries in the bound vector repeat, we will replace them with a single entry. In particular, if $a_{ij}=\cdots=a_{ij'}$, then we replace $(a_{ij''},u_{ij''})_{j''=j}^{j'}$ with $(a_{ij},\max_{j\le j''\le j'} u_{ij})$ instead.

Observe that the concave envelope of $\ephi$ at $u$ is computed by lifting $u$ to $s$ in $Q := \{\bar{s}\mid \bar{s}=Z(a)^{-1}(z), z\in \Delta\}$, where $Z(a)^{-1}$ maps $(z_1,\ldots,z_d)$ to $(\bar{s}_1,\ldots,\bar{s}_d)$ as below:
\begin{equation}\label{eq:Za_inv}
\bar{s}_{ij} = a_{i0}z_{i0}  + \sum_{k = 1}^j (a_{ ik} - a_{i k-1})z_{ik} \quad \text{for } j = 0, \ldots, n. 
\end{equation}
Similarly, the concave envelope of $\ephi$ at $u'$ is computed by lifting $u'$ to $s'$ in $Q':=\{\bar{s}\mid \bar{s}=Z(a')^{-1}(z), z\in \Delta\}$. Since $Q\ne Q'$, the comparison is not directly available. Instead, in the next results, we lift the vectors $(a,u)$ and $(a',u')$ so that they share the same bound vector $\breve{a}$. We remark that the underestimator $u'_{ij}$ is the same as $u_{ij}$ for all $j$ such that $a'_{ij} > a'_{i0}$.
This lifting allows us to compare, in Lemma~\ref{lemma:improvement-bdd} below, the concave envelopes over $Q$ and $Q'$ and show that replacing $(a,u)$ with $(a',u')$ helps tighten the relaxation. 

The proof of Lemma~\ref{lemma:improvement-bdd} requires that we can embed the concave envelopes over $Q$ and $Q'$ on faces of a higher-dimensional simplotope obtained using $\breve{a}$. This is achieved in Lemma~\ref{lemma:iso} by modifying the proof of Lemma 8 in~\cite{he2021new}. 
\begin{lemma}\label{lemma:iso}
Consider the Cartesian product of simplices $Q$ defined as in~(\ref{eq:Q-V}) with a vector $a \in \R^{d \times (n+1)}$. For  $K = (K_1,\ldots,K_d)$ where $K_i \subseteq\{0, \ldots, n \}$, let $F_{K}$ be the face of $Q$ defined as $F_{K} := \prod_{i=1}^d\conv\bigl(\{v_{ij}\bigm| j \in K_i \} \bigr)$ where $v_{ij} =(a_{i0}, \ldots,a_{ij-1}, a_{ij}, \ldots, a_{ij})$, and let $Q^K$ be the Cartesian product of simplices defined with $a_K: = (a_{1K_1}, \ldots, a_{dK_d})$, where $a_{iK_i} = (a_{ij})_{j \in K_i}$.  For each $i$, let $k^*_i$  denote the maximum  element in $K_i$. Let $\ephi_Q(s) = \phi(s_{1n},\ldots,s_{dn})$ and $\ephi_{Q_k}(s_K) = \phi(s_{1k_1^*},\ldots,s_{dk_d^*})$. Then, the hypograph, denoted as $\Phi^{F_K}$, of $\conc_{F_K}(\ephi_Q)$ is affinely isomorphic to that, denoted as $\Phi^{Q_K}$, of  $\conc_{Q_K}(\ephi_{Q_K})$.
\end{lemma}
\begin{compositeproof}
 	Consider an affine map $A_K: s_K \mapsto t$, where $t = (t_1, \ldots, t_d)$ such that 
\begin{equation}\label{eq:Q-F-iso}
t_{ij} = \begin{cases}
 s_{ij} & j \in K_i\backslash \{0,n\} \\
	a_{ij}  & j =0 \\
	s_{ik_i^*} & j = n \\
	(1-\gamma_{ij})s_{il(i,j)} + \gamma_{ij}s_{ir(i,j)} & j \notin K_i \cup \{0\} \cup \{n\} ,
\end{cases}
\end{equation}
where $l(i,j) = \max\{j' \in K_i \cup \{0\} \cup \{n\} \mid j'\leq j\}$, $r(i,j) = \min \{j' \in  K_i \cup \{0\} \cup \{n\} \mid j' \geq j \}$, and $\gamma_{ij} = (a_{ij} - a_{il(i,j)})/(a_{ir(i,j)} - a_{il(i,j)})$. It follows from the second statement of Proposition 5 in~\cite{he2021new} that $t_i$ belongs to the face $F_{K_i\cup\{0\}\cup\{n\}}$ of $Q_i$. Moreover, if $0\not\in K_i$ (resp. $n_i \not\in K_i$), we have
$\frac{t_{i1}-t_{i0}}{a_{i1}-a_{i0}}=1$ (resp. $\frac{t_{in}-t_{in-1}}{a_{in}-a_{in-1}}=0$), showing that $t$ belongs to the face $F_K$ of $Q$. The inverse of $A_K$ is defined as $s \mapsto s_K$ and maps the face $F_K$ into the polytope $Q_{K}$. This is because, for any $s_i \in \vertex(F_{iK_i})$, there exists a $k \in K_i$ such that $s_{ij} = \min\{a_{ij}, a_{ik}\}$ for all $j \in K_i$. Thus, $s_{iK_i} \in Q_i$. Consider the affine transformation, $\Pi_K$, defined as $(s_K, \phi) \mapsto \bigl(A_K(s_K), \phi\bigr)$ and its inverse, $\Pi^{-1}_K$, $(s,\phi) \mapsto (s_K, \phi)$. Note that, in calling the projection operation as an inverse of $\Pi_K$, we are interpreting $\Pi_K$ as a transformation into the affine hull of $\Phi^{F_K}$. For all $s\in F_K$, since $s_{in}=s_{ik_{i}^*}$, it follows that $\ephi_Q(s) = \ephi_{Q_k}(s_K)$. Therefore, $(s,\phi)$ belongs to the hypograph of $\ephi_Q$ over $F_K$ if and only if $(s_K,\phi)$ is in the hypograph of $\ephi_{Q_k}$. In other words, $\Pi_K \Phi^{Q_K} = \Phi^{F_K}$ and $\Pi^{-1}_K\Phi^{F_K} = \Phi^{Q_K}$. Therefore, $\conv(\Phi^{F_K}) = \conv(\Pi_K\Phi^{Q_K}) = \Pi_K\conv(\Phi^{Q_K})$ and, similarly, $\conv(\Phi^{Q_K}) = \Pi^{-1}_K\conv(\Phi^{F_K})$, completing the proof. \Halmos
\end{compositeproof}

The next lemma describes the actual lifting operation, whose purpose is to add bounds and underestimators in \eqref{eq:ordered-oa} so that the concave envelope of $\phi$ obtained with fewer bounds resides as the concave envelope over a face in the larger system. Let $\PPolytope(n') = \bigl\{(a',u')\bigm| u'_{i0} = a'_{i0}\le \cdots\le a'_{in'_{i}}, u'_{ij}\le \min\{u'_{i{n'_i}}, a'_{ij}\}\forall (i,j) \bigr\}$. We simply write $\PPolytope$ when $n'_i = n$ for all $i$. We denote the discrete univariate function that maps $a_{ij}$ to $u_{ij}$ for $j=1,\ldots,n_i$ as $\xi_{i,a_i}(\cdot;u_i)$.
\begin{lemma}\label{lemma:lifta}
Let $\baa_{iK_i} = \ba_i$, $\bn_i=|K_i|-1$, and $\ban_i=|\baa_i|-1$. Assume $(\ba,u^{\ba})\in \PPolytope(\bar{n})$ and $(\baa,u^{\baa})\in \PPolytope(\breve{n})$ and that $\ba$ and $\baa$ are strictly increasing. We lift $(\ba,u^{\ba})$ to $(\baa,u^{\baa})$ by defining $u^{\baa}_{ij}$ as follows:
\begin{equation}\label{eq:defineu}
u^{\baa}_{ij} = \left\{\begin{alignedat}{2}
&u^{\ba}_{i,r-1} &\quad& j\text{ is the $r^{th}$ largest entry in } K_i\\
&\baa_{i0}-\epsilon && j\not\in K_i, j > 0, j < \ban_i \\
&\baa_{i0} && j\not\in K_i, j=0\\
&u^{\ba}_{i,\bn_i} &&j\not\in K_i, j=\ban_i.
\end{alignedat}\right.  
\end{equation}
Let $Q_{\ba}$ (resp. $Q_{\baa}$) be the simplex defined with $\ba$ (resp. $\baa$). Let $s^{\baa}_i = \conc(\xi_{i,\baa})(\cdot;u^{\baa}_i)$ and $s^{\ba}_i = \conc(\xi_{i,\ba})(\cdot;u^{\ba}_i)$, where $\xi_{i,\ba}(y;u^{\ba}_i)=u^{\ba}_{ij}$ if $y=\ba_{ij}$. Then, $\conc_{Q_{\baa}}(\ephi)(s^{\baa}) =
\conc_{Q_{\ba}}(\ephi)(s^{\ba})$.
\end{lemma}
\begin{proof}
First, we show that $\conc(\xi_{i,\baa_i})(y;u^\baa_i) = h_i(y)$ for every $y \in [\baa_{i0}, \baa_{i\breve{n}_i}]$, where
\begin{equation}\label{eq:hy}
    h_i(y) = \left\{\begin{alignedat}{2}
    &y &\quad& y\le \ba_{i0}\\
    &\conc(\xi_{i,\ba})(y;u^{\ba}_i) && y\in [\ba_{i0},\ba_{i\bn_i}]\\
    &u^{\ba}_{i\bn_i} && y\ge \ba_{i\bn_i}.\\
    \end{alignedat}\right.
\end{equation}
To begin, we show that $h_i(y)$ is concave and $h_i(\baa_{ij})\ge u^{\baa}_{ij}$, where the inequality is strict if $j\not\in K_i\cup\{0\}\cup\{\ban_i \}$. The concavity of $h_i$ follows since $1\ge \frac{s^{\ba}_{i1}-s^{\ba}_{i0}}{\ba_{i1}-\ba_{i0}}\ge \frac{s^{\ba}_{i\bn_i}-s^{\ba}_{i,\bn_i-1}}{\ba_{i,\bn_i}-\ba_{i,\bn_i-1}}\ge 0$. We now show that $h_i(\baa_{ij})\ge u^{\baa}_{ij}$. If $j\in K_i$, the result follows since $h_i(\baa_{ij} ) = \conc(\xi_{i,\ba})(\baa_{ij} ;u^{\ba}_i)$ and the concave envelope construction used a point $(\baa_{ij}, u^{\baa}_{ij})$. Now, let $j\not\in K_i$. If there does not exist $j'\in K_i$ such that $j' < j$, then $h_i(\baa_{ij}) = \baa_{ij}\ge \baa_{i0}$, where the last inequality is strict if $j > 0$. Now, assume that there exists $j'\in K_i$ such that $j' < j$ and $j<\ban_i$. Then, $h_i(\baa_{ij})\ge \bar{a}_{i0} > u^{\baa}_{ij}$. Finally, if $j=\ban_i$, we have $h_i(\baa_{ij})=u^{\ba}_{i\bn_i}=u^{\baa}_{ij}$. Clearly, $\conc(\xi_{i,\baa})(\cdot;u^{\baa}_i)\le h_i(\cdot)$ since $h_i(\cdot)$ is concave and overestimates the left-hand-side. Combining this with the observation that $\conc(\xi_{i,\baa})(\cdot,u^{\baa})$ is constructed by including all points that are extremal in the hypograph of $h_i(y)$, it follows that $\conc(\xi_{i,\baa})(\cdot;u^{\baa}_i)= h_i(\cdot)$.

Consequently, $s^{\baa}_{iK_i} = s^{\ba}_i$. Moreover, we can conclude that $s^{\baa} = A_K(s^{\ba})$ and $s^{\ba} = A^{-1}_K(s^{\baa})$, where $A_K$ is as defined in~(\ref{eq:Q-F-iso}).  Therefore,  by Lemma~\ref{lemma:iso}, $s^\baa$ belongs to the face $F_K$ of $Q_{\breve{a}}$, and $\conc_{Q_{\baa}}(\ephi)(s^\baa) = \conc_{F_K}(\ephi)(s^\baa) = \conc_{Q_{\ba}}(\ephi)(s^\ba)$. \Halmos

%
\end{proof}

Finally, we are ready to compare the concave envelope of $\ephi$ over $Q$ at $u$ with that over $Q'$ at $u'$. 
\begin{lemma}~\label{lemma:improvement-bdd}
Consider vectors $(a,u)$ and $(a',u')$ in  $\PPolytope$. Assume that 
\begin{enumerate}
    \item $a'_{i0}\ge a_{i0}$ and whenever $a'_{ij} > a'_{i0}$, we have $a'_{ij} \le a_{ij}$;
    \item $u'_{ij}=u_{ij}$ for all $j$ such that $a_{ij} > a'_{i0}$ or $j=n$.
\end{enumerate}
Let $Q$ (resp. $Q'$) be the Cartesian product of simplices defined as in~(\ref{eq:Q-V}) with respect to $a$ (resp. $a'$). Then, $\conc_{Q'}(\ephi)(s') \leq \conc_{Q}(\ephi)(s)$, where for $i$ and $j$, $s'_{ij} = \conc(\xi_{i,a'_i})\bigl(a'_{ij};u'_i\bigr)$ and $s_{ij} = \conc(\xi_{i,a_i})\bigl(a_{ij};u_i\bigr)$.
\end{lemma}
\begin{proof}
We construct $\baa$ by merge sorting $a$ and $a'$ and removing any replicates. We lift $(a,u)$ (resp. $(a',u')$) into $(\baa,\breve{u})$ (resp. $(\baa', \breve{u}')$) using the transformation \eqref{eq:defineu} in Lemma~\ref{lemma:lifta} and compute $\breve{s}'$ (resp. $\breve{s}$) so that $\breve{s}'_{ij} = \conc(\xi_{i,\breve{a}_i})\bigl(\breve{a}_{ij};\breve{u}'_i\bigr)$ and $\breve{s}_{ij} = \conc(\xi_{i,\breve{a}_i})\bigl(\breve{a}_{ij};\breve{u}_i\bigr)$. Throughout, wherever we use Lemma~\ref{lemma:lifta} to lift, say $(\bar{a},\bar{u})$ to $(\breve{a},\tilde{u})$, we use the argument at the beginning of the proof, to combine repeated entries in $\bar{a}$ and correspondingly modify the $\bar{u}$ coordinate so that it is the maximum among for the repeated entries. Then, we will show that 
\begin{enumerate}
    \item\label{itm:equalconv} $\conc_{Q'}(\ephi)(s') = \conv_{Q^\baa}(\ephi)(\breve{s}')$ and $\conc_{Q}(\ephi)(s') = \conv_{Q^\baa}(\ephi)(\breve{s})$,
    \item\label{itm:sincrease} $\breve{s}_i\le \breve{s}'_i$ for all $i$,
    \item\label{itm:sendequal} $\breve{s}_{i\breve{n}_i} =  \breve{s}'_{i\breve{n}_i}$ for all $i$.
\end{enumerate}
With these relations, the result follows easily using the following argument:
\begin{equation*}
    \conc_{Q}(\ephi)(s) =  \conc_{Q_\baa}(\ephi)(\breve{s} ) \le \conc_{Q_\baa}(\ephi)(\breve{s}') = \conc_{Q'}(\ephi)(s'),
\end{equation*}
where the inequality follows from Proposition 8 in~\cite{he2021new} since $\breve{s}_i \leq \breve{s}'_i$ with $\breve{s}_{i\breve{n}_i} = \breve{s}'_{i\breve{n}_i}$. 

Now, we prove the claimed Items~\ref{itm:equalconv}-\ref{itm:sendequal}. Item~\ref{itm:equalconv} follows directly from Lemma~\ref{lemma:lifta}. Item~\ref{itm:sendequal} follows since $\breve{s}_{i\breve{n}_i} = u_{in} = u'_{in}=\breve{s}'_{i\breve{n}_i}$, where the first and last equality follow since $a_{in}$ (resp $a'_{in}$) are the largest coordinates and the concave envelope and, by \eqref{eq:hy} and \eqref{eq:defineu}, $\breve{s}_{i\breve{n}_i}=\conc(\xi_{i,\breve{a}_i})\bigl(a_{in};\breve{u}_i\bigr) = \breve{u}_{i\breve{n}_i} = u_{in}$ (resp. $\breve{s}'_{i\breve{n}_i}=\conc(\xi_{i,\breve{a}_i})\bigl(a'_{in};\breve{u}'_i\bigr) = \breve{u}'_{i\breve{n}_i} = u_{in}$). Finally, we show Item~\ref{itm:sincrease}. To see this, we obtain $(a',u')$ from $(a,u)$ in a series of steps. 
Assume at the beginning of the $k^{\text{th}}$ step we have $(a^{k-1},u^{k-1})$. Then, we find the lexicographically minimum $(i,j)$ such that $(a^{k-1}_{ij},u^{k-1}_{ij})\ne (a'_{ij},u'_{ij})$ and construct $(a^k,u^k)\in \PPolytope$ so that it satisfies a few conditions. First, the set of unmatched pairs $(i',j')$ where $(a^{k},u^{k})$ differs from $(a',u')$ strictly reduces and, in particular, $(a^{k}_{ij},u^{k}_{ij}) = (a'_{ij},u'_{ij})$. Second, $(a^k,u^k)$ is constructed so that $u^k_{ij'} = u_{ij'} = u'_{ij'}$ when $j' > 0$. Third, if $(a^{k}_{i0},u^k_{i0}) = (a'_{i0},u'_{i0})$ then $(a^{k}_{ij'},u^k_{ij'}) = (a'_{ij'},u'_{ij'})$ for all $j'$ such that $a_{ij'} < a'_{i0}$.
As before, we lift each $(a^r,u^r)$ into $(\baa,\breve{u}^r)$ using the transformation \eqref{eq:defineu} in Lemma~\ref{lemma:lifta}. For each $r$ and $i'$, let $h^r_{i'}(\cdot) = \conc(\xi_{i',\breve{a}_{i'}})\bigl(\cdot;\breve{u}^r_{i'}\bigr)$ be as given in \eqref{eq:hy}. We only need to show that the function $h^r_{i'}(\cdot)$ is non-decreasing in $r$.  Since the number of unmatched pairs decreases the process converges in finite number of steps, say $k'$. Then, we have $\breve{s}_{i'} = h^0_{i'}(\breve{a}) \le h^{k'}_{i'}(\breve{a}) = \breve{s}'_{i'}$.

Now, we define $(a^k_{i'},u^k_{i'}) = (a^{k-1}_{i'},u^k_{i'})$ for $i' \ne i$.  First, consider the case $j=0$. Let $(a^k_{i0},u^k_{i0}) = (a'_{i0},a'_{i0})$, for all $j'>0$ such that $a_{ij'}\le a'_{i0}$,  $(a^k_{ij'},u^k_{ij'}) = (a'_{i0},u'_{ij'})$, and $(a^k_{ij'},u^k_{ij'}) = (a^{k-1}_{ij'},u^{k-1}_{ij'})$ otherwise. Since $j=0$, we have $(a^{k-1}_{i0},u^{k-1}_{i0}) = (a_{i0},a_{i0})$. Since Lemma~\ref{lemma:lifta} adds a point $(a_{i0},a_{i0})$ during the lifting, we have $h^k_i(a^{k-1}_{i0})  = h^k_i(a_{i0}) \ge a_{i0} = u^{k-1}_{i0}$. Moreover, if $j' > 0$ and $a_{ij'}\le a'_{i0}$ we have $h^k_i(a^{k-1}_{ij'}) \ge (1-\lambda)a_{i0} + \lambda a'_{i0} = a^{k-1}_{ij'}\ge u^{k-1}_{ij'}$, where $\lambda = \frac{a^{k-1}_{ij'}-a_{i0}}{a'_{i0}-a_{i0}}$. The first inequality is because $h^r_i(\cdot)$ a concave function whose epigraph contains $(a_{i0},a_{i0})$ and $(a'_{i0},a'_{i0})$, and the second inequality is because $(a^{k-1},u^{k-1})\in \PPolytope$. For $j'$ such that $a_{ij} > a'_{i0}$, we have $h^k_i(a^{k-1}_{ij'}) =h^k_i(a^k_{ij'}) \ge u^k_{ij'} = u^{k-1}_{ij'}$. Therefore, it follows that $h^k_i(\cdot)\ge h^{k-1}_i(\cdot)$. Moreover, it is easy to verify that $(a^k,u^k)$ satisfies the conditions so that unmatched pairs reduce, $(i,j)$ is no longer an unmatched pair, and $(i,j')$ is not unmatched if $a_{ij'} < a'_{i0}$. We argue that $(a^k,u^k)\in \PPolytope$. Clearly, $a^k_i$ is sorted and $u^k_{i0} = a^k_{i0}$. Now, consider $j'>0$. If $a_{ij'}\le a'_{i0}$, we have $u^k_{ij'}=u'_{ij'}\le \min\{u'_{in},a'_{ij'}\} =\min\{u^{k}_{in}, a^k_{ij'}\}$. Otherwise, $u^k_{ij'}=u^{k-1}_{ij'}\le \min\{u^{k-1}_{in}, a^{k-1}_{ij'}\} = \min\{u^{k}_{in}, a^k_{ij'}\}$. 

Now, assume that $j>0$. By construction, $j$ must be such that $a_{ij} > a'_{i0}$. Then, we construct $(a^k_{ij},u^k_{ij}) = (a'_{ij},u'_{ij})$ and $(a^k_{i'j'},u^k_{i'j'}) = (a^{k-1}_{ij},u^{k-1}_{ij})$ otherwise. Then,  for $j'\ne j$, we have $h^k_i(a^{k-1}_{ij'}) = h^k_i(a^k_{ij'}) \ge u^k_{ij'} = u^{k-1}_{ij'}$. Also, $h^k_i(a^{k-1}_{ij}) \ge h^k_i(a^k_{ij})\ge u^k_{ij} = u^{k-1}_{ij}$, where the first inequality is because $h^k_i(\cdot)$ is non-decreasing by Lemma~5 of \cite{he2021new} and $a^k_{ij} = a'_{ij} \le a_{ij} = a^{k-1}_{ij}$. The last equality is because of the hypothesis that $u_{ij}=u'_{ij}$. To see that $(a^k,u^k)\in \PPolytope$, observe that $a^k_{ij} = a'_{ij} \ge a'_{i,j-1} = a^{k-1}_{i,j-1} = a^{k}_{i,j-1}$, where the inequality is by the definition of $a'$, and the second equality is because $(i,j)$ is the lexicographically minimum unmatched pair. Further, $u^k_{ij} = u'_{ij} \le \min\{u'_{in},a'_{ij}\} = \min\{u^k_{in}, a^k_{ij}\}$, where the inequality is because $(a',u')\in \PPolytope$. For $j'\ne j$, we have $u^k_{ij'}=u^{k-1}_{ij'}\le \min\{u^{k-1}_{in},a^{k-1}_{ij'}\} = \min\{u^k_{in}, a^k_{ij'}\}$, where the inequality is because $(a^{k-1},u^{k-1})\in \PPolytope$. \Halmos
\end{proof}

Then, as in the proof of Theorem~\ref{them:DCR}, it is easy to show that given a point $(\bar{x},\bar{\phi},\bar{u})$, such that $\bar{\phi} = \phi\circ f (\bar{x})$ and $\bar{u} = u(\bar{x})$, there is a point $(\bar{x},\bar{\phi},\bar{s},\bar{z},\bar{\delta})\in \SSb$ where the $u$ in constraints \eqref{eq:DCRLemmaRemoveReplicate} can be chosen to be $\bar{u}$. Then, as described at the beginning of the proof, this point can be lifted to belong to $\SSa$ showing that \eqref{eq:DCR+} is an MICP relaxation of the hypograph of $\phi\circ f$. Then, after removing repeated entries in $a'$ as described above, the proof of Proposition~\ref{prop:eval} shows that $\phi_{{\cal H}_+}(\bar{x},\bar{f})=\conc_{Q'}(\bar{\phi})(s')$.  
Also, Proposition~\ref{prop:eval} shows that $\phi_{{\cal H}}(\bar{x},\bar{f}) = \conc_{Q}(\ephi)(s)$ where $s_{ij} =\conc(\xi_{i,a_i})(a_{ij};u_i^*)$, where $u_i^*$ is defined therein. Observe that if $\delta_{ik}=1$ and $\delta_{ik+1}=0$, we have $a'_{i0} = a_{i\tau(i,k)}$ because $b_{i0k} = a_{i\tau(i,k)} - a_{i\tau(i,k-1)}$. Since $(a,u^*)$ and $(a',\tilde{u})$ satisfy the hypothesis in Lemma~\ref{lemma:improvement-bdd}, it shows that $\phi_{{\cal H}_+}(\bar{x},\bar{f}) \le \phi_{{\cal H}}(\bar{x},\bar{f})$.\Halmos
\begin{OldVersion}
=========================

\begin{lemma}\label{lemma:lifta}
Let $\baa = (\baa_1, \ldots,\baa_d) \in \R^{\sum_{i=1}^d(\breve{n}_i+1)}$ and $\ba = (\ba_1, \ldots, \ba_d) \in  \R^{d \times (n+1)}$ so that for each $i$ there exists an index set $K_i \subseteq \{0, \ldots, \breve{n}_i\} $ such that $\baa_{iK_i} = \ba_i$.  Let  $P_{\ba}$ (resp. $P_{\baa}$) be the polytope defined as in in~(\ref{eq:P-3}) with $a=\ba$ (resp. $a = \baa$). For $\bar{u} \in  P_{\ba}$, we modify $\bar{u}$ to construct $\breve{u}$ in $P_{\baa}$, where $\breve{u} = (\breve{u}_1, \ldots, \breve{u}_d)$ such that for each $i$, $\breve{u}_{iK_i} =  \bar{u}_{i}$ and, for $j \notin K_i$,
\begin{equation*}
\breve{u}_{ij} = \left\{\begin{alignedat}{2}
&\baa_{i0} && \quad\text{if } j=0\\
&\bar{u}_{in} && \quad\text{if } j= \breve{n}_i \\
&\baa_{i0}-\epsilon && \quad\text{if } j > 0, j < \breve{n}_i .
\end{alignedat}\right.  
\end{equation*}
For each $i$, let $\breve{s}_{ij} = \conc(\xi_{i,\baa_i})(\baa_{ij} ;\breve{u}_i)$ for  $j \in \{0, \ldots, \breve{n}_i \}$ and $\bar{s}_{ik} = \conc(\xi_{i,\ba_i})(\ba_{ik};\bar{u}_i)$ for $k \in \{0, \ldots, n\}$. Then, $\conc_{Q_{\baa}}(\ephi)(\breve{s}) =
\conc_{Q_{\ba}}(\ephi)(\bar{s})$, where $Q_{\ba}$ (resp. $Q_{\baa}$) be the Cartesian product of simplices defined with $\ba$ (resp. $\baa$).
\end{lemma}
\begin{proof}
We start with showing that $\breve{s} = A_K(\bar{s})$ and $\bar{s} = A^{-1}_K(\breve{s})$, where $A_K$ is defined as in~(\ref{eq:Q-F-iso}). To prove this, it suffices to show that $\conc(\xi_{i,\baa_i})(y;\breve{u}_i) = h_i(y)$ for every $y \in [\baa_{i0}, \baa_{i\breve{n}_i}]$, where
\begin{equation*}
    h_i(y) = \left\{\begin{alignedat}{2}
    &y &\quad& y\le \ba_{i0}\\
    &\conc(\xi_{i,\ba_i})(y;\bar{u}_i) && y\in [\ba_{i0},\ba_{in}]\\
    &\bar{u}_{in} && y\ge \ba_{in}.\\
    \end{alignedat}\right.
\end{equation*}
Clearly, we have $h_i(y) \leq \conc(\xi_{i,\baa_i})(y;\breve{u}_i)$ since the envelope is constructed by including all extremal in the hypograph of $h_i(y)$. To show $h_i(y) \geq \conc(\xi_{i,\baa_i})(y;\breve{u}_i)$, it suffices to show that $h_i(y)$ is concave and $h(\baa_{ij})\ge \breve{u}_{ij}$. The concavity of $h_i$ follows from $1\ge \frac{\bar{s}_{i1}- \bar{s}_{i0}}{\ba_{i1}-\ba_{i0}}\ge \frac{\bar{s}_{in}- \bar{s}_{i,n-1}}{\ba_{in}-\ba_{i,n-1}}\ge 0$. Next, we show  $h_i(\baa_{ij})\ge \breve{u}_{ij}$, where the inequality is strict if $j \notin  K_i \cup \{0\} \cup \{\breve{n}_i\}$. If $j\in K_i$, the result follows since $h_i(\baa_{ij}) = \conc(\xi_{i,\ba_i})(\baa_{ij};\bar{u}_i)$ and the concave envelope construction used the point $(\baa_{ij}, \breve{u}_{ij})$. So, let $j\not\in K_i$. If there does not exist $j'\in K_i$ such that $j' < j$, then $h_i(\baa_{ij}) = \baa_{ij}\ge \baa_{i0}$, where the last ineqaulity is strict if $j > 0$. Now, assume that there exist $j'\in K_i$ such that $j' < j$ and $j<\breve{n}_i$. Then, $h_i(\baa_{ij})\ge \bar{a}_{i0} > \breve{u}_{ij}$. Finally, if $j= \breve{n}_i$, we have $h_i(\baa_{ij})=\bar{u}_{in}= \breve{u}_{ij}$. 

Now, by Lemma~\ref{lemma:iso}, $\breve{s}$ belongs to the face $F_K$ of $Q_{\breve{a}}$. Therefore, we obtain $\conc_{Q_{\baa}}(\ephi)(\breve{s}) = \conc_{F_K}(\ephi)(\breve{s}) = \conc_{Q_{\ba}}(\bar{s})$, where the second equality follows from Lemma~\ref{lemma:iso}.  \Halmos
\end{proof}

\begin{lemma}\label{lemma:moveu}
Let $u_i$ be a vector in $P_i \subseteq \R^{n+1}$ defined as in~(\ref{eq:P-3}) with vector $a_i$, and  consider two distinct indices $j',j''$ such that $j'<j''$, $j' \neq 0$ and $u_{ij''} \leq a_{ij'}$. Let $\hat{u}_i$ be defined so that $\hat{u}_{ij} = u_{ij}$ for $j \notin \{j', j''\}$, $\hat{u}_{ij'} = \max\{u_{ij'}, u_{ij''}\}$, and  $\hat{u}_{ij''} = a_{i0} - \epsilon$ if $j'' < n_i$ and  $\hat{u}_{ij''} = u_{in}$ if $j'' = n$. Then, $\conc(\xi_{i,a_i})(y;\hat{u}_i)\ge \conc(\xi_{i, a_i})(y;u_i)$ where the equality holds when $y = a_{in}$.
\end{lemma}
\begin{compositeproof}
Let $h(\cdot) = \conc(\xi_{i,a_i})(\cdot;\hat{u}_i)$. It suffices to show that  $h(a_{ij})\ge u_{ij}$ for $j \in \{j',j''\}$. This follow from $h(a_{ij}) \ge h(a_{ij'})\ge \hat{u}_{ij'} = \max\{u_{ij'},u_{ij''}\}\ge \hat{u}_{ij}$ for $j \in \{j',j''\}$, where the first inequality is because $h(\cdot)$ is non-decreasing by Lemma 5 of~\cite{he2021new}, the second inequality is by concave envelope construction, and the equality and the last inequality hold by definition. \Halmos
\end{compositeproof}

Here, we prove Lemma~\ref{lemma:improvement-bdd}. First, we construct $\baa$ by merge sorting $a$ and $a'$ and removing any replicates. Then, we lift $(a,u)$ (resp. $(a',u)$) to $(\baa, \breve{u})$ (resp. $(\baa, \breve{u}')$) as in Lemma~\ref{lemma:lifta}. Let $\breve{s}_{ij} = \conc(\xi_{i,\baa_i})(\baa_{ij};\breve{u}_{i})$, and $\breve{s}'_{ij} = \conc(\xi_{i,\baa_i})(\baa_{ij};\breve{u}'_{i})$. It follows from Lemma~\ref{lemma:lifta} that $\conc_{Q_\baa}(\ephi)(\breve{s}) = \conc_{Q}(\ephi)(s)$ and $\conc_{Q_\baa}(\ephi)(\breve{s}') = \conc_{Q'}(\ephi)(s')$. Next, we show that $\breve{s}_i \leq \breve{s}'_i$ with the equality holds for the last coordinate. To do so, we carry out the following procedure. Initialize $k=0$, $a^0_i=a_i$, and let $u^0_i$ be obtained by lifting $(a_i,u_i)$ to $(\baa_i,u^0_i)$ as in Lemma~\ref{lemma:lifta}. Let $J^k=\{j\mid a'_{ij} < a^k_{ij}\}$. As long as $J^k$ is not empty, let $\bar{j}$ be the smallest index in $J^k$. Then, there exist $j', j''$ with $j' < j'' \neq 0$  such that $\baa_{ij''} = a^k_{i\bar{j}}$ and $\baa_{ij}= a'_{i\bar{j}}$. Moreover, $u^k_{ij''} = u_{i\bar{j}} \leq a'_{i\bar{j}} =\baa_{ij'}$. Therefore, using the construction in Lemma~\ref{lemma:moveu}, we obtain $u_i^{k+1}$ and observe that $\conc(\xi_{i,\baa_i})(\cdot;u^{k+1}_i)\ge \conc(\xi_{i,\baa_i})(\cdot;u^{k}_i)$. Let $a^{k+1}_{ij} = a_{ij}$ for $j\neq \bar{j}$ and $a^{k+1}_{i\bar{j}} = a'_{i\bar{j}}$, increase $k$ by $1$ and iterate.  Since $\bar{j} \in J^k \setminus J^{k+1}$ implies $|J^{k+1}|<|J^{k}|$, it follows that the procedure terminates in $t$ iterations. Now, observe that $\breve{s}_{ij}$ (resp. $\breve{s}'_{ij}$) is obtained by evaluating $\conc(\xi_{i\baa_i})(\cdot;u^{0})$ (resp. $\conc(\xi_{i\baa_i})(\cdot;u^{t})$) at $\baa_{ij}$. Therefore,  $\breve{s}_i \leq \breve{s}'_i$ with the equality holds for the last coordinate. Last, we complete the proof by observing that
\begin{equation*}
    \conc_{Q}(\ephi)(s) =  \conc_{Q_\baa}(\ephi)(\breve{s} ) \le \conc_{Q_\baa}(\ephi)(\breve{s}') = \conc_{Q'}(\ephi)(s'),
\end{equation*}
where the inequality follows from Proposition 8 in~\cite{he2021new}. \Halmos

\end{OldVersion}
\subsection{Proof of Proposition~\ref{prop:MIP-phi-lambda-log}}\label{app:MIP-phi-lambda-log}

Let $E$ denote the above formulation. As mentioned above, \eqref{eq:SOS2-log} is an ideal formulation for $\lambda_i\in \Lambda_i$, where $\lambda_i$ additionally satisfy SOS2 constraints \cite{vielma2011modeling}. Then, the validity of $E$ follows directly from Proposition~\ref{prop:MIP-phi-z-simu} by considering the transformation $T$ that maps $\Delta$ to $\Lambda$ defined as $\lambda_{ij} = z_{ij} - z_{ij+1}$ for $j \leq n-1$ and $\lambda_{in} = z_{in}$, and its inverse $T^{-1}$ defined as $z_{ij} = \sum_{j' = j}^n \lambda_{ij'}$ for all $j$. Similar to the ideality proof of Theorem~\ref{them:MIP-phi-z}, the ideality of the formulation can be established by applying Lemma~\ref{lemma:ideal} recursively because, by assumption, for every extreme point $(\lambda, \theta) $ of the convex hull of $\hypo(\theta \mcirc A)$, we have that $\lambda \in \vertex(\Lambda)$, which is a subset of $\{0,1\}^{d \times (n+1)}$,
 and, for each $i$ and $k$, the inequality $\sum_{j \notin L_{ik}} \lambda_{ij} \leq 1 - \sum_{j \notin R_{il}} \lambda_{ij}$ in~(\ref{eq:SOS2-log}) defines a face of $\Lambda$.\Halmos

\bibliographystyle{plain} 
\bibliography{arxSubmit}	

\begin{thebibliography}{10}

\bibitem{anstreicher2010computable}
Kurt~M Anstreicher and Samuel Burer.
\newblock Computable representations for convex hulls of low-dimensional
  quadratic forms.
\newblock {\em Mathematical Programming}, 124(1-2):33--43, 2010.

\bibitem{balas1998disjunctive}
Egon Balas.
\newblock Disjunctive programming: Properties of the convex hull of feasible
  points.
\newblock {\em Discrete Applied Mathematics}, 89(1-3):3--44, 1998.

\bibitem{beale1970special}
Evelyn Martin~Lansdowne Beale and John~A Tomlin.
\newblock Special facilities in a general mathematical programming system for
  non-convex problems using ordered sets of variables.
\newblock {\em OR}, 69(447-454):99, 1970.

\bibitem{belotti2009branching}
Pietro Belotti, Jon Lee, Leo Liberti, Francois Margot, and Andreas W{\"a}chter.
\newblock Branching and bounds tightening techniques for non-convex
  \text{MINLP}.
\newblock {\em Optimization Methods \& Software}, 24(4-5):597--634, 2009.

\bibitem{benson2004concave}
Harold~P Benson.
\newblock Concave envelopes of monomial functions over rectangles.
\newblock {\em Naval Research Logistics}, 51(4):467--476, 2004.

\bibitem{burer2015gentle}
Samuel Burer.
\newblock A gentle, geometric introduction to copositive optimization.
\newblock {\em Mathematical Programming}, 151(1):89--116, 2015.

\bibitem{crama1989recognition}
Yves Crama.
\newblock Recognition problems for special classes of polynomials in 0--1
  variables.
\newblock {\em Mathematical Programming}, 44(1):139--155, 1989.

\bibitem{dantzig1960significance}
George~B Dantzig.
\newblock On the significance of solving linear programming problems with some
  integer variables.
\newblock {\em Econometrica, Journal of the Econometric Society}, pages 30--44,
  1960.

\bibitem{de2010triangulations}
J.~De~Loera, J.~Rambau, and F.~Santos.
\newblock {\em Triangulations: Structures for Algorithms and Applications}.
\newblock Algorithms and Computation in Mathematics. Springer Berlin
  Heidelberg, 2010.

\bibitem{gounaris2008tight}
Chrysanthos~E Gounaris and Christodoulos~A Floudas.
\newblock Tight convex underestimators for $\mathcal{C}^2$-continuous problems:
  Ii. multivariate functions.
\newblock {\em Journal of Global Optimization}, 42(1):69--89, 2008.

\bibitem{gupte2013solving}
Akshay Gupte, Shabbir Ahmed, Myun~Seok Cheon, and Santanu Dey.
\newblock Solving mixed integer bilinear problems using milp formulations.
\newblock {\em SIAM Journal on Optimization}, 23(2):721--744, 2013.

\bibitem{he2021new}
Taotao He and Mohit Tawarmalani.
\newblock A new framework to relax composite functions in nonlinear programs.
\newblock {\em Mathematical Programming}, 190(1):427--466, 2021.

\bibitem{he2021tractable}
Taotao He and Mohit Tawarmalani.
\newblock Tractable relaxations of composite functions.
\newblock {\em Mathematics of Operations Research}, 2021.

\bibitem{horn1994topics}
Roger~A Horn, Roger~A Horn, and Charles~R Johnson.
\newblock {\em Topics in matrix analysis}.
\newblock Cambridge university press, 1994.

\bibitem{huchette2019combinatorial}
Joey Huchette and Juan~Pablo Vielma.
\newblock A combinatorial approach for small and strong formulations of
  disjunctive constraints.
\newblock {\em Mathematics of Operations Research}, 44(3):793--820, 2019.

\bibitem{ibaraki1976integer}
Toshimde Ibaraki.
\newblock Integer programming formulation of combinatorial optimization
  problems.
\newblock {\em Discrete Mathematics}, 16(1):39--52, 1976.

\bibitem{jeroslow1984modelling}
Robert~G Jeroslow and James~K Lowe.
\newblock Modelling with integer variables.
\newblock In {\em Mathematical Programming at Oberwolfach II}, pages 167--184.
  Springer, 1984.

\bibitem{locatelli2014convex}
Marco Locatelli and Fabio Schoen.
\newblock On convex envelopes for bivariate functions over polytopes.
\newblock {\em Mathematical Programming}, 144(1-2):65--91, 2014.

\bibitem{lovasz1983submodular}
L{\'a}szl{\'o} Lov{\'a}sz.
\newblock Submodular functions and convexity.
\newblock In {\em Mathematical Programming The State of the Art}, pages
  235--257. Springer, 1983.

\bibitem{markowitz1957solution}
Harry~M Markowitz and Alan~S Manne.
\newblock On the solution of discrete programming problems.
\newblock {\em Econometrica: journal of the Econometric Society}, pages
  84--110, 1957.

\bibitem{mccormick1976computability}
Garth~P McCormick.
\newblock Computability of global solutions to factorable nonconvex programs:
  Part i — \text{Convex} underestimating problems.
\newblock {\em Mathematical Programming}, 10(1):147--175, 1976.

\bibitem{meyer2004trilinear}
Clifford~A Meyer and Christodoulos~A Floudas.
\newblock Trilinear monomials with mixed sign domains: Facets of the convex and
  concave envelopes.
\newblock {\em Journal of Global Optimization}, 29(2):125--155, 2004.

\bibitem{misener2012global}
Ruth Misener and Christodoulos~A Floudas.
\newblock Global optimization of mixed-integer quadratically-constrained
  quadratic programs (miqcqp) through piecewise-linear and edge-concave
  relaxations.
\newblock {\em Mathematical Programming}, 136(1):155--182, 2012.

\bibitem{misener2014antigone}
Ruth Misener and Christodoulos~A Floudas.
\newblock \text{ANTIGONE}: algorithms for continuous/integer global
  optimization of nonlinear equations.
\newblock {\em Journal of Global Optimization}, 59(2-3):503--526, 2014.

\bibitem{misener2011apogee}
Ruth Misener, Jeffrey~P Thompson, and Christodoulos~A Floudas.
\newblock \text{APOGEE}: Global optimization of standard, generalized, and
  extended pooling problems via linear and logarithmic partitioning schemes.
\newblock {\em Computers \& Chemical Engineering}, 35(5):876--892, 2011.

\bibitem{nagarajan2019adaptive}
Harsha Nagarajan, Mowen Lu, Site Wang, Russell Bent, and Kaarthik Sundar.
\newblock An adaptive, multivariate partitioning algorithm for global
  optimization of nonconvex programs.
\newblock {\em Journal of Global Optimization}, 74(4):639--675, 2019.

\bibitem{rikun1997convex}
Anatoliy~D Rikun.
\newblock A convex envelope formula for multilinear functions.
\newblock {\em Journal of Global Optimization}, 10(4):425--437, 1997.

\bibitem{savage1997survey}
Carla Savage.
\newblock A survey of combinatorial gray codes.
\newblock {\em SIAM review}, 39(4):605--629, 1997.

\bibitem{sherali1997convex}
Hanif~D Sherali.
\newblock Convex envelopes of multilinear functions over a unit hypercube and
  over special discrete sets.
\newblock {\em Acta mathematica vietnamica}, 22(1):245--270, 1997.

\bibitem{sherali2013reformulation}
Hanif~D Sherali and Warren~P Adams.
\newblock {\em A reformulation-linearization technique for solving discrete and
  continuous nonconvex problems}, volume~31.
\newblock Springer Science \& Business Media, 2013.

\bibitem{sherali1992global}
Hanif~D Sherali and Cihan~H Tuncbilek.
\newblock A global optimization algorithm for polynomial programming problems
  using a reformulation-linearization technique.
\newblock {\em Journal of Global Optimization}, 2(1):101--112, 1992.

\bibitem{tawarmalani2010inclusion}
Mohit Tawarmalani.
\newblock Inclusion certificates and simultaneous convexification of functions.
\newblock working paper, 2010.

\bibitem{tawarmalani2013explicit}
Mohit Tawarmalani, Jean-Philippe~P Richard, and Chuanhui Xiong.
\newblock Explicit convex and concave envelopes through polyhedral
  subdivisions.
\newblock {\em Mathematical Programming}, 138(1-2):531--577, 2013.

\bibitem{tawarmalani2001semidefinite}
Mohit Tawarmalani and Nikolaos~V Sahinidis.
\newblock Semidefinite relaxations of fractional programs via novel
  convexification techniques.
\newblock {\em Journal of Global Optimization}, 20(2):133--154, 2001.

\bibitem{tawarmalani2002convex}
Mohit Tawarmalani and Nikolaos~V Sahinidis.
\newblock Convex extensions and envelopes of lower semi-continuous functions.
\newblock {\em Mathematical Programming}, 93(2):247--263, 2002.

\bibitem{tawarmalani2004global}
Mohit Tawarmalani and Nikolaos~V Sahinidis.
\newblock Global optimization of mixed-integer nonlinear programs: A
  theoretical and computational study.
\newblock {\em Mathematical programming}, 99(3):563--591, 2004.

\bibitem{topkis2011supermodularity}
Donald~M Topkis.
\newblock {\em Supermodularity and complementarity}.
\newblock Princeton university press, 2011.

\bibitem{vielma2015mixed}
Juan~Pablo Vielma.
\newblock Mixed integer linear programming formulation techniques.
\newblock {\em SIAM Review}, 57(1):3--57, 2015.

\bibitem{vielma2010mixed}
Juan~Pablo Vielma, Shabbir Ahmed, and George Nemhauser.
\newblock Mixed-integer models for nonseparable piecewise-linear optimization:
  Unifying framework and extensions.
\newblock {\em Operations research}, 58(2):303--315, 2010.

\bibitem{vielma2011modeling}
Juan~Pablo Vielma and George~L Nemhauser.
\newblock Modeling disjunctive constraints with a logarithmic number of binary
  variables and constraints.
\newblock {\em Mathematical Programming}, 128(1-2):49--72, 2011.

\bibitem{vigerske2018scip}
Stefan Vigerske and Ambros Gleixner.
\newblock \text{SCIP}: Global optimization of mixed-integer nonlinear programs
  in a branch-and-cut framework.
\newblock {\em Optimization Methods and Software}, 33(3):563--593, 2018.

\end{thebibliography}

\end{document}